\documentclass[psamsfonts]{amsart}

\usepackage{amssymb, amsthm, amsmath, amsfonts, mathrsfs}
\usepackage{enumerate}
\usepackage[all,arc]{xy}
\usepackage{hyperref}
\usepackage{graphicx}
\usepackage{overpic}
\usepackage{xcolor}
\usepackage{mathtools}

\newtheorem{thm}{Theorem}[section]
\newtheorem{cor}[thm]{Corollary}
\newtheorem{prop}[thm]{Proposition}
\newtheorem{lem}[thm]{Lemma}
\newtheorem{conj}[thm]{Conjecture}
\newtheorem{quest}[thm]{Question}

\theoremstyle{definition}
\newtheorem{defn}[thm]{Definition}

\newtheorem{con}[thm]{Construction}

\newtheorem{notn}[thm]{Notation}

\theoremstyle{remark}
\newtheorem{rem}[thm]{Remark}
\newtheorem{rems}[thm]{Remarks}

\makeatletter
\let\c@equation\c@thm
\makeatother
\numberwithin{equation}{section}

\definecolor{ao}{hsb}{0.67,1,1}
\definecolor{or}{hsb}{0.067,1,1}
\definecolor{green}{hsb}{0.33,1,0.5}

\title{(1,1) non-L-space knots are persistently foliar}

\author{Qingfeng Lyu}
\address{Department of Mathematics\\
  Boston College\\
  Chestnut Hill, MA 02467}
\email{lyuqi@bc.edu}
\subjclass[2020]{57K30, 57R30, 57K10}
\keywords{L-space conjecture, taut foliations, Floer homology, (1,1) knots}
\date{\today}

\begin{document}

\begin{abstract}
  We prove that (1,1) non-L-space knots in $S^3$ and lens spaces are persistently foliar. This verifies the L-space knot conjecture of Delman-Roberts for (1,1) knots, thus providing positive evidence for the L-space conjecture.
\end{abstract}

\maketitle

\section{Introduction}
\label{sec:1}

\subsection{The L-space knot conjecture}
\label{subsec:1.1}

It is a long-standing problem in 3-manifold topology to determine which closed 3-manifolds admit taut foliations. Novikov showed that such a manifold must be $S^2\times S^1$ or irreducible~\cite{novikov1965topology}. Gabai constructed taut foliations in all 3-manifolds with positive first Betti number~\cite{gabai1983foliations}. Collective work of~\cite{ozsvath2004holomorphic,bowden2016approximating,kazez2017c0} shows that if a closed 3-manifold $M$ admits a co-oriented taut foliation, then it cannot be an L-space. The L-space conjecture~\cite{boyer2013spaces,juhasz2015survey} predicts the converse, that is, for an irreducible rational homology sphere, if it is not an L-space, then it admits a co-oriented taut foliation. 

A principal way of constructing such rational homology spheres is by doing Dehn surgeries on knots. A knot in an L-space is called an L-space knot if it admits another nontrivial L-space surgery. To construct taut foliations in the complement of an L-space knot that extend to different (but certainly not all) Dehn fillings, the most successful way so far is to construct a properly embedded laminar branched surface~\cite{li2002laminar,li2003boundary}, and analyze slopes carried by its boundary train track. For examples, see~\cite{roberts2000taut,roberts2001taut,krishna2020taut,zhao2023co,krishna2025taut}.

On the other hand, for a non-L-space knot in an L-space, we would expect taut foliations in most of the Dehn surgeries. Instead of working with boundary train tracks, there is a more effective way of constructing taut foliations, which allows us to find taut foliations in \textit{all but one} Dehn surgeries. This construction is classical, and is often referred to as the ``even meridional cusps'' construction. See~\cite{gabai1992taut}, Operation 2.4.1;~\cite{calegari2007foliations}, Examples 4.18-4.22; or~\cite{delman2020taut}, Proposition 4.10 for more details. It is by far the most common construction of taut foliations in non-L-space knot complements. For examples, see~\cite{delmanrobertsAM,dunfield2020floer,delman2020taut,delman2021persistently,santoro2024taut}. In particular, Dunfield conjectured to some extent the universality of this construction based on his experiments (\cite{dunfield2020floer}, Conjecture 8.4).

Due to its common appearance, Delman and Roberts extracted the following definition and made the following ``knot version'' of the L-space conjecture:

\begin{defn}[\cite{delman2020taut}, Definition 1.7]
  A knot in a 3-manifold is called \textit{persistently foliar}, if except for one meridional slope, all boundary slopes of the knot complement are strongly realized by co-oriented taut foliations (i.e. to each boundary slope there exists a co-oriented taut foliation of the knot complement, which intersects the boundary torus transversely in a \textit{linear} foliation of that slope).
  \label{defn:pfoliar}
\end{defn}

\begin{conj}[\cite{delman2020taut}, Conjecture 1.9, the ``L-space knot conjecture'']
  A knot in an L-space is persistently foliar if and only if it has no non-trivial L-space or reducible surgeries.
  \label{conj:knotL}
\end{conj}

Delman and Roberts verified this conjecture for alternating knots and Montesinos knots~\cite{delmanrobertsAM}. The only L-space alternating knots are the torus knots $T(2,2k+1)$, and the only L-space Montesinos knots are the torus knots $T(2,2k+1)$ and the pretzel knots $P(-2,3,2k+1)$~\cite{baker2018montesinos}. In~\cite{delmanrobertsAM} it is actually proved that all other alternating and Montesinos knots are persistently foliar.

\subsection{(1,1) knots}
\label{subsec:1.2}

In this paper we consider (1,1) knots. These are knots that could be placed in 1-bridge positions with respect to genus 1 Heegaard splittings, or equivalently, knots that could be represented by doubly pointed genus one Heegaard diagrams (see subsection~\ref{subsec:2.1} for more details). Classical examples of (1,1) knots include the torus knots and the 2-bridge knots. They gain popularity as the Floer homology theories flourish, since they are relatively abundant, while their knot Floer homologies remain relatively easy to compute.

A Floer-theoretic motivation of the L-space conjecture is to find topological characterizations of L-spaces and L-space knots. For (1,1) knots such characterizations have already been done by Greene-Lewallen-Vafaee using the (1,1) Heegaard diagrams~\cite{greene2018space}. Their result also indicates that the family of (1,1) L-space knots is highly nontrivial (for example, this includes all 1-bridge braids). 

Using the description in~\cite{greene2018space}, we show that:

\begin{thm}
  (1,1) non-L-space knots in $S^3$ and lens spaces are persistently foliar.
\label{thm:main}
\end{thm}

Since the L-space knot conjecture is only stated for knots in L-spaces, it follows immediately that

\begin{cor}
  The L-space knot conjecture holds for (1,1) knots.
\end{cor}

Since 3-manifolds admitting taut foliations are irreducible, we also obtain the following corollary, which aligns with the cabling conjecture~\cite{acuna1986knot}, and its generalization to knots in lens spaces~\cite{baker2014cabling}:

\begin{cor}
  (1,1) non-L-space knots in $S^3$ and lens spaces do not admit reducible surgeries.
  \label{cor:reducible}
\end{cor}

We remark that for (1,1) knots in $S^3$ this is already known. In fact, (1,1) knots in $S^3$ are strongly invertible~\cite{morimoto1996identifying}, thus satisfy the cabling conjecture~\cite{eudave1992band}. On the other hand, by~\cite{taylor2025genus} (or~\cite{morimoto1991unknotting}\cite{bleiler1994two} for the $S^3$ case), (1,1) cabled knots in $S^3$ and lens spaces must be cables of the unknot, torus knots, or their lens-space-analogue ``core knots''. Such (1,1) knots are L-space knots.

\subsection{Heegaard foliations}
\label{subsec:1.3}

We use branched surfaces obtained from the (1,1) Heegaard diagrams to construct taut foliations. In 2020, Sarah Rasmussen first proposed a way to construct taut foliations from a Heegaard splitting of a certain 3-manifold~\cite{rasmussen2020taut}, where she called them ``Heegaard foliations''. This idea was later developed by Tao Li using branched surfaces in~\cite{li2022taut}, where he generalized Sarah's result, proving that for a Heegaard genus two 3-manifold $M$, if $\pi_1(M)$ is left orderable, then $M$ admits a co-orientable taut foliation. Our construction of branched surfaces can be seen as a modification of Li's construction. The existing results imply that the construction of taut foliations from Heegaard splittings is sharp, at least for manifolds with low Heegaard genus.

The topological description of (1,1) non-L-space knots in~\cite{greene2018space} guarantees the existence of an interesting \textit{source} sector on the branched surface we constructed. We can then ``reverse'' the co-orientation of this sector to get a new branched surface carrying interesting laminations. (We remark that this ``reversing'' operation also plays a central role in the ``double diamond replacements'' of Delman-Roberts in~\cite{delman2020taut}.) More details of the construction will be given in section~\ref{sec:2-2}. This provides an example where we can construct taut foliations from the \textit{Heegaard} diagram features of non-L-space knots, which we hope could be generalized later to reveal deeper connections between \textit{Heegaard} Floer homology and taut foliations.

\begin{quest}
  Are there similar ``source sector'' conditions on Heegaard diagrams in general which guarantee that a knot is non-L-space or persistently foliar?
\end{quest}

We also made a slight generalization to the classical ``even meridional cusps'' construction (see Lemma~\ref{lem:cusp}). In order for our branched surface from Heegaard splittings to live in the knot complement, we allow our branched surface to have circle boundary components. What we essentially observe is that, as long as there are enough meridional cusps, we can use the extra cusps to care the meridional circle boundaries. More details are discussed in subsection~\ref{subsec:2.4}. This in general allows us to construct branched surfaces (that carry interesting laminations) from meridional surfaces of a knot.

\begin{quest}
  Can we construct interesting foliations from other meridional surfaces (e.g. (essential) Conway spheres) using Lemma~\ref{lem:cusp}?
\end{quest}

The technically complicated part of this paper is to prove our branched surfaces fully carry laminations (we refer to it as the ``lamination problem''). Here our methodology is based on the laminar branched surface theory in~\cite{li2002laminar}, and we will heavily use the techniques developed by the author in~\cite{lyu2024knot}, as well as the results in~\cite{lyu2024knot} themselves. Interestingly, the main result there is about the slope detection theory developed by Boyer-Clay and Boyer-Gordon-Hu in~\cite{boyer2017foliations,boyer2021slope}, but the main result of this paper is essentially not. Philosophically, this provides an example where slope detections actually helped with finding taut foliations even in closed, hyperbolic 3-manifolds.

A key technique we used to solve the lamination problem both in~\cite{lyu2024knot} and here is a blow up of branched surfaces. In general, we want to use the laminar branched surface theory in~\cite{li2002laminar}, especially the sink disk free condition, to show that a branched surface fully carries a lamination. By blowing up the branched surface along a sub-branched surface, we can simplify the intersections of the branch locus, thus better control the sink disks. This allows us to make an inductive construction of laminations, starting from simpler branched surfaces. Details are discussed in subsection~\ref{subsec:5.4}, where we do a case study in the context of (1,1) knots, and summarize the techniques in more generality in Definition~\ref{defn:blowup} and Proposition~\ref{prop:blowup}.

\vspace{6pt}

\textbf{Organisation of this paper.} In section~\ref{sec:2-1} we review the related theories of (1,1) knots and branched surfaces. In section~\ref{sec:2-2} we define our branched surfaces for any (1,1) non-L-space knot in $S^3$ or a lens space, and prove that if any of these branched surfaces fully carries a lamination, then the corresponding knot is persistently foliar. In section~\ref{sec:3} we discuss reduction operations on (1,1) diagrams which simplifies the lamination problem. In sections~\ref{sec:4} and~\ref{sec:5} we deal with the two terminating cases of our reductions respectively. The proof of Theorem~\ref{thm:main} is given at the end of section~\ref{sec:3}, modulo the technical details.

\vspace{6pt}

\textbf{Acknowledgements} We thank Cameron Gordon for kindly pointing out the connection to strongly invertible knots, and John Baldwin and Josh Greene for helpful conversations. Special thanks to Zipei Nie for his interest and stimulating conversations which initiated this project. The author is deeply in debt to his advisor Tao Li for his patience, encouragement, and many valuable comments. 

\section{Preparations}
\label{sec:2-1}

In this section we briefly review the related theories of (1,1) knots and branched surfaces. This is largely a rewritten version of~\cite{lyu2024knot}, section 2.

\subsection{(1,1) knots and (1,1) diagrams}
\label{subsec:2.1}

A (1,1) knot is a knot that admits a (1,1) diagram, or doubly pointed genus one Heegaard diagram $(\Sigma,\alpha,\beta,z,w)$, where $\Sigma$ denotes the Heegaard torus, $\alpha,\beta$ the curves of the Heegaard diagram, and $z,w$ the basepoints. The knot is recovered from the diagram by taking properly embedded, boundary parallel arcs in each Heegaard solid torus connecting the basepoints $z,w$, while avoiding the compression disks bounded by $\alpha,\beta$ respectively. We assume the torus $\Sigma$ and the two curves $\alpha,\beta$ are oriented. Moreover, \textbf{we assume the ambient space to be a rational homology sphere}, i.e. $S^3$ or some lens space.

\begin{figure}[!hbt]
  \begin{overpic}[scale=0.9]{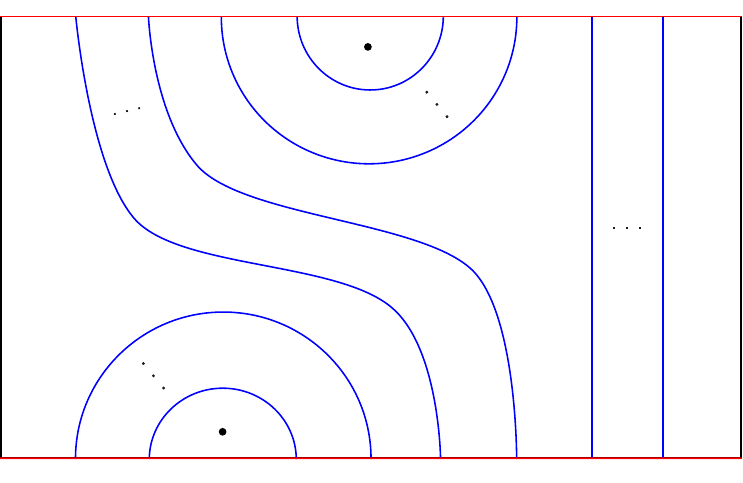}
      \put(27,7){$z$}
      \put(46,59){$w$}
      \put(96,4.5){$\color{red}\alpha$}
      \put(90.5,55){$\color{ao}\beta$}
      \put(9.5,0){$1\;\;\;2\;\;\;...$}
      \put(89,0.5){$p$}
      \put(5.5,64){$(s+1)\;\;\;...$}
      \put(18,12){$q$}
      \put(18,47){$r$}
      \put(56,49){$q$}
      \put(63,40){$(p-2q-r)$}
  \end{overpic}
  \caption{Reduced (1,1)-diagram parametrized by $(p,q,r,s)$}
  \label{fig:reduced11}
\end{figure}

A (1,1) diagram is called reduced if each bigon contains a basepoint. According to~\cite{rasmussen2005knot}, a reduced (1,1) diagram can be put in a standard form and parametrized by a 4-tuple $(p,q,r,s)$ (see Figure~\ref{fig:reduced11}, where we say $\alpha$ is put in \textbf{standard position}). Now $\alpha$ cuts $\beta$ into several arcs. The $\beta$-arcs connecting the same side of $\alpha$ are called \textbf{rainbow arcs}, and the $\beta$-arcs connecting different sides of $\alpha$ are called \textbf{vertical arcs}. The reduced (1,1) diagram is called \textbf{non-simple} if it contains rainbow arcs.

\vspace{3pt}

\textbf{Hyperelliptic involution and symmetry for (1,1) diagrams}

An important feature for (1,1) diagrams is the symmetry that comes from the hyperelliptic involution.

\begin{figure}[!hbt]
  \begin{overpic}[scale=0.6]{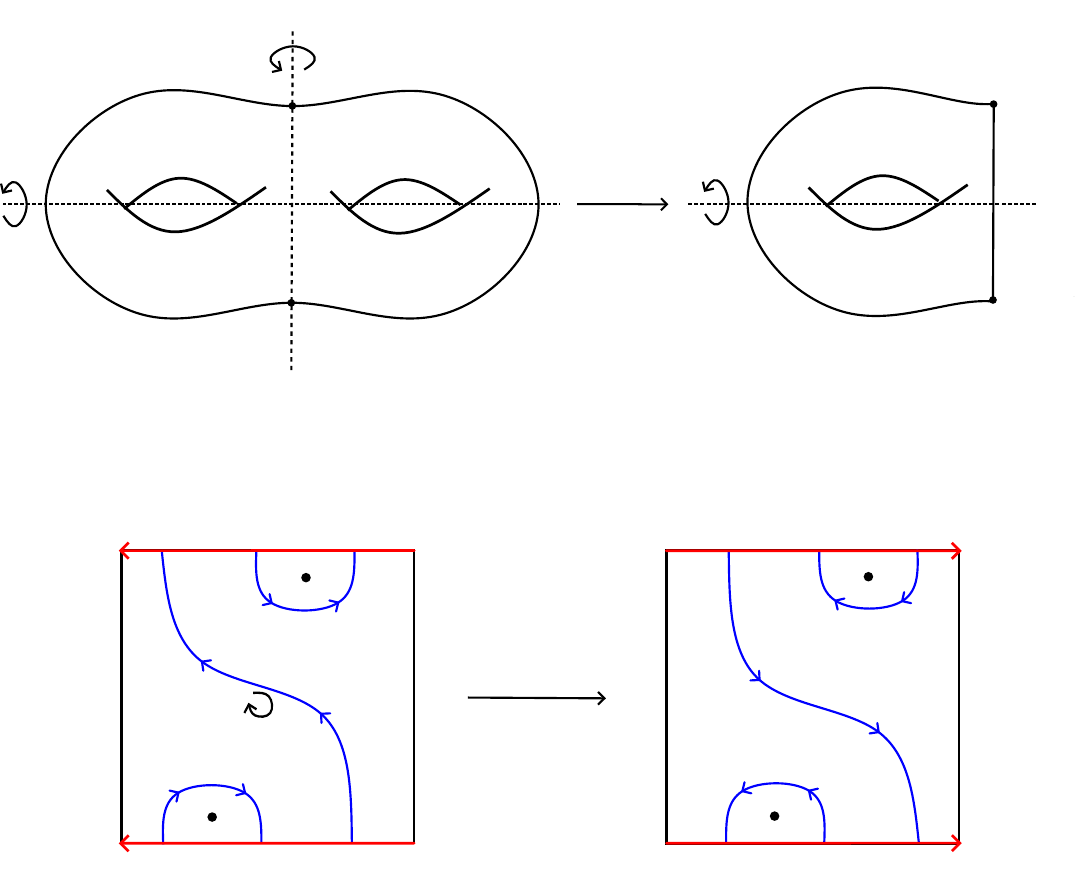}
      \put(50,42){$(a)$}
      \put(50,0){$(b)$}
      \put(25,54){$z$}
      \put(92,53){$z$}
      \put(24.5,70){$w$}
      \put(92,72){$w$}
      \put(2.3,59){$\tilde{h}$}
      \put(67,59){$h$}
      \put(29,77){$\tau$}
      \put(24,13){$h$}
      \put(16,28){$\color{ao}\beta$}
      \put(68,28){$\color{ao} h(\beta)$}
      \put(36,1.5){$\color{red}\alpha$}
      \put(86,0.5){$\color{red}h(\alpha)$}
      \put(17,5){$z$}
      \put(25,28){$w$}
      \put(72,5){$z$}
      \put(81,28){$w$}
  \end{overpic}
  \caption{Hyperelliptic involution}
  \label{fig:hypinvo}
\end{figure}

We can regard the torus $\Sigma$ with 2 basepoints $(z,w)$ as a quotient orbifold of a genus-2 surface by a $(\mathbb{Z}/2\mathbb{Z})$-action generated by $\tau$, as shown in Figure~\ref{fig:hypinvo}.$(a)$. Then the hyperelliptic involution of the genus-2 surface $\tilde{h}$ descends to $h$ a $\pi$-rotation of the square representing $\Sigma$, exchanging the two basepoints, see Figure~\ref{fig:hypinvo}.$(a)(b)$. On the other hand, we know the hyperelliptic involution $\tilde{h}$ fixes all isotopy classes of simple closed curves, reversing orientations of the non-separating ones while preserving those of the separating ones~\cite{haas1989geometry}. Since $\alpha$ and $\beta$ are essential in $\Sigma$, their lifts are non-separating, and hence their orientations are reversed by $h$. Since the diagram obtained by doing $\pi$-rotation is still reduced, we know $(\Sigma,h(\alpha),h(\beta),z,w)$ is actually the same (1,1) diagram with all orientations reversed. See Figure~\ref{fig:hypinvo}.$(b)$. This gives us an important symmetry of the (1,1) diagrams.

\vspace{3pt}

\textbf{Sink, source, and parallel arcs and sectors}

We note that not all 4-tuples give a (1,1) diagram, since here $\beta$ needs to be an essential simple closed curve. We remind the readers of the following definitions and lemmas in~\cite{lyu2024knot} that further characterize (1,1) diagrams:

We consider the regions of $\Sigma$ cut off by $\alpha$ and $\beta$. Suppose $\alpha$ is placed in standard position, then for any non-simple (1,1) diagram (as depicted in Figure~\ref{fig:reduced11}), there are two bigons each containing one basepoint, two hexagons or one octagon bounded by both rainbow arcs and vertical arcs, and many quadrilaterals bounded by either rainbow arcs or vertical arcs. From now on unless otherwise specified, when we say bigons in a reduced (1,1) diagram, we refer to the regions cut out by the $\alpha$ and $\beta$ curves, i.e. only the \textit{innermost} bigons.

\begin{figure}[!hbt]
  \begin{overpic}[scale=0.6]{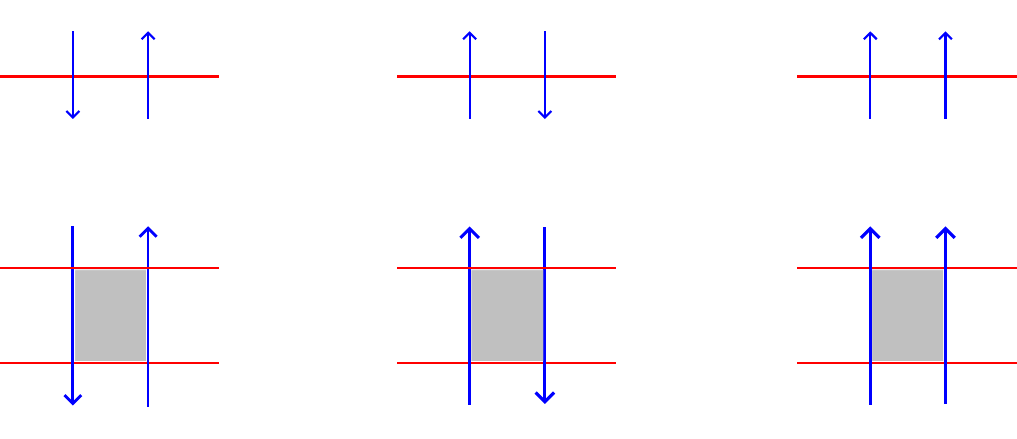}
      \put(22,34){$\color{red}\alpha$}
      \put(16,40){$\color{ao}\beta$}
      \put(4,26){$a.\;$sink arc}
      \put(42,26){$b.\;$source arc}
      \put(80,26){$c.\;$parallel arc}
      \put(2.5,0){$d.\;$sink sector}
      \put(40,0){$e.\;$source sector}
      \put(78.4,0){$f.\;$parallel sector}
  \end{overpic}
  \caption{Arcs and sectors}
  \label{fig:sink_source_parallel}
\end{figure}

\begin{defn}[\cite{lyu2024knot}, Definition 2.3]
  Fixing orientations of $(\Sigma,\alpha,\beta)$, for each $\alpha$-arc cut out by $\beta$, we say it is ($\beta$-)\textbf{sink} if it always lies to the left of the $\beta$ curve at its endpoints, ($\beta$-) \textbf{source} if it always lies to the right of the $\beta$ curve at endpoints, and ($\beta$-)\textbf{parallel} otherwise. See Figure~\ref{fig:sink_source_parallel} $a\sim c$. It follows immediately that for a quadrilateral region its two boundary $\alpha$-arcs must be of the same type. We call it a ($\beta$-)\textbf{sink sector} if its boundary $\alpha$-arcs are sink, a ($\beta$-)\textbf{source sector} if its boundary $\alpha$-arcs are source, and a ($\beta$-)\textbf{parallel sector} otherwise. See Figure~\ref{fig:sink_source_parallel} $d\sim f$. In addition, we call a bigon region a ($\beta$-)\textbf{sink bigon} if its boundary $\alpha$-arc is sink, and a ($\beta$-)\textbf{source bigon} if its boundary $\alpha$-arc is source.

  By interchanging $\alpha$ and $\beta$ in the above definition, we can also define $\alpha$-sink (source, parallel) arcs, $\alpha$-sink (source, parallel) sectors, and $\alpha$-sink (source) bigons.
  \label{def:sink_source_parallel}
\end{defn}

\vspace{3pt}

\textbf{Sink and source tubes}

A crucial observation here is that along a boundary $\alpha$-arc of a $\beta$-sink sector, it can only connect to another $\beta$-sink sector or some non-quadrilateral region. It follows that the $\beta$-sink sectors are connected along their boundary $\alpha$-arcs to form some tubes. The following Lemma is essential to defining \textit{the} tube:

\begin{lem}[\cite{lyu2024knot}, Remark 2.4(2) and Proposition 2.5]
  In a non-simple, reduced (1,1) diagram, there is exactly one sink bigon and one source bigon. Moreover, there is a unique sink $\alpha$-arc and a unique source $\alpha$-arc among the boundary $\alpha$-arcs of the hexagons or octagon.
  \label{lem:pre_sink_tube}
\end{lem}

\begin{proof}[Proof sketch]
  Use the hyperelliptic involution symmetry to analyze the boundary $\alpha$-arcs of the bigons and hexagons or octagon.
\end{proof}

\begin{defn}[\cite{lyu2024knot}, Definition 2.6]
  In a non-simple, reduced (1,1) diagram, the ($\beta$-)\textbf{sink} \textbf{tube} is the union of ($\beta$-)sink sectors that connect the ($\beta$-)sink bigon to the sink boundary $\alpha$-arc of the hexagons or octagon, see Figure~\ref{fig:sk_tube}. The ($\beta$-)\textbf{source tube} is the union of ($\beta$-)source sectors that connect the ($\beta$-)source bigon to the source boundary $\alpha$-arc of the hexagons or octagon.
  \label{def:sink_tube}
\end{defn}

\begin{figure}[!bht]
  \begin{overpic}[scale=0.7]{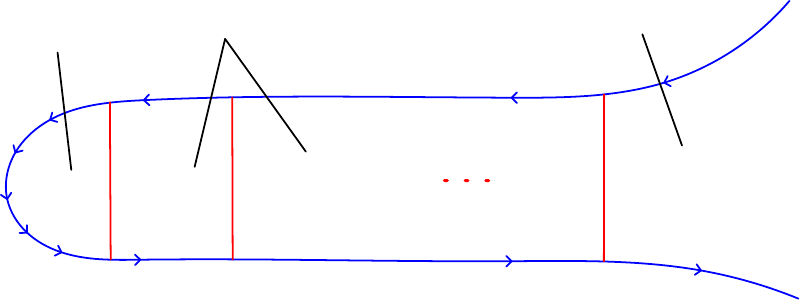}
      \put(-5,33){sink bigon}
      \put(20,34){sink sectors}
      \put(55,35){hexagon or octagon}
  \end{overpic}
  \caption{($\beta$-)sink tube}
  \label{fig:sk_tube}
\end{figure}

\begin{lem}[\cite{lyu2024knot}, Proposition 2.7]
  The ($\beta$-)sink (resp. source) tube contains all ($\beta$-)sink (resp. source) sectors.
  \label{lem:char_sink_tube}
\end{lem}

\vspace{3pt}

\textbf{Coherent and incoherent diagrams}

An important landmark in the theory of (1,1) knots is established by~\cite{greene2018space}, where the authors characterized (1,1) L-space knots in terms of their (1,1) diagrams. For a standard reduced (1,1) diagram as depicted in Figure~\ref{fig:reduced11}, two rainbow arcs around the same basepoint are called \textbf{inconsistent}, if they are in different directions. A reduced (1,1) diagram is called \textbf{coherent}, if there are no inconsistent arcs. The following observation is made in~\cite{greene2018space}:

\begin{thm}[\cite{greene2018space}, Theorem 1.2]
  A reduced (1,1) diagram represents an L-space knot if and only if it is coherent.
  \label{thm:11L}
\end{thm}

\subsection{Branched surfaces}
\label{subsec:2.2}

A (co-oriented) \textbf{branched surface} $\mathcal{B}$ is a compact space locally modelled on Figure~\ref{fig:brsf}.$(a)$. The subset of points that have no neighborhood isomorphic to $\mathbb{R}^2$ is called the \textbf{branch locus} $L(\mathcal{B})$. According to the local picture we can think of the branch locus as a collection of transversely-intersecting \textit{immersed} curves in $B$. Points where the immersed curves intersect are called \textbf{double points}. Away from the double points, we can define the \textbf{branch direction} on branch locus up to homotopy, such that the direction is tangent to the branched surface, transverse to the branch locus, and always points to the side with fewer components, see Figure~\ref{fig:brsf}.$(a)$. Connected components of $\mathcal{B}-L(\mathcal{B})$ are called \textbf{branch sectors} of the branched surface.

\begin{figure}[!hbt]
    \begin{overpic}[scale=0.6]{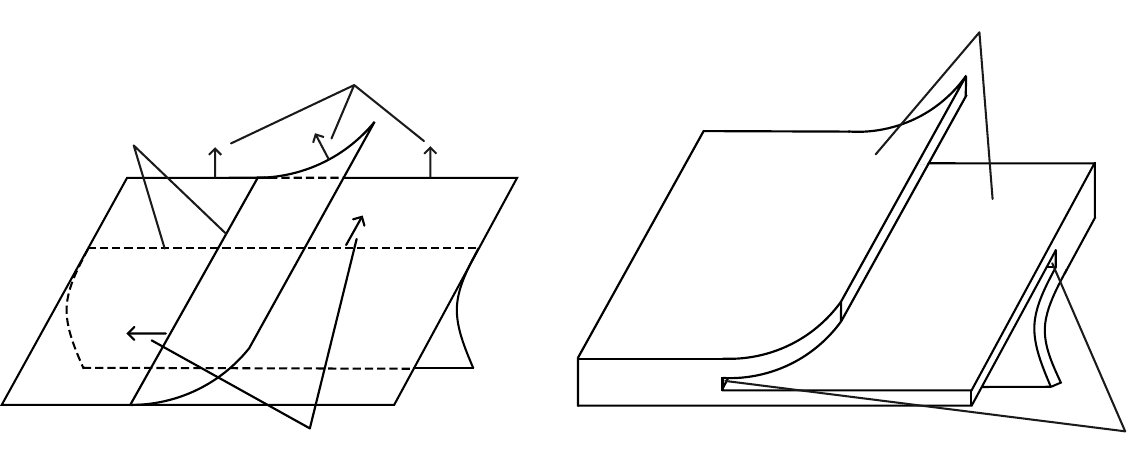}
        \put(0,29){branch locus}
        \put(20,35){co-orientations}
        \put(28,1){branch direction}
        \put(85,40){$\partial_h$}
        \put(99,0){$\partial_v$}
        \put(20,0){$(a)$}
        \put(70,0){$(b)$}
    \end{overpic}
    \caption{(Co-oriented) branched surface and its regular neighborhood}
    \label{fig:brsf}
\end{figure}

For a (co-oriented) branched surface $\mathcal{B}$, we can define its regular neighborhood $N(\mathcal{B})$ to be an $I$-bundle over the branched surface. This can be constructed by taking a trivial $I$-bundle for each sector and then pasting them along the branch locus so that $N(\mathcal{B})$ is locally modelled on Figure~\ref{fig:brsf}.$(b)$. We can then define the \textbf{horizontal boundary} $\partial_h N(\mathcal{B})$ to be the boundary of $N(\mathcal{B})$ consisting of the boundaries of the fibers, and the \textbf{vertical boundary} $\partial_v N(\mathcal{B})$ to be the remaining boundary of $N(\mathcal{B})$, which consists of interior segments of $I$-fibers at the branch locus and possibly also fibers transverse to the boundary of the branched surface. For a branched surface with circle boundary, the horizontal boundary of its regular neighborhood is a disjoint union of compact oriented surfaces, and the vertical boundary a disjoint union of annuli. As we regard $N(\mathcal{B})$ as an $I$-bundle over $\mathcal{B}$, there is a bundle projection $\pi:N(\mathcal{B})\rightarrow \mathcal{B}$ collapsing the $I$-fibers. Notice that $\pi(\partial_v N(\mathcal{B}))=L(\mathcal{B})\cup \partial \mathcal{B}$.

\vspace{3pt}

\textbf{Laminar branched surfaces}

Nowadays we use laminar branched surfaces to construct (essential) laminations and (taut) foliations in 3-manifolds. We refer the readers to~\cite{li2002laminar} for details of the laminar branched surface theory. In particular, laminar branched surfaces are refinements of the essential branched surfaces in~\cite{gabai1989essential}, and are defined as follows:

\begin{defn}
  A branched surface $\mathcal{B}$ in a closed, oriented 3-manifold $M$ is called \textit{laminar} if 
  \begin{enumerate}[(i)]
      \item $\partial_h N(\mathcal{B})$ is incompressible in $M-int(N(\mathcal{B}))$, no component of $\partial_h N(\mathcal{B})$ is a sphere, and $M-int(N(\mathcal{B}))$ is irreducible (where $int(X)$ is the interior of $X$),
      \item there is no monogon in $M-int(N(\mathcal{B}))$,
      \item there is no Reeb component (i.e. $\mathcal{B}$ does not carry any torus that bounds a solid torus in $M$), and
      \item there is no sink disk (a \textit{sink disk} is a disk branch sector where the branch directions at boundary always point inwards, see Figure~\ref{fig:skdsk}).
  \end{enumerate}
  \label{def:laminar}
\end{defn}

\begin{figure}[!hbt]
  \begin{overpic}[scale=0.5]{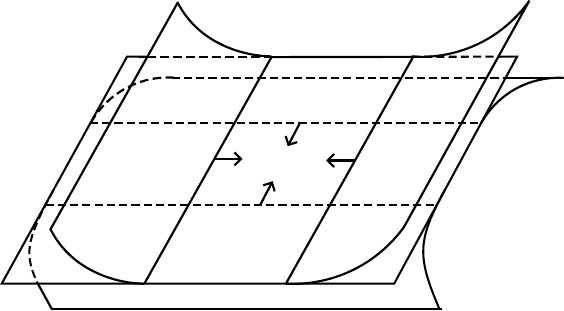}
      
  \end{overpic}
  \caption{Sink disk}
  \label{fig:skdsk}
\end{figure}

\begin{thm}[\cite{li2002laminar}, Theorem 1]
  Laminar branched surfaces in closed, oriented 3-manifolds fully carry (essential) laminations.
  \label{thm:li}
\end{thm} 

In this paper we will use the following lemma, which is directly derived from Theorem~\ref{thm:li}:

\begin{lem}[\cite{lyu2024knot}, Lemma 2.11]
  Let $\mathcal{B}$ be a co-oriented branched surface with boundary a union of circles. If $\mathcal{B}$ is sink disk free, not carrying any torus, and no component of $\partial_h N(\mathcal{B})$ is a disk or a sphere, then $\mathcal{B}$ fully carries a lamination.
  \label{lem:sk_disk_free}
\end{lem}

\vspace{3pt}

\textbf{Diamonds, and reversing sectors}

We further highlight the following notation and construction, which will be used in the next subsection.

\begin{notn}[Diamonds]
  To describe a branched surface where sectors are attached to a planar diagram, we use the diamond notations in Figure~\ref{fig:delmond} for branch directions that point out of the planar diagram. The notation originated from~\cite{wu2012persistently}, and is developed in~\cite{delman2020taut}. In particular Figure~\ref{fig:delmond} is from~\cite{delman2020taut}.
\end{notn}

\begin{figure}[!hbt]
  \begin{overpic}[scale=0.3]{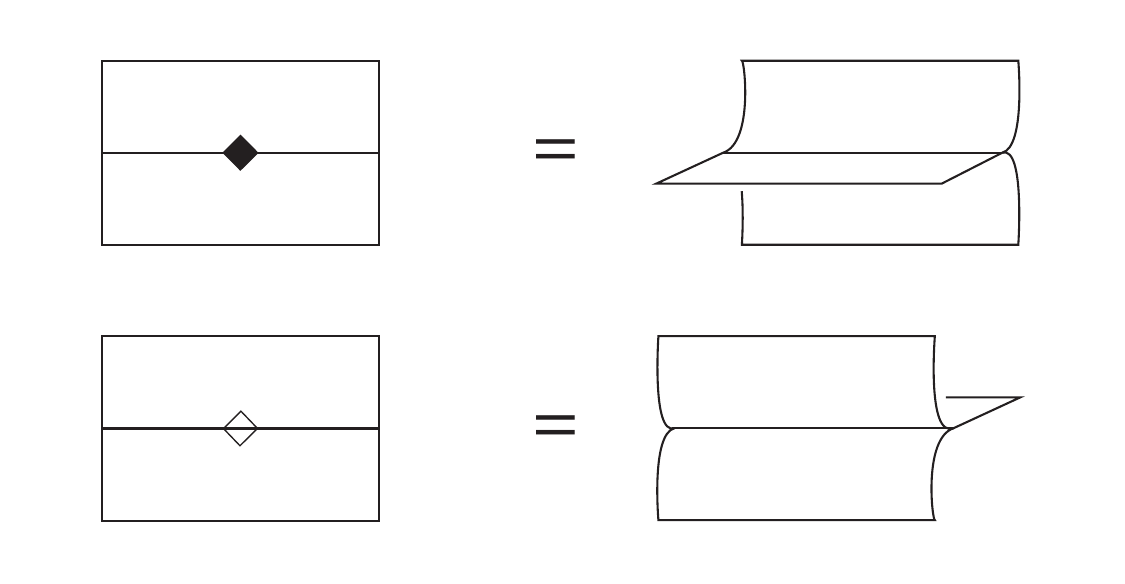}
    
  \end{overpic}
  \caption{Diamond notations}
  \label{fig:delmond}
\end{figure}

\begin{con}[Reversing co-orientations]
  Let $\mathcal{B}$ be a co-oriented branched surface. Let $S$ be a branch sector of $\mathcal{B}$, where along $\partial S$ the branch directions always point out of $S$. We can then reverse the co-orientation of $S$ and get a new branched surface $\mathcal{B'}$, see Figure~\ref{fig:revsector}, where colored lines are the branch locus (as immersed curves) and arrows indicate branch directions.
\end{con}

\begin{figure}[!hbt]
  \begin{overpic}[scale=0.6]{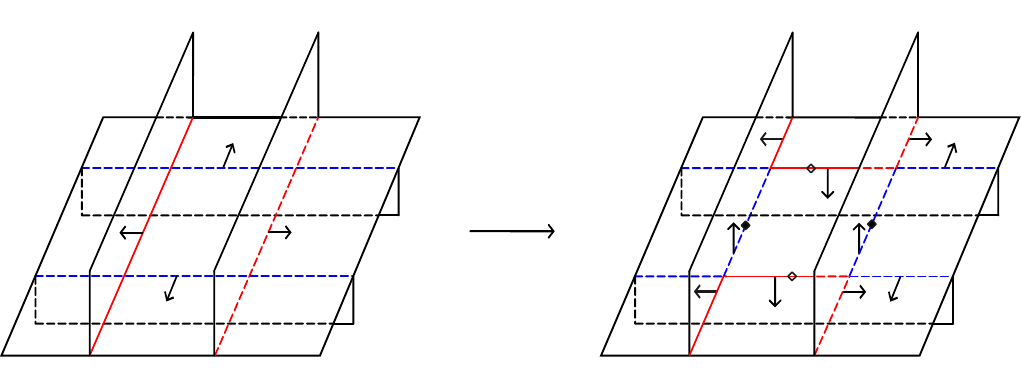}
    \put(35,23){\small $+$}
    \put(95,23){\small $+$}
    \put(90,13){\small $+$}
    \put(65,13){\small $+$}
    \put(72,8){\small $+$}
    \put(80,23){\small $+$}
    \put(75,13){\small $-$}
    \put(17,13){\small $+$}
  \end{overpic}
  \caption{Reversing the co-orientation of a source sector}
  \label{fig:revsector}
\end{figure}

\begin{rems}~\
  \begin{enumerate}
  \item As can be observed in Figure~\ref{fig:revsector}, the branch locus $L(\mathcal{B})$ changes as a collection of immersed curves.
  \item In general, we can reverse the co-orientation of a sub-branched surface of $\mathcal{B}$, as long as along the boundary the branch direction always points out. We will use a slightly more general case below, where one should be convinced of the construction by checking the co-orientations of the sectors.
  \end{enumerate}
\end{rems}

\vspace{3pt}

\textbf{Splittings}

We also remind the readers of the splitting operations that will be used in the next few sections. A \textbf{splitting} of a branched surface $\mathcal{B}$ along an oriented compact surface $S\subset N(\mathcal{B})$ (where $S$ is transverse to the $I$-fibers) results in a branched surface $\mathcal{B}^S$, such that there is an inclusion of $I$-bundle $N(\mathcal{B}^S)\subset N(\mathcal{B})$ respecting the $I$-fibers, where $N(\mathcal{B})\backslash N(\mathcal{B}^S)=N(S)\cong S\times I$. Notice that if $\partial S\cap \partial_v N(\mathcal{B})=\varnothing$, then the splitting is just creating a bubble, so we generally hope this does not happen. Typically, $S$ is a disk and $\partial S$ intersects $\partial_v N(\mathcal{B})$ in arcs, and the splitting is locally modelled on one of the pictures in Figure~\ref{fig:splittings}.

\begin{figure}[!hbt]
  \begin{overpic}[scale=0.8]{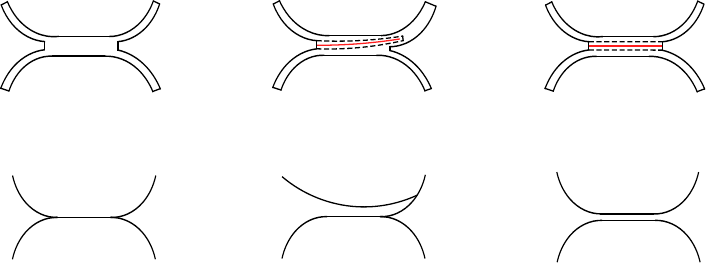}
      
  \end{overpic}
  \caption{Splittings}
  \label{fig:splittings}
\end{figure}

We further remark that splittings along disks locally modelled on the middle picture of Figure~\ref{fig:splittings} will not change the topology of $N(\mathcal{B})$, nor the topology of its horizontal and vertical boundaries. On the other hand, splittings locally modelled on the right picture of Figure~\ref{fig:splittings} will change the topology of $N(\mathcal{B})$.

\section{The branched surface models}
\label{sec:2-2}

\subsection{The modified Heegaard branched surfaces}
\label{subsec:2.3}

In this subsection we define the ``modified Heegaard branched surfaces'' for a reduced, incoherent (1,1) diagram. The name comes from Sarah Rasmussen's ``Heegaard foliations''.

Let $(\Sigma,\alpha,\beta,z,w)$ be a reduced, incoherent (1,1) diagram. We pick orientations of the curves $\alpha,\beta$ so that the $\alpha$-source bigon is $\beta$-sink (in other words, the boundary of an innermost bigon is neither clockwise nor counter-clockwise), and put $\alpha$ in standard position. Suppose the basepoint $w$ is contained in the $\alpha$-source bigon.

Similar to~\cite{lyu2024knot}, Definition 2.13, we can construct a branched surface from the (1,1) diagram. We first pick the torus $\Sigma$. Then we attach a disk along the $\beta$ curve to the positive side of $\Sigma$, and a disk along the $\alpha$ curve to the negative side of $\Sigma$. We further smooth these disks to get a branched surface, such that the branch direction always points to the left of the oriented curves $\alpha,\beta$ (when observed on $\Sigma$ from the positive side). We can then remove a small open disk in the interior of each of the two bigons containing the basepoints. We call the resulting branched surface $\mathcal{B}$ the \textbf{Heegaard branched surface of the (1,1) diagram}.

We can count rainbow arcs starting from $w$ and outwards, until we find the first arc inconsistent with previous arcs, which we denote as $\beta_0$. This $\beta_0$, together with an $\alpha$-arc $\alpha_0$, bounds a (generalized) bigon $S$ in the reduced (1,1) diagram (see Figure~\ref{fig:reversing}). We further label the $\beta$-arcs in the interior of $S$ as $\beta_1,...,\beta_n$ in order\footnote{In this paper we usually think of arcs as closed, i.e., including the endpoints. So precisely speaking the $\beta$-arcs here are actually \textit{properly embedded} in $S$. We nevertheless sometimes say they are ``in the interior'', in contrast with the arc $\beta_0$ ``on the boundary''.}, with $\beta_1$ being the innermost arc, i.e., the boundary $\beta$-arc of the $\alpha$-source bigon. Let $S_0$ be the quadrilateral sector whose boundary $\beta$-arcs are $\beta_0$ and $\beta_n$. Then $S_0$ is both $\alpha$-source and $\beta$-source. It is then a source sector of the Heegaard branched surface.

\begin{defn}[Type I modified Heegaard branched surface]
  For any quadrilateral source sector $T$ on $\Sigma$ in the Heegaard branched surface, we can reverse its co-orientation and obtain a new branched surface $\mathcal{B}_T$. We call $\mathcal{B}_T$ the \textit{type I modified Heegaard branched surface associated to $T$.}
  \label{defn:typeIBS}
\end{defn}

\begin{figure}[!htb]
  \begin{overpic}[scale=0.5]{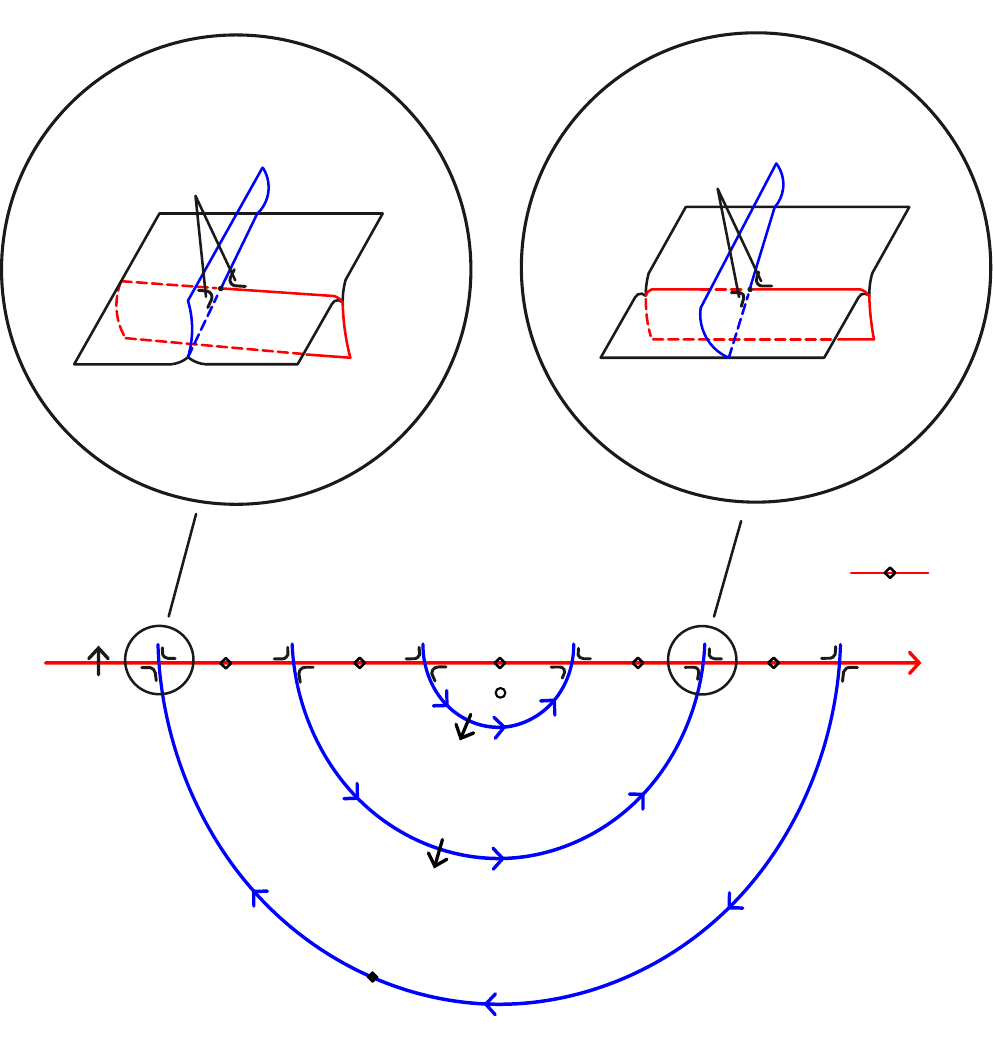}
    \put(65,5){\color{blue} $\beta_0$}
    \put(90,44){\color{red} $\alpha_0$}
    \put(8,28){$+$}
    \put(22,28){$-$}
    \put(35,28){$-$}
    \put(43.4,33){\small $-$}
    \put(8,43){$+$}
    \put(4,82){branch loci}
    \put(53.5,82.5){branch loci}
    \put(49,33){\small $w$}
    \put(30,16){$S_0$}
    \put(58,17){\color{blue} $\beta_2$}
    \put(51,28){\color{blue} $\beta_1$}
  \end{overpic}
  \caption{Type II modified Heegaard branched surface}
  \label{fig:reversing}
\end{figure}

\begin{defn}[Type II modified Heegaard branched surface] 
  We notice that $\partial S=\alpha_0\cup \beta_0$ is clockwise, indicating that along $\alpha_0\cup \beta_0$ the branch directions always point out of $S$. We can then reverse the co-orientations of the branch sectors contained in $S$ and get a new branched surface. As a result, the branch direction along $\alpha_0$ points to the $\alpha$-disk, while the branch direction along $\beta_0$ points to the $\beta$-disk. Moreover, the branch direction along the $\beta$-arcs in the interior of $S$ now points to the right instead. Again see Figure~\ref{fig:reversing}.

  We call the resulting branched surface $\mathcal{B}_{\alpha}$ the \textit{type II modified Heegaard branched surface of the (1,1) diagram associated to $\alpha$}.
  \label{defn:typeIIBS}
\end{defn}

\vspace{4pt}

\begin{rems}~\
  \begin{enumerate}
    \item Similar to~\cite{lyu2024knot}, we can regard these branched surfaces as embedded in the ambient space the knot lives in, where $\Sigma$ is identified with the Heegaard torus of the ambient space. 
    \item The type I modified branched surface depends on the source sector $T$ to be reversed. The type II modified branched surface depends on the curve in standard position. Here different choices will usually give different branched surfaces. On the other hand, the two different admissible orientations of $(\alpha,\beta)$ (for the $\alpha$-source bigon to be $\beta$-sink) would give homeomorphic branched surfaces, with homeomorphism induced by the hyperelliptic involution of $(\Sigma,z,w)$. Hence it makes sense for us to talk about modified Heegaard branched surfaces of reduced, incoherent (1,1) diagrams without specifying the orientations of the curves.
  \end{enumerate}
  \label{rems:brsfs}
\end{rems}

\begin{lem}
  Let $(\Sigma,\alpha,\beta,z,w)$ be a reduced, incoherent (1,1) diagram. Let $\mathcal{B}$ be a (type I or II) modified Heegaard branched surface of the (1,1) diagram. Then $N(\mathcal{B})$ is homeomorphic to the complement of the (1,1) knot that the diagram represents. Moreover, the vertical boundary $\partial_v N(\mathcal{B})$ contains 6 meridional annuli, with 4 of them coming from the cusps or branch locus, and 2 of them coming from the boundary circles at the 2 basepoints.
  \label{lem:vB}
\end{lem}

\begin{proof}
  As a 2-complex, $\mathcal{B}$ is obtained by attaching compression disks along the $\alpha,\beta$ curves and puncturing the two basepoints. Hence by definition of (1,1) knots, $N(\mathcal{B})$ is homeomorphic to the knot complement. We now consider its vertical boundary. There are two vertical annuli from the two boundary circles, and these are clearly meridional. Now consider the cusps from the branch locus of $\mathcal{B}$. For simplicity of narration we think of these cusps as \textit{circles} on $\partial N(\mathcal{B})$ instead of \textit{annuli}.

  \begin{figure}[!hbt]
    \begin{overpic}[scale=0.4]{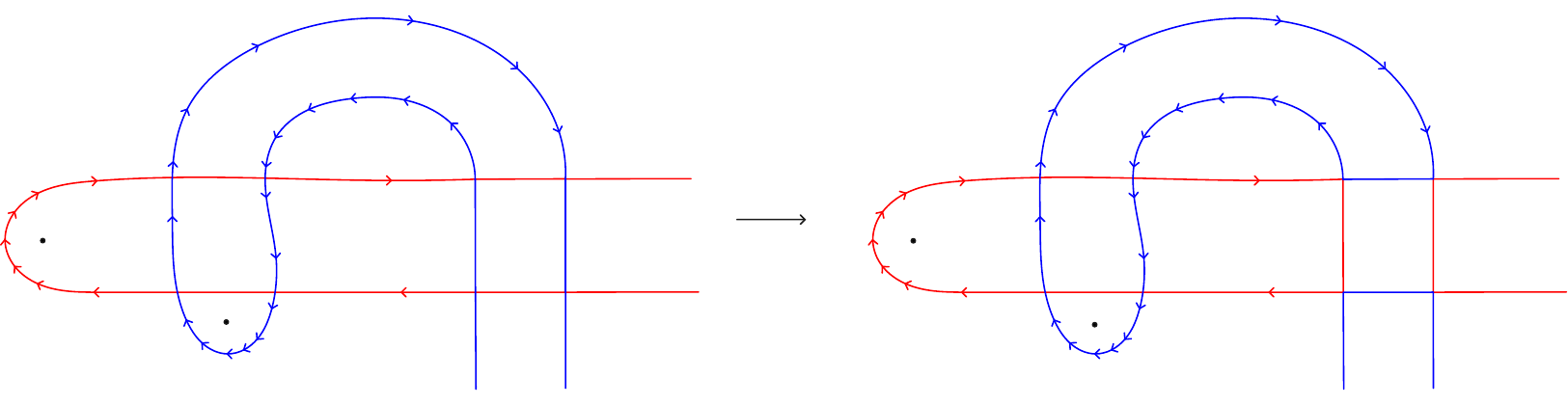}
      \put(32.1,9.8){\small $T$}
      \put(3.3,10){\small $w$}
      \put(59,10){\small $w$}
      \put(15,5){\small $z$}
      \put(70.3,5){\small $z$}
      \put(64,20){\small $C_z$}
      \put(59,4){\small $C_w$}
      \put(92,1){\small $C_{\beta}$}
      \put(95,15.5){\small $C_{\alpha}$}
    \end{overpic}
    \caption{Reversing the source sector $T$}
    \label{fig:typeIrev}
  \end{figure}

  \textbf{Case 1: $\mathcal{B}$ is type I}. Suppose we reversed a source sector $T$. This $T$ is necessarily an intersection of the $\alpha$- and $\beta$-source tubes, see Figure~\ref{fig:typeIrev}. Recall as in Figure~\ref{fig:revsector}, reversing $T$ would change the branch locus as immersed curves. After the reversing, we obtain 4 immersed curves as depicted in Figure~\ref{fig:typeIrev} right, which become the 4 cusps in $N(\mathcal{B})$.  
  
  We henceforth show that these 4 cusps are meridians. There are 2 cusps that come from $\beta$ away from $S$. One of them is a trivial curve on $\Sigma$, the disk bounded by which contains the basepoint $z$, and we call this cusp $C_z$. The other remains essential on $\Sigma$, and we call it $C_{\beta}$. See Figure~\ref{fig:typeIrev}. Similarly, we can denote the other 2 cusps $C_w$ and $C_{\alpha}$. Moreover, we denote the boundary circles at $z,w$ as $\partial_z,\partial_w$ respectively.
  
  There is a horizontal boundary component bounded by $\partial_z$ and $C_z$, which is actually the disk bounded by $C_z$ on the positive side of $\Sigma$ minus $z$. In particular it is an annulus. Since $\partial_z$ is a meridian, so is $C_z$. Similarly, we know $C_w$ is a meridian.

  There is a horizontal boundary component bounded by $C_z$ and $C_\beta$, consisting of one side of the $\beta$-disk and one side of $T$, which are joined along the boundary $\beta$-arcs of $T$; in particular it is an annulus. Since $C_z$ is a meridian, so is $C_{\beta}$. Similarly, we know $C_{\alpha}$ is also a meridian.

  It follows that in this case $\partial_v N(\mathcal{B})$ consists of 6 meridional annuli. On $\partial N(\mathcal{B})$ these annuli appear in the circular order $\partial_w,C_w,C_{\alpha},\partial_z,C_z,C_{\beta}$.

  \vspace{4pt}

  \textbf{Case 2: $\mathcal{B}$ is type II}. Suppose the branched surface is associated to $\alpha$. We first show that the branch locus of $\mathcal{B}$ consists of 4 immersed circles. As depicted in Figure~\ref{fig:reversing}, the boundary of the innermost bigon containing the basepoint $w$ becomes a circle branch locus, which we denote as $D_w$. Also, $\beta_0\cup (\alpha-\alpha_0)$ is another circle, and we denote it $D_{\alpha}$. The quadrilateral sector $S_0$ is in the $\beta$-source tube. After reversing the co-orientations, branch-locus-wise the boundary $\beta$-arcs of the $\beta$-source tube is cut along $S_0$, and re-joined by the boundary $\alpha$-arcs of $S_0$, see Figure~\ref{fig:reversing} and Figure~\ref{fig:sourcetube}. On the other hand, for the $\beta$-strands connected by $\beta_i$ ($1\leq i<n$) before the reversing, we notice that after the reversing they're connected by the concatenation of $\beta_{i+1}$ and two short $\alpha$-arcs from $\alpha_0$ instead (see the two strands connected by $\beta_1$ in Figure~\ref{fig:reversing} for an example). In particular the reversing does not change the topology of the branch locus here at $\beta_i$ ($1\leq i<n$). It follows that, after reversing the orientation of $S$, the original $\beta$ branch locus is separated into two branch locus circles, as depicted in Figure~\ref{fig:sourcetube}. We call the left circle in Figure~\ref{fig:sourcetube} $D_z$ and the right circle $D_{\beta}$. By far we have enumerated all the branch locus of $\mathcal{B}$, and there are indeed 4 (immersed) circles $D_w,D_{\alpha},D_z,D_{\beta}$. 

  \begin{figure}[!hbt]
    \begin{overpic}[scale=0.6]{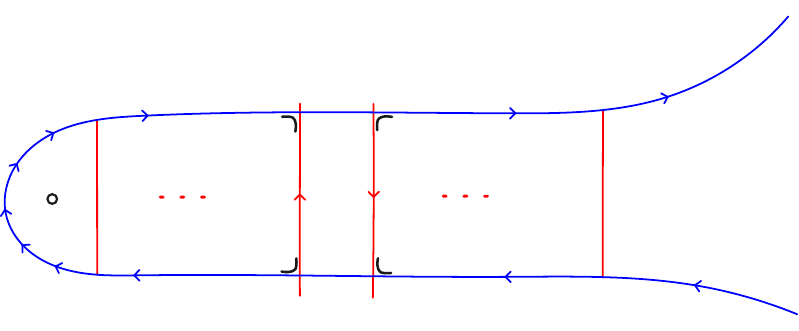}
      \put(40,16){$S_0$}
      \put(8,16){$z$}
    \end{overpic}
    \caption{The $\beta$-source tube and the reversed sector $S_0$}
    \label{fig:sourcetube}
  \end{figure}

  Now we prove that all these branch locus circles are meridians. We sill denote the two boundary circles $\partial_z, \partial_w$. There is a horizontal boundary component bounded by $\partial_w$ and $D_w$ - the punctured bigon; in particular it is an annulus. Since $\partial_w$ is a meridian, so in $D_w$.

  There is a horizontal boundary component bounded by $\partial_w$ and $D_{\alpha}$, consisting of $S-w$ on the negative side of $\Sigma$ and one side of the $\alpha$-disk, that are joined along $\alpha_0$ (again see the local pictures in Figure~\ref{fig:reversing}). In particular it is an annulus. Since $\partial_w$ is a meridian, so is $D_{\alpha}$.

  As depicted in Figure~\ref{fig:sourcetube}, $D_z$ bounds a disk on $\Sigma$ containing $z$. Moreover, this $\beta$-source tube does not intersect the interior of $S$. Hence there is a horizontal boundary component bounded by $D_z$ and $\partial_z$, namely this disk minus $z$ on the positive of $\Sigma$; in particular it is an annulus. Since $\partial_z$ is a meridian, so is $D_z$.

  For $D_{\beta}$, we show that there is a properly embedded arc on $\partial N(\mathcal{B})-\{D_z,D_w\}$, transverse to the cusps, connecting $D_z$ to $D_w$, while crossing $D_{\beta}$ in an odd number of times. If there is such an arc, then since both $D_z$ and $D_w$ are meridians, $N(\mathcal{B})$ is cut into 2 annuli by them, and this arc is an essential arc connecting two boundary components of an annuli. Now since the algebraic intersection number of $D_{\beta}$ and this arc is nontrivial, $D_{\beta}$ cannot be null-homotopic on the annulus. Since moreover it is a simple closed curve, it must be parallel to the boundary, thus a meridian.
  
  We can pick an arc from $z$ to $w$ on the positive side of $\Sigma$, while avoiding the $\beta$-curve. This arc would first travel along the $\beta$-source tube, then pass the polygon(s) and possibly some $\beta$-parallel sectors, and will finally enter the $\beta$-sink tube. Along the $\beta$-source tube, it will first cross $D_z$ once and then $D_{\beta}$ once (see Figure~\ref{fig:sourcetube}). Then it may pass some $\beta$-parallel sectors inside $S$, and will cross $D_{\beta}$ twice (at the boundary $\alpha$-arcs) each time it passes through such a sector (see the sector bounded by $\beta_1$ and $\beta_2$ in Figure~\ref{fig:reversing}). Finally, it will enter the sink tube and passes through $D_w$ once when entering the bigon. Restricting this arc between its intersections with $D_z$ and $D_w$ gives the arc we want.

  So far we have shown that $\partial_v N(\mathcal{B})$ consists of 6 meridional annuli. On $\partial N(\mathcal{B})$ these annuli appear in the circular order $\partial_w,D_w,D_{\beta},D_z,\partial_z,D_{\alpha}$.
\end{proof}

\subsection{From laminations to taut foliations}
\label{subsec:2.4}

In this subsection we prove the following proposition:

\begin{prop}
  Let $(\Sigma,\alpha,\beta,z,w)$ be an incoherent reduced (1,1) diagram representing some (1,1) non-L-space knot $K$. Let $\mathcal{B}$ be a (type I or II) modified Heegaard branched surface of the (1,1) diagram. If $\mathcal{B}$ fully carries a lamination, then $K$ is persistently foliar.
  \label{prop:lamitofoli}
\end{prop}

We make use of the following Lemma~\ref{lem:cusp}, which can be regarded as a generalization to the classical ``even meridional cusps'' argument:

\begin{figure}[!hbt]
  \begin{overpic}[scale=0.7]{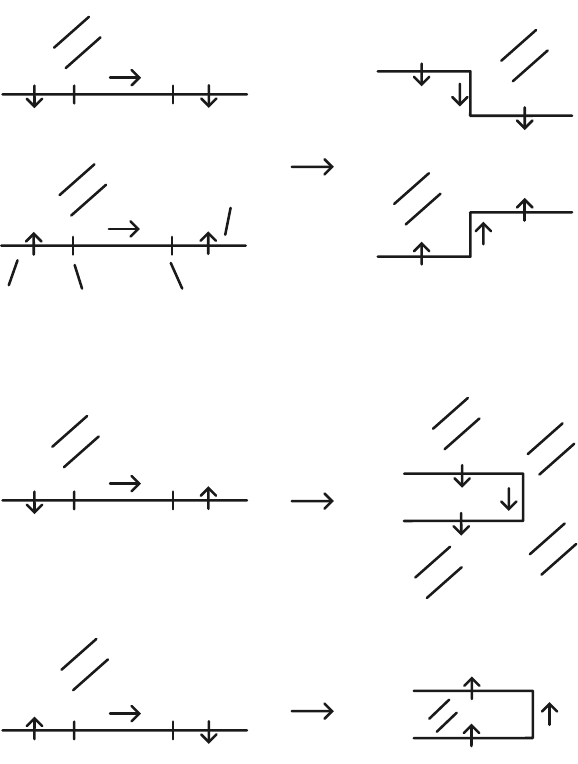}
    \put(14.5,83.5){$A_i$}
    \put(14.5,64){$A_i$}
    \put(14.5,30.5){$A_i$}
    \put(14.5,0.5){$A_i$}
    \put(62,86.5){$A_i$}
    \put(57,68.5){$A_i$}
    \put(69.5,33){$A_i$}
    \put(70,0.5){$A_i$}
    \put(2,94){$M$}
    \put(2,75){$M$}
    \put(2,43){$M$}
    \put(2,12){$M$}
    \put(-3,59){tail}
    \put(-10,56){annulus}
    \put(8,59){tail}
    \put(8,56){circle}
    \put(21,59){head circle}
    \put(23,77){head}
    \put(23,74){annulus}
    \put(60,93){$M$}
    \put(60,75){$M$}
    \put(30,50){$(a)$ stairs}
    \put(66,43){$M$}
    \put(31,25){$(b)$ cusp}
    \put(64,5){$M$}
    \put(29,-2){$(c)$ boundary}
  \end{overpic}
  \caption{Stairs, cusps, and boundaries}
  \label{fig:cusp_bdry}
\end{figure}

\begin{defn}
  Let $M$ be a \textit{knot manifold} (i.e. an irreducible, orientable, connected 3-manifold with boundary an incompressible torus), and $\mathcal{F}$ a co-oriented foliation of $M$. We call $\mathcal{F}$ \textbf{regular}, if there exist some parallel, disjoint annuli $A_1,...,A_{n}\subset \partial M$ ($n\geq 1$), such that $\mathcal{F}$ is transverse to these annuli, and tangent to $\partial M$ elsewhere.

  We can classify these transverse annuli into 3 classes, according to the co-orientations of the two tangent annuli next to them. Call the two boundary circles of $A_i$ the \textit{tail} and \textit{head} circles, such that the co-orientation of $\mathcal{F}$ along $A_i$ is pointing from the tail circle to the head circle. Call the tangent annuli attached to the tail circle the \textit{tail annulus} of $A_i$, and the tangent annuli attached to the head circle the \textit{head annulus} of $A_i$. See Figure~\ref{fig:cusp_bdry}.$(a)$.
  
  If the co-orientations of both tangent annuli next to $A_i$ point inwards or outwards, we call $A_i$ a \textbf{stair}. If the co-orientation of the tail annulus points outwards while the co-orientation of the head annulus points inwards, we call $A_i$ a \textbf{cusp}. If the co-orientation of the tail annulus points inwards while the co-orientation of the head annulus points outwards, we call $A_i$ a \textbf{boundary}. See Figure~\ref{fig:cusp_bdry}.
\end{defn}

\begin{lem}
  Let $M$ be a knot manifold. Suppose $\mathcal{F}$ is a co-oriented, regular, taut foliation of $M$ such that $\mathcal{F}$ has no disk leaves. Suppose there are no stairs. Let $b,c$ be the number of boundaries and cusps respectively. Then:
  \begin{itemize}
    \item If $c\geq b+2$, then $M$ is persistently foliar.
    \item If $c=b\geq 1$, then the slope represented by the transverse annuli is CTF-detected. 
  \end{itemize}
  \label{lem:cusp}
\end{lem}

\begin{proof}
  First notice that $b+c$ is even since $\mathcal{F}$ is co-oriented. We suppose $c\geq b+2$ and do inductions on $b$. When $b=0$ this is essentially a classical argument by Gabai (see~\cite{gabai1992taut}, or~\cite{calegari2007foliations}, Example 4.22). We remark that the situation here can be understood as that we have already capped off all the Reeb annuli. Moreover, we can blow up the leaves containing the tangent annuli, and then remove the (tangent annuli)$\times I$ part, see Figure~\ref{fig:blowup}; this is equivalent to the blowing up operations needed in~\cite{calegari2007foliations}, Example 4.22. (In particular if this generates a new disk leaf, then locally we can find a disk leaf of the original foliation.) Hence the argument in~\cite{calegari2007foliations} goes through.

  \begin{figure}[!hbt]
    \begin{overpic}[scale=0.4]{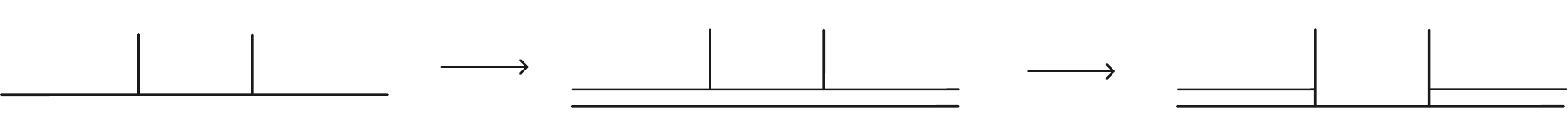}
      \put(0,-2){\small $M$}
      \put(5.5,4){\small $A_i$}
      \put(16.4,4){\small $A_{i+1}$}
      \put(11,-1){\small $B_i$}
      \put(26,5){\small blow up}
      \put(64,5){\small remove}
      \put(64,0){\small $B_i\times I$}
    \end{overpic}
    \caption{Blowing up leaves with annuli tangent to $\partial M$}
    \label{fig:blowup}
  \end{figure}

  Now suppose $b\geq 1$. Then $c\geq 1$. We let $A_1,...,A_{b+c}$ be all the transverse annuli on $\partial M$ in circular order. Let $B_i$ be the tangent annulus in between $A_i$ and $A_{i+1}$. We can pick a boundary that is next to a cusp. Suppose $A_1$ is a boundary and $A_2$ is a cusp. Consider $A_3$.

  If $A_3$ is a boundary, then we can attach an $A\times I$ ($A$ being an annulus) to the current foliation, such that $A_2$ is attached to $A\times \{0\}$, and $B_1,B_2$ are attached to $\partial A\times I$, see Figure~\ref{fig:elim_bdry}.$(a)$. We can then extend the foliation on $A_2$ through this $A\times I$ and get a new foliation $\mathcal{F'}$. $\mathcal{F'}$ is again co-oriented, taut, and regular, while both the number of boundaries and cusps decrease by 1. $\mathcal{F'}$ still has no disk leaves since we're just attaching surfaces with Euler characteristic $\leq 0$. Hence by induction $M$ is persistently foliar.

  \begin{figure}[!hbt]
    \begin{overpic}[scale=0.8]{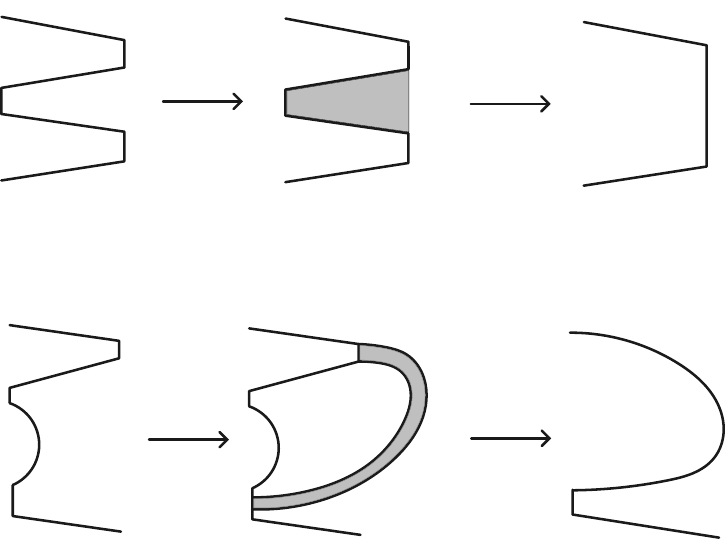}
      \put(-5,70){$M$}
      \put(13,68){$A_1$}
      \put(-5,61.5){$A_2$}
      \put(12.6,55){$A_3$}
      \put(5,66.5){$B_1$}
      \put(5,56.8){$B_2$}
      \put(18.5,65){\small attach $A\times I$}
      \put(30,45){$(a)$ $A_3$ is a boundary}
      \put(-5,28){$M$}
      \put(17,27){$A_1$}
      \put(-3,21){$A_2$}
      \put(-3,6){$A_3$}
      \put(0,14){$B_2$}
      \put(5,1){$B_3$}
      \put(19,18){\small blow up}
      \put(9.5,12){\small then attach $A\times I$}
      \put(64,18){\small fill in $A\times I$ hole}
      \put(33,-2){$(b)$ $A_3$ is a cusp}
    \end{overpic}
    \caption{Eliminating 1 boundary using 1 cusp}
    \label{fig:elim_bdry}
  \end{figure}

  Now suppose $A_3$ is a cusp. We can then blow up the leaves containing $B_2$ and $B_3$, and remove the $B_2\times I$ and $B_3\times I$ parts. We can then attach an $A\times I$ connecting the new $A_1$ to the original $A_3$, see Figure~\ref{fig:elim_bdry}.$(b)$. Since $\mathcal{F}$ has no disk leaves, after possibly blowing up leaves we can extend our foliation to this $A\times I$ by the classical argument of Gabai (\cite{gabai1992taut}, Operation 2.4.4). We can then fill in the thickened annulus hole connecting $A_2$ to $A_3$ (again by Gabai's argument). As a result, we get a new co-oriented, taut, regular foliation $\mathcal{F''}$, and again both the number of boundaries and cusps decrease by 1. Since we are just blowing up leaves and attaching $A\times I$, $\mathcal{F''}$ cannot have disk leaves. By induction $M$ is persistently foliar.

  The proof for the $b=c$ case is similar. We just note that when $b=c=1$, $A_3=A_1$ is a boundary, and after attaching and foliating the $A\times I$ we would get a foliation $\mathcal{F'}$ transverse to $\partial M$. In particular, $\partial A_1$ are circle leaves of $\mathcal{F'}|_{\partial M}$, hence $\mathcal{F'}$ CTF-detects the slope of $A_1$, as desired.
\end{proof}

\begin{rem}
  We suppose there are no stairs just for simplicity of narration. In general the stairs could be eliminated first using similar techniques.
\end{rem}

Now we can prove Proposition~\ref{prop:lamitofoli}.

\begin{proof}[Proof of Proposition~\ref{prop:lamitofoli}]
  Suppose $\mathcal{B}$ fully carries a lamination $\mathcal{L}$. After blowing up and isotoping $\mathcal{L}$ we may further assume $\partial_h N(\mathcal{B})\subset \mathcal{L}$. We can then foliate the trivial $I$-bundles of $N(\mathcal{B})|\mathcal{L}$ to obtain a foliation $\mathcal{F}$ of $N(\mathcal{B})$. $\mathcal{F}$ is co-oriented since $\mathcal{B}$ is. Moreover, $\mathcal{F}$ is regular with 4 cusps and 2 boundaries. $\mathcal{F}$ cannot have disk leaves, since the meridian of the knot complement cannot bound a disk. By Lemma~\ref{lem:cusp}, to show that $K$ is persistently foliar we only need to show that $\mathcal{F}$ is taut in $N(\mathcal{B})$.

  \vspace{6pt}

  \textbf{Claim.} There exists an oriented, closed loop $\gamma$ in $N(\mathcal{B})$ that consists of $I$-fibers of $N(\mathcal{B})$ and subarcs in $\partial_h N(\mathcal{B})$, such that 
  \begin{enumerate}
    \item the direction of the $I$-fibers agree with the co-orientation,
    \item the loop intersects every component of $\partial_h N(\mathcal{B})$. 
  \end{enumerate}

  \begin{proof}[Proof of the claim]
    We first notice that, if along $\partial N(\mathcal{B})$ some boundaries and cusps appear alternatively, then we can simply take $I$-fibers of these vertical annuli and connect them by arcs in $\partial_h N(\mathcal{B})$, and the resulting arc will agree with the co-orientation, see Figure~\ref{fig:elim_bdry}.$(a)$.
    
    \textbf{Case 1: $\mathcal{B}$ is type I}. Suppose again we reversed the sector $T$. We suppose the co-orientation is from the negative side of $\Sigma$ to the positive side of $\Sigma$ outside $T$. By the above observation, along the co-orientation we have arcs connecting $C_{\alpha}\rightarrow \partial_z\rightarrow C_z$ and $C_w\rightarrow \partial_w\rightarrow C_\beta$.
    
    Consider the reversed sector $T$. Here the co-orientation is from the positive side of $\Sigma$ to the negative side. It follows that, an $I$-fiber of $T$ together with the two sides of $T$ (as part of the horizontal boundary) would connect $C_z,C_{\beta}$ to $C_w,C_{\alpha}$. We then have the desired loop $\gamma$ as follows: \[C_{\alpha}\rightarrow \partial_z\rightarrow C_z\xrightarrow{T}C_w\rightarrow \partial_w\rightarrow C_\beta\xrightarrow{T}C_{\alpha}.\]

    Notice that this loop intersects every component of $\partial_h N(\mathcal{B})$ since it intersects every component of $\partial_v N(\mathcal{B})$.

    \vspace{4pt}
    
    \textbf{Case 2: $\mathcal{B}$ is type II}. Suppose again the branched surface is associated to $\alpha$. Here we can take $I$-fibers of $D_w$-$\partial_w$-$D_{\alpha}$-$\partial_z$-$D_z$ and connect them along $\partial_h N(\mathcal{B})$. As a result, we get a long arc connecting the two horizontal boundary components that are connected to $D_{\beta}$, and in particular this arc intersects every component of $\partial_h N(\mathcal{B})$.

    \begin{figure}[!hbt]
      \begin{overpic}{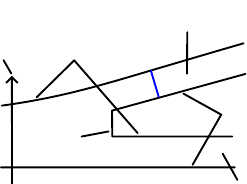}
        \put(-40,52){\small co-orientation}
        \put(23,54){\small $A_{\beta}^+$}
        \put(22,18){\small $D_{\beta}$}
        \put(70,63){\small $\beta$-disk}
        \put(91,28){\small $A_{\beta}^-$}
        \put(98,-4){\small $\Sigma$}
      \end{overpic}
      \caption{A local picture of $D_{\beta}$ in $N(\mathcal{B})$ away from $S_0$}
      \label{fig:Cbeta}
    \end{figure}

    We use $A_{\beta}^+,A_{\beta}^-$ to denote the two horizontal boundary components connected to $D_{\beta}$, and suppose $D_z$ is connected to $A_{\beta}^-$. Suppose along the co-orientation our long arc is going from $A_{\beta}^+$ to $A_{\beta}^-$. Then the $I$-fiber of $D_{\beta}$ along the co-orientation is also going from $A_{\beta}^+$ to $A_{\beta}^-$. However, an $I$-fiber through the $\beta$-disk also connects $A_{\beta}^+$ and $A_{\beta}^-$ (see Figure~\ref{fig:Cbeta}), and along the co-orientation is going in the opposite direction of the $I$-fiber of $D_{\beta}$, thus from $A_{\beta}^-$ to $A_{\beta}^+$ instead. We can then concatenate this $I$-fiber of the $\beta$ disk with our long arc, and get the desired loop $\gamma$.
  \end{proof}

  Now we use $\gamma$ to prove that our foliation $\mathcal{F}$ is taut. For any leaf $F$, we can find an $I$-fiber $I_1$ transverse to it. Suppose along the co-orientation $I_1$ goes from $H_1$ (a horizontal boundary component) to $H_2$ (another). We can then take a subarc of our loop $\gamma$ going from $H_2$ to $H_1$, and concatenate it with $I_1$ along subarcs of $H_1$ and $H_2$. Slightly perturbing the resulting loop\footnote{We are using the \textit{topologically taut} definition throughout this paper, where such perturbing is possible for even the $C^0$ foliations (see~\cite{kazez2016taut}, section 3).} near the horizontal boundary components we would get a transverse loop that intersects $F$. This proves the tautness of $\mathcal{F}$.
\end{proof}

\section{Reductions on diagrams}
\label{sec:3}

The rest of this paper is dedicated to proving the following proposition:

\begin{prop}
  Let $(\Sigma,\alpha,\beta,z,w)$ be a reduced, incoherent (1,1) diagram. Then there exists a (type I or II) modified Heegaard branched surface of the diagram that fully carries a lamination.
  \label{prop:lami!}
\end{prop}

We will use an inductive argument to prove Proposition~\ref{prop:lami!}. In this section~\ref{sec:3} we describe how we could reduce an incoherent (1,1) diagram to a simpler one, such that the reduction behaves well with the branched surfaces we defined. In sections~\ref{sec:4} and~\ref{sec:5} we deal with the two possible terminating cases of our reductions respectively. A proof of Proposition~\ref{prop:lami!} modulo the terminating cases is given at the end of this section.

\subsection{Carrying curve}
\label{subsec:3.1}

In this subsection we remind the readers of the carrying curve defined in~\cite{lyu2024knot}, subsection 3.2.

\begin{defn}[\cite{lyu2024knot}, Definition 3.5]
  Let $(\Sigma,\alpha,\beta,z,w)$ be a reduced, non-simple (1,1)-diagram. An essential curve $\gamma$ is said to carry $\alpha$ if 
  \begin{enumerate}[(i)]
      \item it is disjoint from and isotopic to $\alpha$ in $\Sigma$ (without the marked points $z,w$),
      \item it only intersects the $\alpha$-parallel $\beta$-arcs.
      \item $(\Sigma,\gamma,\beta,z,w)$ is also a reduced (1,1)-diagram.
  \end{enumerate}
  \label{def:carrying}
\end{defn}

\begin{prop}[\cite{lyu2024knot}, Proposition 3.7]
  Let $(\Sigma,\alpha,\beta,z,w)$ be a reduced, non-simple (1,1)-diagram. Then there exists $\gamma$ carrying $\alpha$ and it is unique up to isotopy in the twice-punctured torus $(\Sigma,z,w)$. In fact, this $\gamma$ passes through all $\alpha$-parallel sectors and intersects each $\alpha$-parallel $\beta$-arc exactly once, see Figure~\ref{fig:carrying}.
  \label{prop:carrying}
\end{prop}

\begin{figure}[!htb]
  \begin{overpic}[scale=0.47]{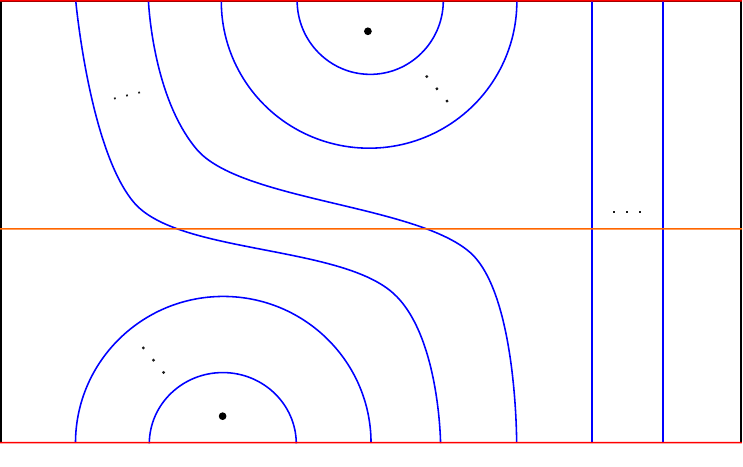}
      \put(95,34){$\color{or} \gamma$}
      \put(95,4){$\color{red} \alpha$}
      \put(91,50){$\color{ao} \beta$}
  \end{overpic}
  \caption{Carrying curve}
  \label{fig:carrying}
\end{figure}

By replacing $\alpha$ with its carrying curve $\gamma$, we get a ``simpler'' (1,1) diagram $(\Sigma,\gamma,\beta,z,w)$. More precisely, we observe the following lemma:

\begin{lem}
  Let $(\Sigma,\alpha,\beta,z,w)$ be a reduced, non-simple (1,1)-diagram. Let $\gamma$ be the carrying curve of $\alpha$. Then the geometric intersection number of $(\gamma,\beta)$ is strictly less than that of $(\alpha,\beta)$.
  \label{lem:simpler}
\end{lem}

\begin{proof}
  See Figure~\ref{fig:carrying}. In fact, if the $(\alpha,\beta)$-diagram is parametrized by the 4-tuple $(p,q,r,s)$ (see also Figure~\ref{fig:reduced11}), then the $(\alpha,\beta)$-intersection number is $p$, while the $(\gamma,\beta)$-intersection number is $(p-2q)$. Since the $(\alpha,\beta)$-diagram is non-simple, $q>0$. The lemma then follows.
\end{proof}

\subsection{Type I and type II reductions}
\label{subsec:3.2}

In this subsection we define type I and type II reductions. The diagrammatic behaviors for the two types are quite different, and we have different strategies towards them.

\begin{figure}[!hbt]
  \begin{overpic}{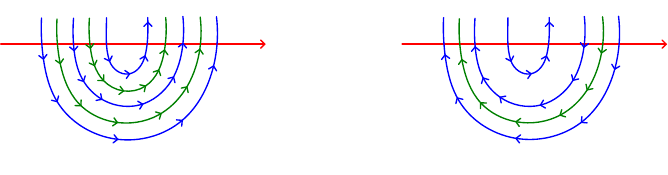}
    \put(40,18){\color{red} $\alpha$}
    \put(100,18){\color{red} $\alpha$}
    \put(31,10){\color{blue} $\beta$}
    \put(91,10){\color{blue} $\beta$}
    \put(28.3,24){\color{green} $\delta$}
    \put(88.3,24){\color{green} $\delta$}
    \put(5,0){$(a)$ Type I reduction}
    \put(65,0){$(b)$ Type II reduction}
  \end{overpic}
  \caption{Two types of reductions}
  \label{fig:reductions}
\end{figure}

\begin{defn}
  Let $(\Sigma,\alpha,\beta,z,w)$ be a reduced, incoherent (1,1) diagram. A \textit{reduction} on (1,1) diagrams is the replacement of one of $\alpha,\beta$ with its carrying curve. We denote $(\Sigma,\alpha,\beta,z,w) \rightarrow (\Sigma,\gamma,\beta,z,w)$ if we are replacing $\alpha$ with $\gamma$.

  Suppose $\delta$ is the carrying curve of $\beta$, and moreover suppose $(\Sigma,\alpha,\delta,z,w)$ is non-simple. The reduction $(\Sigma,\alpha,\beta,z,w) \rightarrow (\Sigma,\alpha,\delta,z,w)$ is called \textbf{type I} if there is only one $\beta$-arc in the interior of a $(\alpha,\delta)$-bigon (there are two such bigons, and an equal number of $\beta$-arcs in each bigon by the hyperelliptic involution symmetry), see Figure~\ref{fig:reductions}.$(a)$. The reduction is called \textbf{type II} if there are more than one $\beta$-arcs in the interior of a $(\alpha,\delta)$-bigon, see Figure~\ref{fig:reductions}.$(b)$.
  \label{defn:reductions}
\end{defn}

\vspace{3pt}

\begin{rem}
  In this definition we did not specify the orientation of $\delta$. This is a little complicated. In fact, since the orientation of $\alpha$ is given, there is only one orientation of $\delta$ that allows us to define modified Heegaard branched surfaces. \textbf{For type I reductions, we can always pick $\delta$ to be of the same orientation as $\beta$} (when regarded as curves on the torus $\Sigma$).

  This is not always the case for type II reductions, see Figure~\ref{fig:reductions}.$(b)$, where $\delta$ has the same orientation as $\beta$, but with the orientations of $(\alpha,\delta)$ we cannot define modified Heegaard branched surfaces. In fact, for type II reductions we do not pick the orientation of $\delta$ to define modified Heegaard branched surfaces for the $(\alpha,\delta)$-diagram; we will pick the opposite direction instead (again see Figure~\ref{fig:reductions}.$(b)$). The spoiler is that whenever there is a type II reduction, we can find a type I modified Heegaard branched surface of the original diagram that fully carries a lamination. This is done in section~\ref{sec:5}, and makes use of the branched surface in~\cite{lyu2024knot}, where the orientations of $(\alpha,\delta)$ need to be different from those defining modified Heegaard branched surfaces (see Definition~\ref{def:associated}).
  \label{rem:delta_ori}
\end{rem}

\subsection{Admissible reductions}
\label{subsec:3.3}

We can always do reductions on a non-simple (1,1) diagram to get a simpler diagram. However, not all such reductions behave well with our modified Heegaard branched surfaces. In this subsection we discuss when a reduction is ``good''. Due to the spoiler above in Remark~\ref{rem:delta_ori}, we can focus on type I reductions.

\begin{defn}
  Let $(\Sigma,\alpha,\beta,z,w)$ be a reduced, incoherent (1,1) diagram, and $\delta$ be the carrying curve of $\beta$, such that the reduction $(\Sigma,\alpha,\beta,z,w) \rightarrow (\Sigma,\alpha,\delta,z,w)$ is of type I. We say this reduction is \textbf{admissible} to a (type I or II) modified Heegaard branched surface $\mathcal{B}$ of $(\Sigma,\alpha,\beta,z,w)$, if the $\beta$-sink tube is disjoint from the reversed region of $\mathcal{B}$ (both as closed regions).
\end{defn}

\begin{lem}
  Suppose the reduction $(\Sigma,\alpha,\beta,z,w) \rightarrow (\Sigma,\alpha,\delta,z,w)$ is of type I and is admissible to a (type I or II) modified Heegaard branched surface $\mathcal{B}$ of $(\Sigma,\alpha,\beta,z,w)$. Then there is a (type I or II) modified Heegaard branched surface $\mathcal{C}$ of $(\Sigma,\alpha,\delta,z,w)$, such that if $\mathcal{C}$ fully carries a lamination, then $\mathcal{B}$ fully carries a lamination.
  \label{lem:admissible}
\end{lem}

\begin{figure}[!hbt]
  \begin{overpic}[scale=0.4]{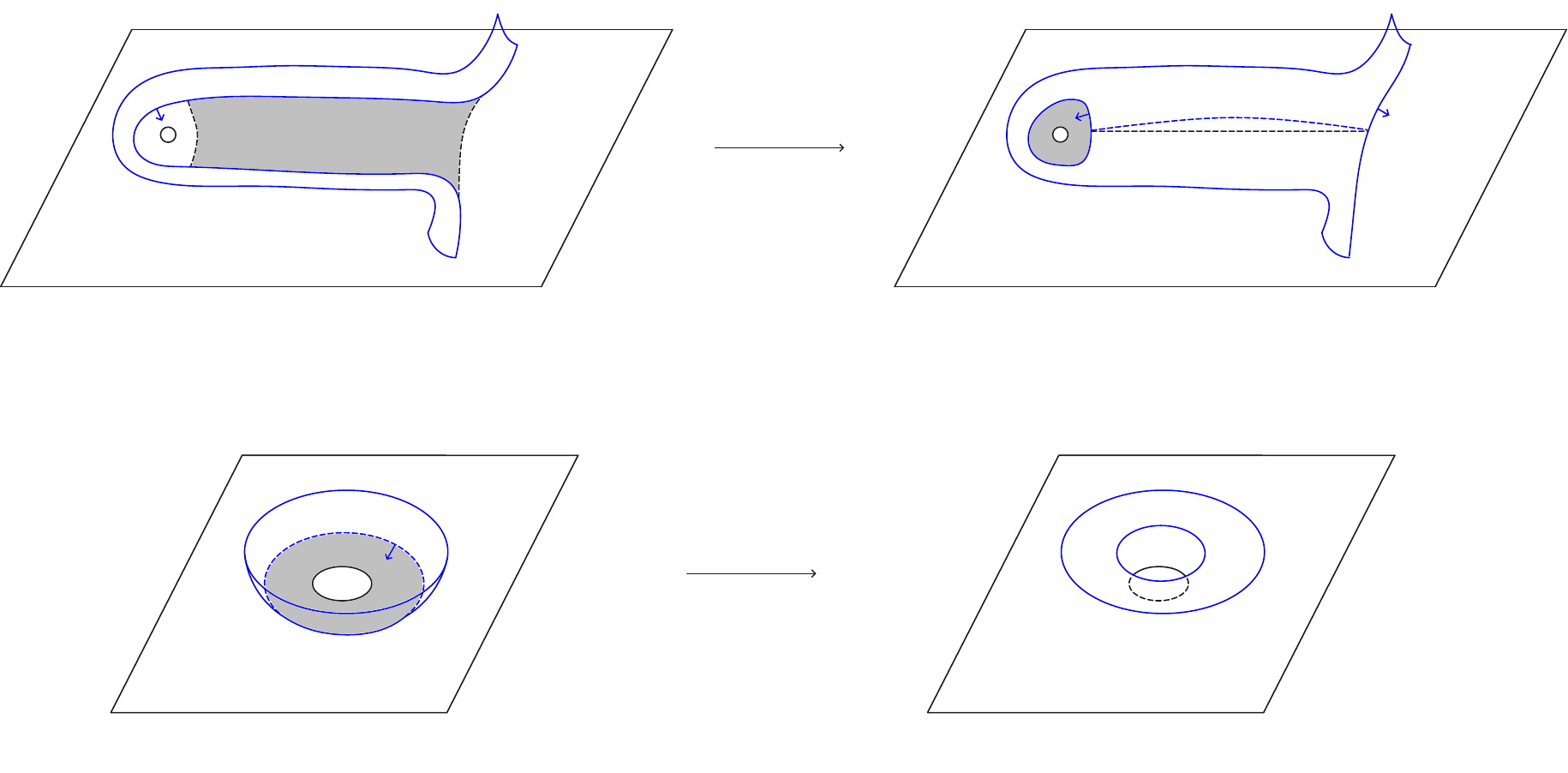}
    \put(8.7,40){\tiny $w$}
    \put(3,32){\Small $\Sigma$}
    \put(21.2,11.7){\tiny $w$}
    \put(10,5){\Small $\Sigma$}
    \put(31,26){$(a)$ Split the $\beta$-sink tube}
    \put(26,-2){$(b)$ Split the punctured disk at $w$}
  \end{overpic}
  \caption{Fully splitting the $\beta$-sink tube}
  \label{fig:full_split}
\end{figure}

\begin{proof}
  Since the sink tube is disjoint from the reversed region, we can fully split the $\beta$-sink tube as in Figure~\ref{fig:full_split}.$(a)$. More precisely, we split along the shaded disk in Figure~\ref{fig:full_split}.$(a)$ left, where the splitting is locally modelled on the right picture of Figure~\ref{fig:splittings}. The rest of the $\beta$-curve on $\Sigma$ contains 2 circles, 1 parallel to $\delta$ and 1 inside the $\beta$-sink bigon bounding a disk punctured at a basepoint. We can then further split along the shaded annulus as in Figure~\ref{fig:full_split}.$(b)$. Call the resulting branched surface $\mathcal{C'}$. Now the $\beta$-disk becomes an annulus, with one boundary circle attached to $\Sigma$ along a curve parallel to $\delta$. Call this curve $\beta_{\delta}$. 
  
  We can orient $\beta_{\delta}$ such that the branch direction points to the left (when observed from the positive side of $\Sigma$). Since the reduction $(\Sigma,\alpha,\beta,z,w) \rightarrow (\Sigma,\alpha,\delta,z,w)$ is of type I, this orientation of $\beta_{\delta}$ agrees with the orientation of $\delta$ (recall we pick the orientation of $\delta$ to be the same as $\beta$ for type I reductions). Hence $\mathcal{C'}$ actually can be regarded as a (type I or II) modified Heegaard branched surface of $(\Sigma,\alpha,\delta,z,w)$, with the $\delta$-disk further punctured. Denote the corresponding (type I or II) modified Heegaard branched surface of $(\Sigma,\alpha,\delta,z,w)$ as $\mathcal{C}$. It is clear that if $\mathcal{C}$ fully carries a lamination, then $\mathcal{C'}$ also fully carries a lamination. Moreover, since $\mathcal{C'}$ is obtained by doing splittings on $\mathcal{B}$, $\mathcal{B}$ fully carries a lamination if $\mathcal{C'}$ does. This $\mathcal{C}$ is what we want.
\end{proof}

\vspace{2pt}

\begin{rems}~\
  \begin{enumerate}
    \item According to the proof, this $\mathcal{C}$ is actually determined by $\mathcal{B}$ and the admissible reduction. We call this $\mathcal{C}$ \textbf{the reduced branched surface of $\mathcal{B}$ under the admissible reduction $(\Sigma,\alpha,\beta,z,w) \rightarrow (\Sigma,\alpha,\delta,z,w)$}.
    \item If the reduction is of type II, then $\beta_{\delta}$ does not necessarily have the same orientation as $\delta$. In fact, if the $(\alpha,\delta)$-bigon contains an even number of $\beta$-arcs, then $\beta_{\delta}$ and $\alpha$ would form sink/source bigons instead.
    \item Notice that in the proof we do not specify an arc in standard position. One should be convinced that for $\gamma$ carrying $\alpha$ and admissible reduction $(\Sigma,\alpha,\beta,z,w) \rightarrow (\Sigma,\gamma,\beta,z,w)$ (naturally here we require the $\alpha$-sink tube to be disjoint from the reversed region) this lemma still holds.
  \end{enumerate}
  \label{rems:admissible}
\end{rems}

The following lemma enables an inductive argument for type I reductions and type I branched surfaces.

\begin{lem}
  Let $(\Sigma,\alpha,\beta,z,w)$ be a reduced, incoherent (1,1) diagram, and $\delta$ be the carrying curve of $\beta$. Suppose the reduction $(\Sigma,\alpha,\beta,z,w) \rightarrow (\Sigma,\alpha,\delta,z,w)$ is of type I, and $(\Sigma,\alpha,\delta,z,w)$ is incoherent. Suppose moreover that there is a type I modified Heegaard branched surface of $(\Sigma,\alpha,\delta,z,w)$ that fully carries a lamination. Then there is a type I modified Heegaard branched surface of $(\Sigma,\alpha,\beta,z,w)$ that fully carries a lamination.
  \label{lem:typeIreds}
\end{lem}

\begin{proof}
  We show that any type I modified Heegaard branched surface of $(\Sigma,\alpha,\delta,z,w)$ is actually the reduced branched surface of a type I modified Heegaard branched surface of $(\Sigma,\alpha,\beta,z,w)$, under the admissible reduction $(\Sigma,\alpha,\beta,z,w) \rightarrow (\Sigma,\alpha,\delta,z,w)$. We notice that since the reduction is of type I, $\beta$ and $\delta$ are of the same orientation (as oriented curves on $\Sigma$).

  We can consider the $\delta$-source tube of the $(\alpha,\delta)$-diagram, and its position in the $(\alpha,\beta)$-diagram. Since the reduction $(\Sigma,\alpha,\beta,z,w) \rightarrow (\Sigma,\alpha,\delta,z,w)$ is type I, the $\delta$-source bigon (of the $(\alpha,\delta)$-diagram) contains a single arc of $\beta$. Now consider the $\delta$-source sector (of the $(\alpha,\delta)$-diagram) next to the $\delta$-source bigon. Since the $(\alpha,\beta)$-diagram is reduced, and $\beta$ is disjoint from $\delta$, $\beta$-arcs in this $\delta$-source sector are parallel arcs connecting the two boundary $\alpha$-arcs. Since the boundary $\alpha$-arc shared with the $\delta$-source bigon intersects $\beta$ twice, this $\delta$-source sector contains two $\beta$-arcs, and in particular the other boundary $\alpha$-arc of this sector also intersects $\beta$ twice. By an inductive argument along the $\delta$-source tube, we know every $\delta$-source sector contains two $\beta$-arcs inside. Since $\beta$ and $\delta$ are of the same orientation, the two $\beta$-arcs inside are the boundary $\beta$-arcs of a $\beta$-source sector. It follows that \textit{a part of} the $\beta$-source tube is nested in the $\delta$-source tube, see Figure~\ref{fig:typeItube}.

  \begin{figure}[!hbt]
    \begin{overpic}[scale=0.9]{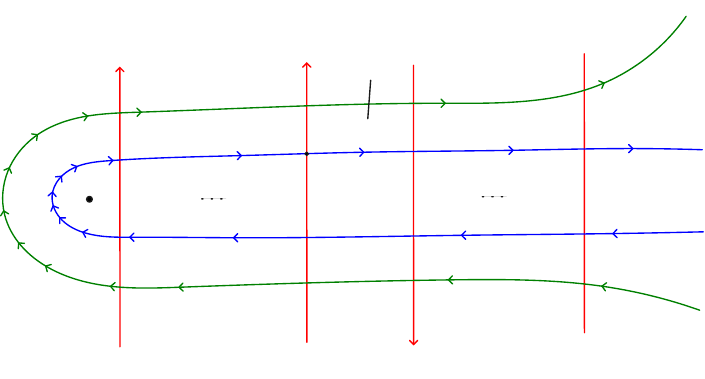}
      \put(10,26){$z$}
      \put(18,44){\color{red} $\alpha$}
      \put(97,17){\color{blue} $\beta$}
      \put(95,7){\color{green} $\delta$}
      \put(48,25){$B_D$}
      \put(52,43){$D$}
      \put(44,33){$X$}
    \end{overpic}
    \caption{Nested source tubes in a type I reduction}
    \label{fig:typeItube}
  \end{figure}

  Now for any type I modified Heegaard branched surface $\mathcal{B}_D$ reversing the source sector $D$ of the $(\alpha,\delta)$-diagram, we knot this sector must be $\delta$-source, and thus lies in the $\delta$-source tube. Then there is a $\beta$-source sector $B_D$ inside it, which is actually a source sector of the $(\alpha,\beta)$-diagram, see Figure~\ref{fig:typeItube}. We claim the type I modified Heegaard branched surface $\mathcal{B}_{B_D}$ of $(\Sigma,\alpha,\beta,z,w)$ reversing $B_D$ is what we want. We only need to check that the reduction is admissible for $\mathcal{B}_{B_D}$, for then $B_D$ is isotoped to $D$ when isotoping $\beta_{\delta}$ to $\delta$. 
  
  To show the reduction is admissible, it suffices to show the four vertices of $B_D$ are disjoint from the $\beta$-sink tube. Suppose a vertex $X$ is on the $\beta$-sink tube. Then there is a $\beta$-sink $\alpha$-arc of the $(\alpha,\beta)$-diagram with one endpoint $X$. But the two $\alpha$-arcs (of the $(\alpha,\beta)$-diagram) with one endpoint $X$ are either source (on the boundary of $B_D$), or parallel (intersecting $\delta$). This gives a contradiction. Hence $B_D$ is disjoint from the $\beta$-sink tube, and the reduction is admissible.
\end{proof}

On the other hand, the following lemma enables an inductive argument for type I reductions and type II branched surfaces. We put $\alpha$ in standard position and recall the notations $S$ and $\alpha_0$ we used in subsection~\ref{subsec:2.3} (see also Figure~\ref{fig:reversing} or Figure~\ref{fig:reducible}).

\begin{lem}
  Let $(\Sigma,\alpha,\beta,z,w)$ be a reduced, incoherent (1,1) diagram, and $\gamma$ be the carrying curve of $\alpha$. Suppose the $\beta$-arcs intersecting $\alpha_0$ on the other side of $S$ are all vertical arcs. Then the diagram $(\Sigma,\gamma,\beta,z,w)$ is incoherent and the reduction $(\Sigma,\alpha,\beta,z,w) \rightarrow (\Sigma,\gamma,\beta,z,w)$ is of type I. Moreover, if the type II modified Heegaard branched surface of $(\Sigma,\gamma,\beta,z,w)$ associated to $\gamma$ fully carries a lamination, then the type II modified Heegaard branched surface of $(\Sigma,\alpha,\beta,z,w)$ associated to $\alpha$ also fully carries a lamination.
  \label{lem:typeIIreds}
\end{lem}

\begin{figure}[!hbt]
  \begin{overpic}[scale=0.9]{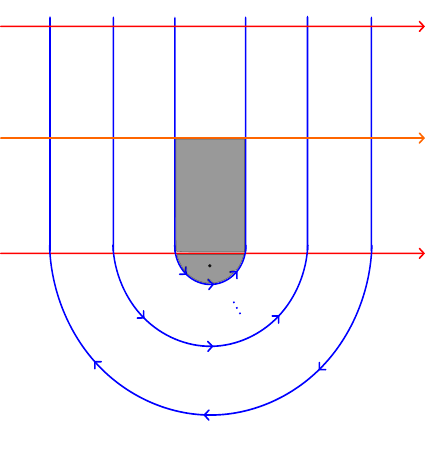}
    \put(19,10){$S$}
    \put(11.5,44){$\overbrace{\qquad\qquad\qquad\qquad\qquad\qquad\quad\;\;\:}$}
    \put(44,49){\color{red} $\alpha_0$}
    \put(93,65){\color{or} $\gamma$}
    \put(93,90){\color{red} $\alpha$}
    \put(68,12){\color{blue} $\beta_0$}
    \put(58,22){\color{blue} $\beta_n$}
    \put(53,36){\color{blue} $\beta_1$}
    \put(47,40){\Small $w$}
  \end{overpic}
  \caption{Admissible reductions for type II branched surfaces}
  \label{fig:reducible}
\end{figure}

\begin{proof}
  We recall that all vertical $\beta$-arcs are $\alpha$-parallel, hence $\gamma$ would pass through all these arcs, see Figure~\ref{fig:reducible}. In particular, the diagram $(\Sigma,\gamma,\beta,z,w)$ is incoherent. Moreover, the shaded $(\gamma,\beta)$-bigon contains only 1 $\alpha$-arc in the interior, indicating that the reduction $(\Sigma,\alpha,\beta,z,w) \rightarrow (\Sigma,\gamma,\beta,z,w)$ is type I.

  Now consider the position of the $\alpha$-sink tube. Recall it consists of all the $\alpha$-sink quadrilaterals. By definition the boundary $\beta$-arcs of these quadrilaterals are $\alpha$-sink. Now that $\alpha$ is in standard position, the $\alpha$-sink $\beta$-arcs are rainbow arcs on the side of $\alpha$ different from $S$. In particular the $\alpha$-sink tube is disjoint from $S$ here, which only intersects vertical $\beta$-arcs on the other side of $\alpha_0$.

  It follows that the reduction $(\Sigma,\alpha,\beta,z,w) \rightarrow (\Sigma,\gamma,\beta,z,w)$ is admissible to $\mathcal{B}_{\alpha}$ the type II branched surface associated to $\alpha$. Moreover, as can be seen in Figure~\ref{fig:reducible}, the reduced branched surface of $\mathcal{B}_{\alpha}$ is exactly the type II branched surface of $(\Sigma,\gamma,\beta,z,w)$ associated to $\gamma$. The lemma then follows from Lemma~\ref{lem:admissible}.
\end{proof}

\subsection{Two terminating cases}
\label{subsec:3.4}

In this subsection we describe how we are going to reduce the diagram, and introduce the two terminating cases.

For simplicity of narration we say an arc is \textbf{connected to} a region (this ``region'' could also be an arc) if it is not contained in the region but intersects the region. For example, the vertical $\beta$-arcs in Lemma~\ref{lem:typeIIreds} are \textit{connected to} $S$ (see also Figure~\ref{fig:reducible}).

Suppose we have a reduced, incoherent (1,1) diagram $(\Sigma,\alpha,\beta,z,w)$, with $\alpha$ in standard position. Whenever all $\beta$-arcs connected to $S$ are vertical, by Lemma~\ref{lem:typeIIreds} we can reasonably replace $\alpha$ by its carrying curve $\gamma$ and get a new, simpler diagram. Now suppose not all $\beta$-arcs connected to $\gamma$ are vertical. We then further consider the two $\beta$-arcs connected to $\beta_1$ (i.e. the boundary $\beta$-arc of the $\beta$-sink bigon, see Figure~\ref{fig:reducible}). They are either both vertical arcs or both rainbow arcs, for otherwise the $\beta$-sink bigon would intersect some hexagon or octagon region, indicating that the $\beta$-sink tube is empty and that the diagram is coherent.

\vspace{6pt}

\textbf{Case 1: the $\beta$-arcs connected to $\beta_1$ are vertical arcs.} In this case we call the diagram \textit{monotone}. Formally we have the following definition:

\begin{defn}
  Let $(\Sigma,\alpha,\beta,z,w)$ be a reduced, incoherent (1,1) diagram, with $\alpha$ in standard position. Recall the region $S$ in subsection~\ref{subsec:2.3}. This diagram is called \textbf{$\alpha$-monotone} if
  \begin{enumerate}
    \item the $\beta$-arcs connected to the $\beta$-sink bigon are vertical, and
    \item not all $\beta$-arcs connected to $S$ are vertical.
  \end{enumerate}
  \label{defn:monotone}
\end{defn}

Examples of monotone diagrams include many strongly almost coherent diagrams (as defined in~\cite{binns20231}). In particular the following diagram $(13,4,1,7)$ (see Figure~\ref{fig:13_4_1_7}) is monotone and strongly almost coherent, where the $\beta$-arcs are connected to $S$ at the nodes $10,11,12,13,1,2$:

\begin{figure}[!hbt]
  \begin{overpic}{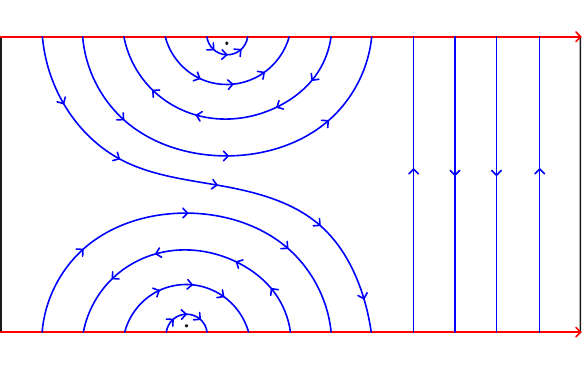}
    \put(6.5,2){1}
    \put(13.2,2){2}
    \put(20.1,2){3}
    \put(27,2){4}
    \put(34.5,2){5}
    \put(41.7,2){6}
    \put(49,2){7}
    \put(56.1,2){8}
    \put(63,2){9}
    \put(69,2){10}
    \put(76,2){11}
    \put(83.3,2){12}
    \put(90.5,2){13}
    \put(6.5,58.5){8}
    \put(13.2,58.5){9}
    \put(19,58.5){10}
    \put(26,58.5){11}
    \put(33,58.5){12}
    \put(40.2,58.5){13}
    \put(49,58.5){1}
    \put(56,58.5){2}
    \put(63,58.5){3}
    \put(70,58.5){4}
    \put(77,58.5){5}
    \put(84.2,58.5){6}
    \put(91.4,58.5){7}
    \put(98,2){\color{red} $\alpha$}
    \put(60,20){\color{blue} $\beta$}
    \put(30,6.7){\tiny $z$}
    \put(39.6,55.3){\tiny $w$}
  \end{overpic}
  \caption{The (1,1) diagram $(13,4,1,7)$}
  \label{fig:13_4_1_7}
\end{figure}

We will eventually prove the following proposition in Section~\ref{sec:4}:

\begin{prop}
  Let $(\Sigma,\alpha,\beta,z,w)$ be a reduced, incoherent (1,1) diagram, with $\alpha$ in standard position. Suppose this diagram is $\alpha$-monotone. Then the type II modified Heegaard branched surface associated to $\alpha$ of this diagram fully carries a lamination. 
  \label{prop:monotone}
\end{prop}

\begin{rem}
  It is necessary for us to use type II modified Heegaard branched surfaces here. In fact, for the diagram in Figure~\ref{fig:13_4_1_7}, the type I modified Heegaard branched surface (there is only one!) admits a twisted disk of contact, thus never fully carries a lamination. 
\end{rem}

\vspace{6pt}

\textbf{Case 2: the $\beta$-arcs connected to $\beta_1$ are rainbow arcs.} In this case we call the diagram \textit{strongly incoherent}. More precisely we have the following definition:

\begin{defn}
  Let $(\Sigma,\alpha,\beta,z,w)$ be a reduced, incoherent (1,1) diagram. The diagram is called \textbf{$\alpha$-strongly incoherent}, if when we put $\alpha$ in standard position, the rainbow arcs next to the innermost rainbow arcs are inconsistent with the innermost rainbow arcs.
  \label{defn:strongly_incoherent}
\end{defn}

\vspace{1pt}

\begin{rems}~\
  \begin{enumerate}
    \item In our case where the $\beta$-arcs connected to $\beta_1$ are rainbow arcs, the diagram is actually $\beta$-strongly incoherent.
    \item Suppose $\alpha$ is in standard position and recall the region $S$ we defined in subsection~\ref{subsec:2.3}. Then the diagram is $\alpha$-strongly incoherent if and only if there is a single $\beta$-arc $\beta_1$ in the interior of $S$.
    \item Suppose the reduction $(\Sigma,\alpha,\beta,z,w) \rightarrow (\Sigma,\alpha,\delta,z,w)$ is of type II. Then $(\Sigma,\alpha,\beta,z,w)$ is $\alpha$-strongly incoherent. 
  \end{enumerate}
\end{rems}

We will prove the following proposition in Section~\ref{sec:5}:

\begin{prop}
  Let $(\Sigma,\alpha,\beta,z,w)$ be a reduced, ($\alpha$- or $\beta$-)strongly incoherent (1,1) diagram. Then there exists a type I modified Heegaard branched surface of the diagram that fully carries a lamination.
  \label{prop:strongly_incoherent}
\end{prop}

\subsection{The actual reduction procedure}
\label{subsec:3.5}

We can now give a proof of Proposition~\ref{prop:lami!}, modulo Proposition~\ref{prop:monotone} and Proposition~\ref{prop:strongly_incoherent} above.

\begin{proof}[Proof of Proposition~\ref{prop:lami!}]
  We can place $\alpha$ in standard position. Again recall the notations in subsection~\ref{subsec:2.3}. If all $\beta$-arcs connected to $S$ are vertical, we can replace $\alpha$ with its carrying curve $\gamma$ and get a new (1,1) diagram $(\Sigma,\gamma,\beta,z,w)$. Now by Lemma~\ref{lem:typeIIreds}, this new diagram $(\Sigma,\gamma,\beta,z,w)$ is incoherent and the reduction $(\Sigma,\alpha,\beta,z,w) \rightarrow (\Sigma,\gamma,\beta,z,w)$ is of type I. We can then put $\gamma=\gamma_1$ in standard position and and think of the $\beta$-arcs connected to the new ``$S$'' similarly defined in the diagram $(\Sigma,\gamma_1,\beta,z,w)$. If all these $\beta$-arcs are still vertical, then we can replace $\gamma_1$ by its carrying curve $\gamma_2$ and consider the new diagram $(\Sigma,\gamma_2,\beta,z,w)$. Again by Lemma~\ref{lem:typeIIreds}, this new diagram is still incoherent, and $(\Sigma,\gamma_1,\beta,z,w)\rightarrow (\Sigma,\gamma_2,\beta,z,w)$ is of type I.

  We can repeat the procedure whenever possible. Since the (1,1) diagrams are non-simple (actually \textit{incoherent}), by Lemma~\ref{lem:simpler} this procedure must stop after finitely many steps. Suppose it stops at some incoherent diagram $(\Sigma,\gamma_n,\beta,z,w)$. Then this diagram is either $\gamma_n$-monotone or $\beta$-strongly incoherent (see discussions in subsection~\ref{subsec:3.4}). 
  
  Suppose first it is $\gamma_n$-monotone. Then by Proposition~\ref{prop:monotone} the type II branched surface associated to $\gamma_n$ fully carries a lamination. Then we can apply Lemma~\ref{lem:typeIIreds} to the reduction $(\Sigma,\gamma_{n-1},\beta,z,w)\rightarrow(\Sigma,\gamma_n,\beta,z,w)$ and conclude that the type II branched surface of the diagram $(\Sigma,\gamma_{n-1},\beta,z,w)$ associated to $\gamma_{n-1}$ fully carries a lamination. By induction, we know actually the type II branched surface of the diagram $(\Sigma,\gamma_k,\beta,z,w)$ associated to $\gamma_k$ fully carries a lamination for all $k$, and finally the type II branched surface of $(\Sigma,\alpha,\beta,z,w)$ associated to $\alpha$ fully carries a lamination.
  
  Now suppose the diagram $(\Sigma,\gamma_n,\beta,z,w)$ is $\beta$-strongly incoherent. Then by Proposition~\ref{prop:strongly_incoherent} there exists some type I branched surface of $(\Sigma,\gamma_n,\beta,z,w)$ fully carrying a lamination. We can then repeatedly apply Lemma~\ref{lem:typeIreds} backwards along our reductions of diagrams (notice all the reductions are of type I), and finally we get a type I branched surface of the diagram $(\Sigma,\alpha,\beta,z,w)$ fully carrying a lamination. This concludes the proof.
\end{proof}

Now Theorem~\ref{thm:main} immediately follows:

\begin{proof}[Proof of Theorem~\ref{thm:main}]
  By Theorem~\ref{thm:11L}, the (1,1) non-L-space knot is represented by some reduced, incoherent (1,1) diagram. By Proposition~\ref{prop:lami!}, there exists a type I or II modified Heegaard branched surface of the diagram that fully carries a lamination. Then by Proposition~\ref{prop:lamitofoli}, this implies that the knot is persistently foliar.
\end{proof}

\section{Monotone diagrams}
\label{sec:4}

In this section we prove Proposition~\ref{prop:monotone}. We will first briefly analyze the diagram and then describe a strategy to split the branched surface to eliminate sink disks.

\subsection{Diagrammatic analysis}
\label{subsec:4.1}

In this subsection we prepare some lemmas by analyzing the diagrams. Without loss of generality we assume the $\alpha$-curve is in standard position and pointing to the right. Then (for there to exist modified Heegaard branched surfaces) the innermost rainbow $\beta$-arcs are also pointing to the right. We call the rainbow $\beta$-arcs around $z$ the \textit{bottom rainbow arcs} and the rainbow $\beta$-arcs around $w$ the \textit{top rainbow arcs}.

Recall the set of notations $S,\alpha_0,\beta_0,\beta_1,...,\beta_n$ before Definition~\ref{defn:typeIBS} (see Figure~\ref{fig:reversing}). We also denote the image of $S$ under hyperelliptic involution as $S'$, the image of $\alpha_0$ as $\alpha_0'$, and the images of $\beta_0,\beta_1,...,\beta_n$ as $\beta_0',\beta_1',...,\beta_n'$ respectively.

\begin{lem}
  The interiors of $\alpha_0$ and $\alpha_0'$ are disjoint.
  \label{lem:alp0disjnt}
\end{lem}

\begin{proof}
  Suppose not. Then some $\beta_i$ ($i\geq 1$) is connected to $\beta_0'$. But since this $\beta_i$ and $\beta_0'$ are in different directions, they are connected at the same side (left or right) of ends. This forces the ends $\beta_1$ to be on $\alpha_0'$. In particular, $\beta_1$ is then connected to rainbow arcs, giving a contradiction.
\end{proof}

\begin{lem}
  $\beta_0$ is connected to a rainbow arc at exactly one end.
  \label{lem:OneEndRainbow}
\end{lem}

\begin{proof}
  By Definition~\ref{defn:monotone} $\beta_0$ is connected to a rainbow arc. Now suppose $\beta_0$ is connected to rainbow arcs at both ends. Then since $\beta_1$ is connected to vertical arcs, these rainbow arcs must be bottom rainbow arcs pointing left. Moreover, since $\beta_0$ is inconsistent with $\beta_1,...,\beta_n$, $\beta_0$ is connected to the \textit{outermost} bottom rainbow arc pointing left at both ends. But then $\beta_0$ and this arc forms a closed curve, leading to a contradiction. Hence $\beta_0$ is connected to a $\beta$-rainbow arc at exactly one end.
\end{proof}

Now we know $\beta_0$ is connected to a rainbow arc at either the left end or the right end. The following lemma indicates that we can assume $\beta_0$ is connected to a rainbow arc at the right end, without loss of generality.

\begin{lem}
  Let $(\Sigma,\alpha,\beta,z,w)$ be a monotone (1,1) diagram with $\alpha$ in standard position, where $\beta_0$ is connected to a rainbow arc at the left end. Suppose the diagram is parametrized by $(p,q,r,s)$ using the notations in~\cite{rasmussen2005knot}. Then the (1,1) diagram parametrized by $(p,q,p-2q-r,2q-s)$, which we denote as $(\tilde{\Sigma},\tilde{\alpha},\tilde{\beta},\tilde{z},\tilde{w})$ with $\tilde{\alpha}$ in standard position, is also monotone, and $\tilde{\beta}_0$ is connected to a rainbow arc at the right end. Moreover, if the type II branched surface of $(\tilde{\Sigma},\tilde{\alpha},\tilde{\beta},\tilde{z},\tilde{w})$ associated to $\tilde{\alpha}$ fully carries a lamination, then the type II branched surface of $(\Sigma,\alpha,\beta,z,w)$ associated to $\alpha$ also fully carries a lamination.
\end{lem}

\begin{proof}
  In fact, the diagram $(p,q,p-2q-r,2q-s)$ can be obtained by taking a left-right reflection of the original diagram $(p,q,r,s)$, and then reversing the orientations of the curves $\alpha,\beta$. For example, by taking reflection of the diagram $(13,4,1,7)$ and reversing the curves we get the diagram $(13,4,4,1)$ (see Figure~\ref{fig:13_4_4_6}; we remark that here we need to further move the bottom to put the diagram in the normal form as in Figure~\ref{fig:reduced11}). In particular, the diagram $(p,q,p-2q-r,2q-s)$ is monotone, and $\tilde{\beta}_0$ is connected to a rainbow arc at the right end instead.

  \begin{figure}[!hbt]
    \begin{overpic}{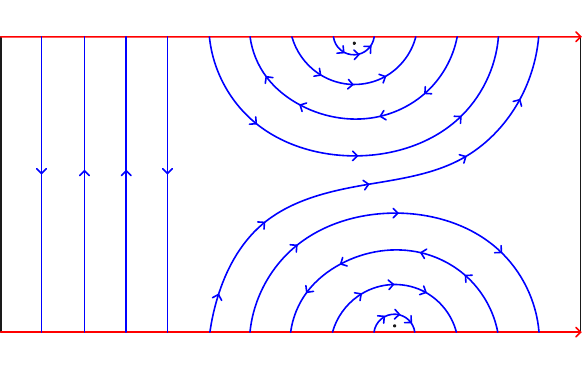}
      \put(6.5,2){9}
      \put(13,2){10}
      \put(20,2){11}
      \put(27,2){12}
      \put(34.4,2){13}
      \put(42,2){1}
      \put(49,2){2}
      \put(56.1,2){3}
      \put(63,2){4}
      \put(70,2){5}
      \put(77,2){6}
      \put(84.4,2){7}
      \put(91.5,2){8}
      \put(6.5,58.5){2}
      \put(13.6,58.5){3}
      \put(20.5,58.5){4}
      \put(27.6,58.5){5}
      \put(34.5,58.5){6}
      \put(41.5,58.5){7}
      \put(49,58.5){8}
      \put(56,58.5){9}
      \put(63,58.5){10}
      \put(70,58.5){11}
      \put(77,58.5){12}
      \put(84.2,58.5){13}
      \put(91.4,58.5){1}
      \put(98,2){\color{red} $\alpha$}
      \put(60,20){\color{blue} $\beta$}
      \put(68,6.7){\tiny $z$}
      \put(58.5,55.3){\tiny $w$}
    \end{overpic}
    \caption{The (1,1) diagram (13,4,4,1)}
    \label{fig:13_4_4_6}
  \end{figure}

  On the level of branched surfaces, we can think of the type II branched surface of $(p,q,r,s)$ (which we denote as $\mathcal{B}$) as embedded in the ambient space (recall Remarks~\ref{rems:brsfs}.$(1)$). Now the left-right reflection of the Heegaard torus $\Sigma$ would induce an orientation-reversing homeomorphism of the ambient spaces. Let $\mathcal{B'}$ be the image of $\mathcal{B}$, then $\mathcal{B'}$ can be described by the (1,1) diagram after reflection. Since the homeomorphism is orientation-reversing, here the branch directions on $\Sigma$ actually point to the right of the oriented curves instead. We can then reverse the orientations of the curves to get the usual description we used in subsection~\ref{subsec:2.3}. In particular, this coincides with the definition of the type II branched surface of the diagram $(p,q,p-2q-r,2q-s)$ (which we denote as $\tilde{\mathcal{B}}$). It follows that there is an (orientation-reversing) homeomorphism between $\mathcal{B}$ and $\tilde{\mathcal{B}}$, hence if one fully carries a lamination, then the other also does.
\end{proof}

From now on for the rest of this section, we always assume $\beta_0$ is connected to a rainbow arc at the right end. Then the local pictures of $S$ and $S'$ are depicted as in Figure~\ref{fig:SnSp_monotone}. We remark that the intermediate $\beta$-arcs may or may not exist: there might be only one $\beta$-arc in the interior of $S$ (i.e. $n=1$).

\begin{figure}[!hbt]
  \begin{overpic}[scale=0.7]{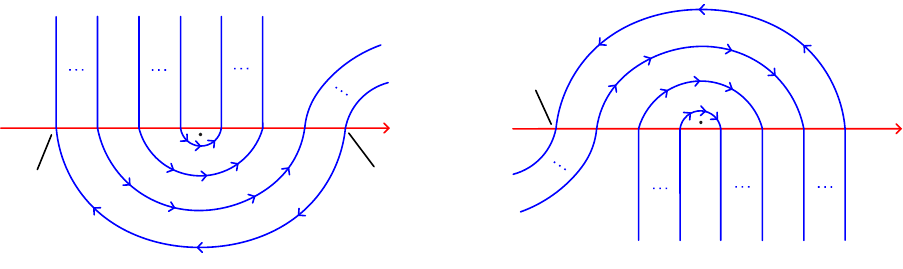}
    \put(7,3){$S$}
    \put(21,16){\tiny $w$}
    \put(30,1.7){\Small \color{blue} $\beta_0$}
    \put(24,13){\Small \color{blue} $\beta_1$}
    \put(28,6){\Small \color{blue} $\beta_n$}
    \put(63,25){\Small \color{blue} $\beta_0'$}
    \put(77,14){\tiny $z$}
    \put(90,26){$S'$}
    \put(-2,9.2){\tiny left end}
    \put(-2,7.5){\tiny of $\alpha_0$}
    \put(38.5,9){\tiny right end of $\alpha_0$}
    \put(47,21){\tiny left end of $\alpha_0'$}
    \put(44,15){\Small \color{red}$\cdot\cdot\cdot\;\alpha_1\cdot\cdot\;\cdot$}
  \end{overpic}
  \caption{Local pictures of $S$ and $S'$}
  \label{fig:SnSp_monotone}
\end{figure}

From the local picture we know the outermost bottom rainbow arc is connected to $S$ at its left end. Also the outermost top rainbow arc is connected to $S'$ at its right end. Let $\alpha_1$ be the (oriented) subarc of $\alpha$ that starts at the right end of $\alpha_0$ and ends at the left end of $\alpha_0'$, see Figure~\ref{fig:SnSp_monotone}. We remark that $\alpha_1$ could be empty (e.g. see the $(13,4,1,7)$ diagram in Figure~\ref{fig:13_4_1_7}). We also denote the (oriented) subarc of $\alpha$ that starts at the right end of $\alpha_0'$ and ends at the left end of $\alpha_0$ as $\alpha_2$. $\alpha_2$ is always non-empty. Moreover, $\alpha=\alpha_0\cup\alpha_1\cup\alpha_0'\cup\alpha_2$.

It is also clear from Figure~\ref{fig:SnSp_monotone} that the right ends of top rainbow arcs and the left ends of bottom rainbow arcs all lie in $\alpha_0\cup\alpha_1\cup\alpha_0'$. This gives us some \textit{monotone} property as one will observe in the next subsection.

\subsection{Splittings}
\label{subsec:4.2}

In this subsection we prove Proposition~\ref{prop:monotone}. The strategy is to split the branched surface by pushing the $\beta$-arcs to get a new, sink-disk-free branched surface. We first remind the readers of some terms we used in~\cite{lyu2024knot}.

\begin{figure}[!hbt]
  \begin{overpic}[scale=0.7]{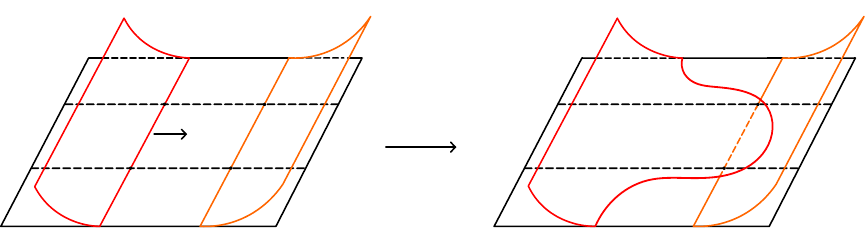}
      \put(14,11){$\color{red} \alpha$}
      \put(26,11){$\color{or} \gamma$}
      \put(83,11){$\color{or} \gamma$}
      \put(33,-1){$\Sigma$}
  \end{overpic}
  \caption{Pushing the $\alpha$-arc onto the $\gamma$-arc}
  \label{fig:pushing_arcs}
\end{figure}

One of the central tricks used in~\cite{lyu2024knot} is to model splittings of branched surfaces on 2-dimensional objects. When the branched surface is obtained by attaching sectors to a fixed surface $\Sigma$, we can describe the branch locus as curves on $\Sigma$. Then there is a special kind of splitting that gains importance in practice, which we call ``\textbf{pushing arcs}''. These splittings are along disks and are locally modelled on the middle picture of Figure~\ref{fig:splittings}. When observed on $\Sigma$, it looks like one arc (say, from curve $\alpha$) is pushed over another arc (from $\gamma$) onto the sector bounded by $\gamma$, and we say that the $\alpha$-arc is \textit{pushed onto} the $\gamma$-arc. See Figure~\ref{fig:pushing_arcs} for a 3-dimensional picture. We remark that one can also push arcs from the same curve.

\vspace{4pt}

We also developed a way to check if a branched surface is sink disk free. Recall the following definition and lemma in~\cite{lyu2024knot}:

\begin{defn}[\cite{lyu2024knot}, Definition 3.17]
  Let $\mathcal{B}$ be a branched surface and $X$ a double point in its branch locus. We say a branch sector $D$ of $\mathcal{B}$ is \textit{around $X$} if $X\in \partial D$. 
  
  We can spot all sectors around $X$ in a local picture of $X$, see Figure~\ref{fig:brsf}.$(a)$. By enumerating these sectors, we can see that there is a unique sector $D$, s.t. there exists a neighborhood $N(X)$ of $X$, where the branch locus arc $N(X)\cap \partial D$ always has its branch direction pointing into $D$. See the upper-left corner of Figure~\ref{fig:brsf}.$(a)$. We call this sector $D$ the \textbf{sink corner} of the double point $X$. We also call the double point $X$ \textbf{safe}, if its sink corner $D$ is not a sink disk.
  \label{def:sk_corner}
\end{defn}

\begin{lem}[\cite{lyu2024knot}, Lemma 3.18]
  Let $\mathcal{B}$ be a branched surface with circle boundary. If the branch locus of $\mathcal{B}$ does not contain an isolated simple circle, then a sink disk of $\mathcal{B}$ must be the sink corner of some double point.
  \label{lem:sk_corner}
\end{lem}

We are now in position to prove Proposition~\ref{prop:monotone}.

\begin{proof}[Proof of Proposition~\ref{prop:monotone}]
  Let $\mathcal{B}_{\alpha}$ be the type II branched surface associated to $\alpha$. We henceforth describe a strategy to push arcs to eliminate possible sink disks. Since our branch directions on $\Sigma-S$ always point to the left of the oriented curves, sink disks on $\Sigma-S$ are disk regions with counter-clockwise boundary. 
  
  We will first push the $\beta$-strands in a small neighborhood of $\alpha$, and then further push the $\beta$-arcs to eliminate sink disks. Recall the type II branched surface associated to $\alpha$ reversed the region $S$, hence this region needs to be fixed when pushing the arcs.

  \vspace{6pt}

  \textbf{Step 1: Push $\beta$-strands in a neighborhood of $\alpha$.} We can classify the $\beta$-strands into two classes according to their orientations. We call the strands pointing upwards the \textit{upward strands} and the strands pointing downwards the \textit{downward strands}.

  \begin{figure}[!hbt]
    \begin{overpic}{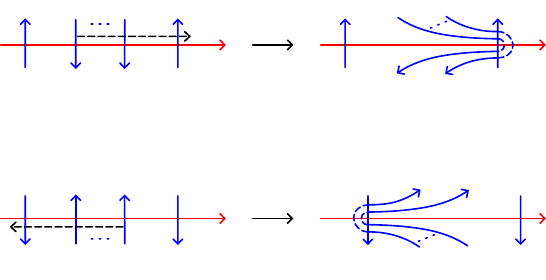}
      \put(39,39.5){\color{red}$\alpha_2$}
      \put(39,7.5){\color{red}$\alpha_1$}
      \put(30,32){$(a)$ Pushing strands on $\alpha_2$}
      \put(30,0){$(b)$ Pushing strands on $\alpha_1$}
    \end{overpic}
    \caption{Pushing arcs near $\alpha$}
    \label{fig:push_near_alpha}
  \end{figure}

  We now push the $\beta$-strands as follows: for the $\beta$-strands on $\alpha_2$, we push the downward strands to the first upward strand on the right, one by one starting from the closest, see Figure~\ref{fig:push_near_alpha}.$(a)$; for the $\beta$-strands on $\alpha_1$, we push the upward strands to the first downward strand on the left, one by one starting from the closest, see Figure~\ref{fig:push_near_alpha}.$(b)$. We remark that these operations are well defined (when $\alpha_1$ is empty, we only push the strands on $\alpha_2$), since the $\beta$-strand at the left end of $\alpha_0$ is upwards (which is the right end of $\alpha_2$), and the $\beta$-strand at the right end of $\alpha_0$ is downwards (which is the left end of $\alpha_1$ when non-empty), see Figure~\ref{fig:SnSp_monotone}. In particular we will not move $\beta_0$.
  
  Finally, we can further push the $\beta$-strands in the interior of $\alpha_0'$ to the nearest end, one by one from the closest. That is, we push the upward strands in the interior to the downward strand at the left endpoint of $\alpha_0'$, and the downward strands in the interior to the upward strand at the right endpoint. We keep the strands on $\alpha_0$ fixed, where the branch direction actually points to the $\alpha$-disk (recall Figure~\ref{fig:reversing}).

  We remark that pushing each strand would create two new double points. The sink corners of these new double points are all sectors on the torus $\Sigma$ (see Figure~\ref{fig:push_near_alpha}).

  \vspace{6pt}

  \textbf{Step 2: Further push $\beta$ to eliminate sink disks.} We now describe how we could further push $\beta$. Let $F$ be the set of $\beta$-arcs containing rainbow $\beta$-arcs pointing left, and vertical $\beta$-arcs pointing upwards that are disjoint from the interior of $\alpha_0'$. 
  
  We first claim that arcs in $F$ are not moved in Step 1. Recall Figure~\ref{fig:SnSp_monotone}, where we see the left ends of bottom rainbow arcs always lie in $\alpha_0\cup\alpha_1\cup\alpha_0'$. On the other hand, the right ends of bottom rainbow arcs are always on $\alpha_2$. For a bottom rainbow arc pointing left, we know its left end is pointing downwards and in $\alpha_1$, while its right end is pointing upwards and in $\alpha_2$; these ends are then never moved in Step 1. Similar things happen for top rainbow arcs pointing left, where the left ends lie in $\alpha_2$ and right ends lie in $\alpha_1$. Moreover, ends of vertical arcs always lie in $\alpha_0'\cup\alpha_2\cup\alpha_0$, hence upward arcs are never moved, unless they intersect $\alpha_0'$.

  In particular, there is at least one vertical arc in $F$, see the vertical $\beta$-arc connected to the right end of $\alpha_0'$ in Figure~\ref{fig:SnSp_monotone}. We also observe that $\beta_0$ and $\beta_0'$ are in $F$. Now these fixed $\beta$-arcs in $F$ and the $\alpha$-curve separate $\Sigma$ into several regions. Again, most of these regions are quadrilaterals. Besides these, there are also two bigons $S$ and $S'$, and two hexagons or an octagon. The idea is to push arcs within each region. 

  \begin{figure}[!hbt]
    \begin{overpic}[scale=0.8]{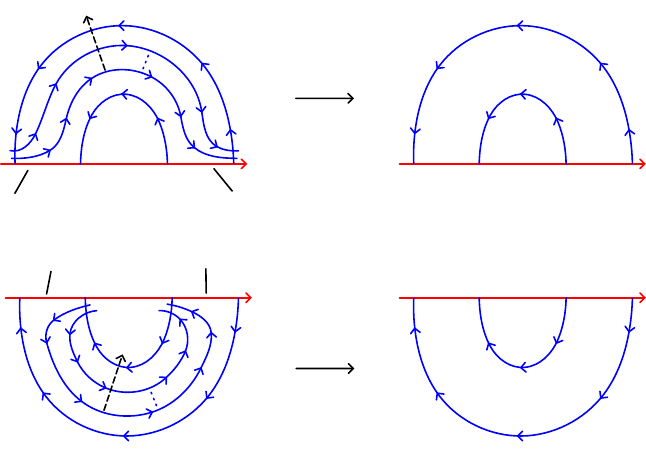}
      \put(0,40){\color{red}$\alpha_1$}
      \put(35,40){\color{red}$\alpha_2$}
      \put(6,33){\color{red}$\alpha_2$}
      \put(30,33){\color{red}$\alpha_1$}
    \end{overpic}
    \caption{Pushing arcs in ``rainbow'' quadrilaterals}
    \label{fig:push_quad_rb}
  \end{figure}

  For each quadrilateral region, we first notice that back on the original diagram, the $\beta$-arcs it contains must be of the same kind. If these $\beta$-arcs are rainbow arcs, then since we picked all rainbow arcs pointing left, we know the $\beta$-arcs in the interior are rainbow arcs pointing right. Moreover, according to Figure~\ref{fig:SnSp_monotone}, the two boundary $\alpha$-arcs of this region are disjoint from the interiors of $\alpha_0$ and $\alpha_0'$, hence are subarcs of $\alpha_1$ and $\alpha_2$ respectively. It follows that the interior $\beta$-arcs all have both ends pushed in Step 1. We can then further push these interior arcs one by one onto the boundary arc, as in Figure~\ref{fig:push_quad_rb}. Notice these operations will eliminate double points without creating new ones.

  \begin{figure}[!hbt]
    \begin{overpic}[scale=0.8]{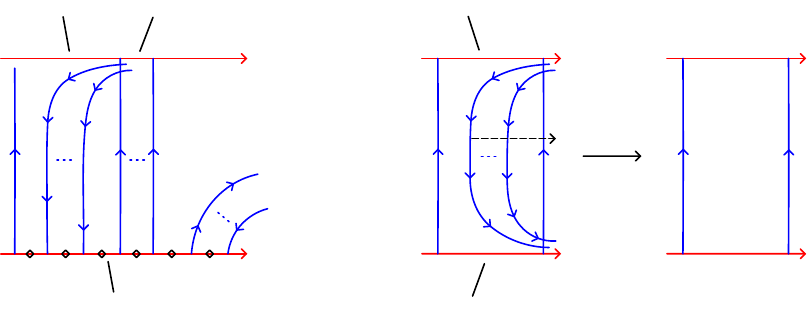}
      \put(13,0){\color{red}$\alpha_0$}
      \put(2,38){{\color{red}$\alpha_2$} or {\color{red}$\alpha_0'$}}
      \put(18,38){\color{red}$\alpha_2$}
      \put(55,38){\color{red}$\alpha_2$}
      \put(55,0){\color{red}$\alpha_2$}
    \end{overpic}
    \caption{Pushing arcs in ``vertical'' quadrilaterals}
    \label{fig:push_quad_vt}
  \end{figure}

  If the quadrilateral region contains vertical $\beta$-arcs, then we further check if it intersects the interiors of $\alpha_0,\alpha_0'$. The regions intersecting the interior of $\alpha_0'$ is the bigon $S'$ on one side, and a polygon on the other side (see Figure~\ref{fig:SnSp_monotone}, recall vertical arcs intersecting the interior of $\alpha_0'$ are not in $F$); in particular they are not quadrilaterals. There might be several quadrilateral regions intersecting the interior of $\alpha_0$ (this happens when there are upward vertical arcs starting from $\alpha_0$ and arriving at $\alpha_2$, see Figure~\ref{fig:push_quad_vt} left; we note that some of these vertical arcs may as well arrive at $\alpha_0'$ instead, in which case these vertical arcs are in some hexagon or octagon, see Figure~\ref{fig:13_4_1_7} for an example). We keep the quadrilateral regions intersecting the interior of $\alpha_0$ unchanged, since there are no sink disks: the branch direction points to the $\alpha$-disk along $\alpha_0$, see Figure~\ref{fig:push_quad_vt} left. For the other quadrilaterals, the boundary $\alpha$-arcs are subarcs of $\alpha_2$. Hence all downward vertical arcs in the interior have both ends pushed in Step 1. We can then further push these arcs one by one onto the boundary arc, see Figure~\ref{fig:push_quad_vt} right. Again, these operations will not create new double points. 

  The reversed bigon region $S$ should be fixed. For the bigon region $S'$, we can further push the interior arcs one by one onto the boundary arc $\beta_0'$, similar to pushing a rainbow quadrilateral. The complexity lies in the hexagons or octagon. These regions necessarily intersect $\alpha_0$ or $\alpha_0'$, and the interior $\beta$-arcs can be further pushed as in Figure~\ref{fig:push_poly}. We remark that here the boundary $\beta$-arcs of the polygon regions are colored in purple instead. Moreover, the blue $\beta$-arcs in the interior of the polygon regions should be understood as with non-negative multiplicities. That is, there could be just no such arcs, or there could be multiple parallel such arcs. We can always push these arcs one by one, whenever they exist.

  \begin{figure}[!hbt]
    \begin{overpic}[scale=0.4]{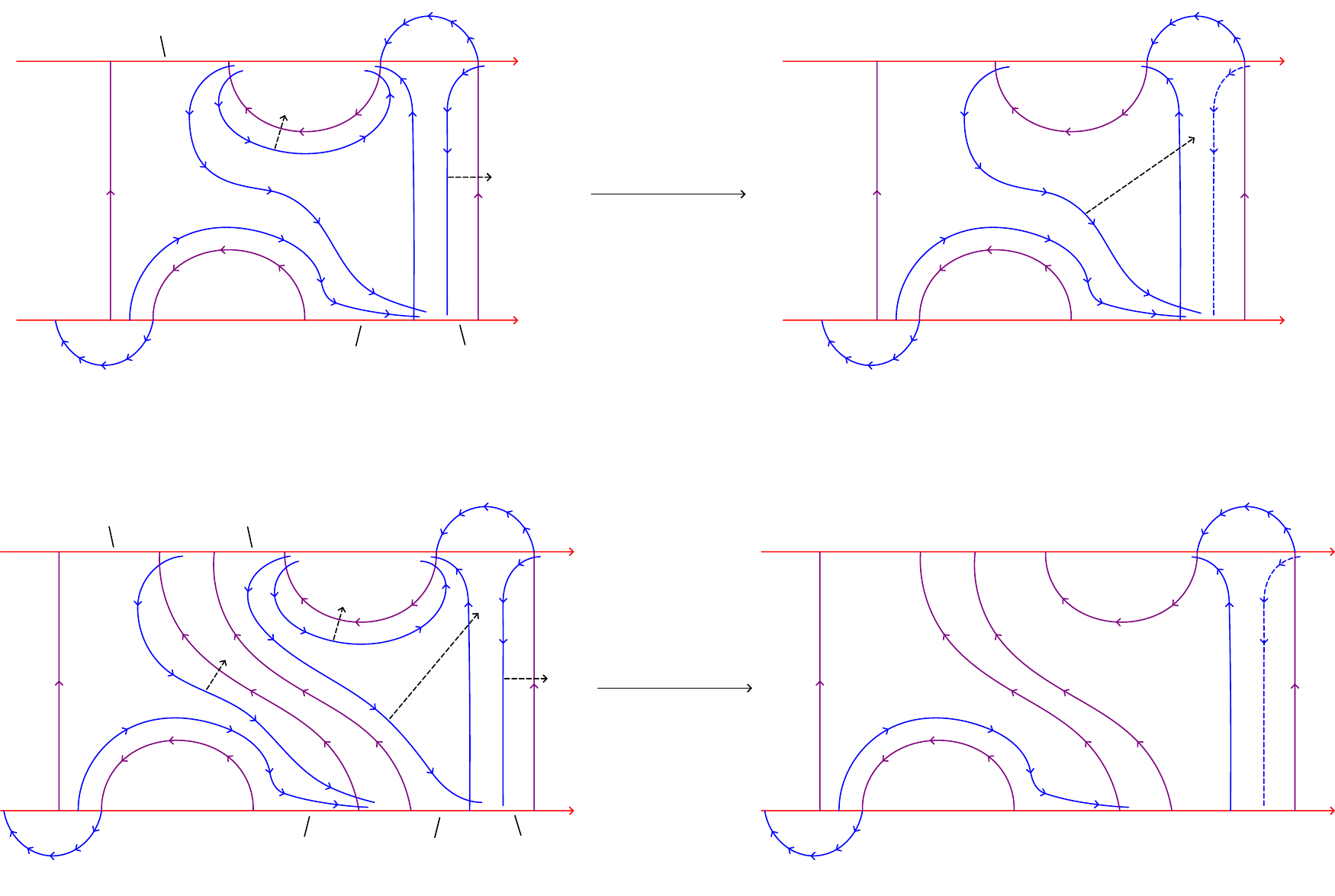}
      \put(40,35){\small $(a)$ One octagon}
      \put(40,0){\small $(b)$ Two hexagons}
      \put(31.4,63.2){\tiny$S'$}
      \put(7,41){\tiny$S$}
      \put(10,65){\tiny\color{red}$\alpha_2$}
      \put(25,40){\tiny\color{red}$\alpha_2$}
      \put(33,40){\tiny{\color{red}$\alpha_2$} or {\color{red}$\alpha_0$}}
      \put(7,28.5){\tiny\color{red}$\alpha_2$}
      \put(17,28.5){\tiny\color{red}$\alpha_2$}
      \put(35.6,26.4){\tiny$S'$}
      \put(3,4){\tiny$S$}
      \put(21,3){\tiny\color{red}$\alpha_2$}
      \put(31.3,3){\tiny\color{red}$\alpha_2$}
      \put(37,3){\tiny{\color{red}$\alpha_2$} or {\color{red}$\alpha_0$}}
    \end{overpic}
    \caption{Pushing arcs in the octagon or hexagons}
    \label{fig:push_poly}
  \end{figure}

  We here explain what happens in the octagon case in details. The two boundary rainbow $\beta$-arcs of the octagon are the outermost rainbow arcs pointing left, which are connected to $\beta_0$ and $\beta_0'$ respectively. We can then spot the boundary vertical $\beta$-arcs of the octagon in the local pictures of $S,S'$ as in Figure~\ref{fig:SnSp_monotone}. In particular the boundary vertical arc on the right is actually connected to $\beta_0'$ (recall $\beta$-arcs intersecting the interior of $\alpha_0'$ are not in $F$), while the boundary vertical arc on the left is the right-most vertical arc connected to $\alpha_0$ (which is necessarily pointing upwards). It follows that the octagon actually looks like Figure~\ref{fig:push_poly}.$(a)$. The octagon has four boundary $\alpha$-arcs. The bottom-left one is part of $\alpha_0$, and the top-right one is $\alpha_0'$. It follows that the top-left one must be part of $\alpha_2$, while the bottom-right one contains part of $\alpha_2$, and may also contain part of $\alpha_0$ at the bottom-right corner. 
  
  The $\beta$-strands on the boundary $\alpha$-arcs of the octagon are pushed in Step 1, as depicted in Figure~\ref{fig:push_poly}.$(a)$ left, where at the bottom right corner, the downward strands are fixed if the $\alpha$-arc is from $\alpha_0$, and is pushed onto the upward strand on the right if the $\alpha$-arc is from $\alpha_2$.

  We can then first push the top rainbow arcs in the interior onto the boundary top rainbow arc, one by one from the closest, if there are any. If the $\alpha$-arc at the bottom-right corner is from $\alpha_2$, we also push the downward vertical arcs connected to $S'$ to the right boundary vertical arc of the octagon, one by one from the closest (if the $\alpha$-arc at the bottom-right corner is from $\alpha_0$, then we fix these downward vertical arcs). Then the octagon would look like Figure~\ref{fig:push_poly}.$(a)$ right. So far these operations do not create new double points.
  
  If there are ``leaning'' vertical arcs pointing downwards, then there are still some sink disks on the torus (recall here sink disks are disk regions with counter-clockwise boundary), as shown in Figure~\ref{fig:push_poly}.$(a)$ right (see also the central region of Figure~\ref{fig:push_over_dbpt} left). We can then push these downward ``leaning'' vertical arcs onto the $\beta$-arcs to the right, one by one from the closest. Notice we have to push over a double point here. Each time we push an arc over the double point, a new double point is created away from the torus $\Sigma$, see Figure~\ref{fig:push_over_dbpt}, where the two bottom-left arcs are pushed over the top-right double point. The arrows in Figure~\ref{fig:push_over_dbpt} indicate the branch directions; the double points away from $\Sigma$ are marked in black dots.

  \begin{figure}[!hbt]
    \begin{overpic}[scale=0.4]{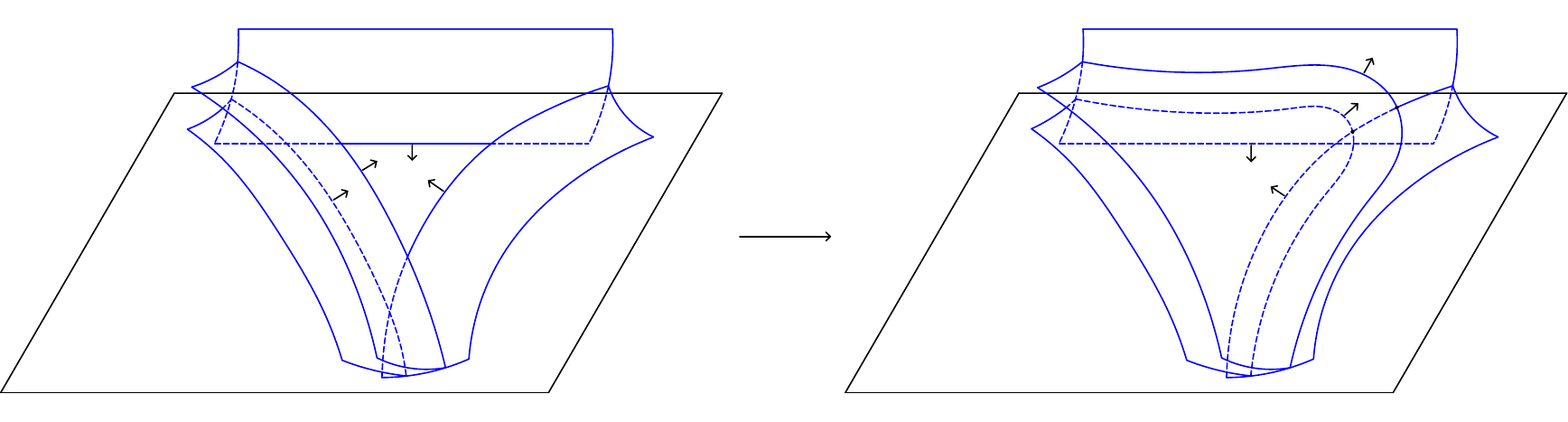}
      \put(4,3){$\Sigma$}
    \end{overpic}
    \caption{Pushing arcs over a double point}
    \label{fig:push_over_dbpt}
  \end{figure}

  The good news is that the sink corners of these double points away from $\Sigma$ are never sink disks. In fact, as can be seen in Figure~\ref{fig:push_over_dbpt} right, the sink corner of such a double point is either a sector whose boundary branch direction points out at the preceding pushed arcs, or is the sector containing the center of the $\beta$-disk (we can assume our splittings only take place at a small neighborhood of the boundary of the $\beta$-disk, that is, the splitting disks never intersect the $I$-fiber at the center of the $\beta$-disk). The latter case happens when there are ``leaning'' vertical arcs pointing downwards, and the outermost top rainbow arc is pointing left. The sector containing the center of the $\beta$-disk is never a sink disk, since its boundary contains $\beta_1$ (which is in the interior of $S$, thus never pushed or pushed over), along which the branch direction is pointing out.

  The two hexagons case is similar. The only difference is that now we have ``leaning'' vertical arcs pointing upwards, see Figure~\ref{fig:push_poly}.$(b)$. These arcs are necessarily in $F$. As a consequence, we only need to push some of the downward ``leaning'' arcs over double points, and others onto these upward ``leaning'' arcs instead.

  We henceforth check that after these splittings our branched surface is sink disk free. Call the new branched surface $\mathcal{B}_{\alpha}^{\dagger}$. Since our splittings are all modelled on the middle picture of Figure~\ref{fig:splittings}, they do not change the topology of $\partial_v N(\mathcal{B}_{\alpha})$. In particular the branch locus of $\mathcal{B}_{\alpha}^{\dagger}$ still consists of 4 immersed curves. One can verify that these curves do intersect at $S$ (which is fixed by the splittings). Hence by Lemma~\ref{lem:sk_corner} we only need to check that sink corners of double points are never sink disks.

  Recall in Step 1 our splittings created new double points on the torus $\Sigma$, while in Step 2 our splittings created new double points only when we are pushing the ``leaning'' arcs over a double point. The new double points in Step 2 are safe by our discussion above. Recall the sink corners of the new double points created in Step 1 are all on the torus $\Sigma$. In Step 2 these double points could vanish, or could remain unchanged locally, or could be pushed over as in Figure~\ref{fig:push_over_dbpt}. In any case the possible sink corners are still on the torus $\Sigma$. However, after our splittings there are no sink disks on $\Sigma$ (one could check this by looking at each region cut out by $F\cup \{\alpha\}$, see Figures~\ref{fig:push_quad_rb}~\ref{fig:push_quad_vt}~\ref{fig:push_poly}). Hence these double points are all safe. 
  
  All other double points of $\mathcal{B}_{\alpha}^{\dagger}$ come from double points of $\mathcal{B}_{\alpha}$. For those double points coming from $\mathcal{B}_{\alpha}$ and disjoint from $S$, their sink corners are also on the torus $\Sigma$. Hence these double points are also safe.

  It remains for us to check the double points in $S$. In particular they are all on $\alpha_0$, see Figure~\ref{fig:reversing}. The sink corner of the double points in the interior $\alpha_0$ is the $\alpha$-disk (see the top-right local picture of Figure~\ref{fig:reversing}), which is not a sink disk since along $(\alpha-\alpha_0)$ the branch direction points out. The sink corners of the double points at the two ends of $\alpha_0$ are again on $\Sigma$, see the top-left local picture of Figure~\ref{fig:reversing}. It follows that all these double points are safe.

  Now that we have checked all double points of $\mathcal{B}_{\alpha}^{\dagger}$ to be safe, by Lemma~\ref{lem:sk_corner} $\mathcal{B}_{\alpha}^{\dagger}$ is sink disk free. Again, since our splittings are all modelled on the middle picture of Figure~\ref{fig:splittings}, they do not change the topology of $\partial_h N(\mathcal{B}_{\alpha})$. It follows that $\partial_h N(\mathcal{B}_{\alpha}^{\dagger})$ still consists of some annuli, as described in Proposition~\ref{lem:vB}. Moreover, $\mathcal{B}_{\alpha}^{\dagger}$ cannot carry any torus, for otherwise this torus would be carried by the co-oriented branched surface $\mathcal{B}_{\alpha}$, which then becomes a non-separating torus in the ambient rational homology sphere (via the embedding described in Remarks~\ref{rems:brsfs}.$(1)$). Now we can apply Lemma~\ref{lem:sk_disk_free} and conclude that $\mathcal{B}_{\alpha}^{\dagger}$ fully carries a lamination. So does $\mathcal{B}_{\alpha}$, since $\mathcal{B}_{\alpha}^{\dagger}$ is obtained by splitting $\mathcal{B}_{\alpha}$.
\end{proof}

\section{Strongly incoherent diagrams}
\label{sec:5}

In this section we prove Proposition~\ref{prop:strongly_incoherent}.

\subsection{Outlines}
\label{subsec:5.1}

In this subsection we outline the proof of Proposition~\ref{prop:strongly_incoherent}. We will put down some lemmas but defer their proofs to later subsections. In particular, a proof of Proposition~\ref{prop:strongly_incoherent} using these lemmas will be given at the end of this subsection.

Suppose now we have a reduced, $\alpha$-strongly incoherent diagram $(\Sigma,\alpha,\beta,z,w)$. We put $\alpha$ in standard position and let $\gamma,\delta$ be the carrying curves of $\alpha,\beta$ respectively.

\vspace{4pt}

\textbf{Case 1: The diagram $(\Sigma,\alpha,\delta,z,w)$ is non-simple.} 

In this case we can make use of the branched surface we constructed in~\cite{lyu2024knot}. In particular, the following lemma will be proved in subsection~\ref{subsec:5.4}. Recall the notation of $S_0$ in subsection~\ref{subsec:2.3}; see also Figure~\ref{fig:reversing}.

\begin{lem}
  Let $(\Sigma,\alpha,\beta,z,w)$ be a reduced, $\alpha$-strongly incoherent (1,1) diagram, and $\delta$ be the carrying curve of $\beta$. Suppose $(\Sigma,\alpha,\delta,z,w)$ is non-simple. Then the type I modified Heegaard branched surface of $(\Sigma,\alpha,\beta,z,w)$ associated to $S_0$ fully carries a lamination.
  \label{lem:StrongIncoherentReduction}
\end{lem}

\textbf{Case 2: The diagram $(\Sigma,\alpha,\delta,z,w)$ is simple, but the diagram $(\Sigma,\alpha,\beta,z,w)$ is also $\beta$-strongly incoherent, and the diagram $(\Sigma,\gamma,\beta,z,w)$ is non-simple.}

In this case we can still apply Lemma~\ref{lem:StrongIncoherentReduction}. In fact, we can interchange $\alpha$ and $\beta$ in the lemma to get a new, $\beta$-strongly incoherent version. The reduction would then become $(\Sigma,\alpha,\beta,z,w)\rightarrow (\Sigma,\gamma,\beta,z,w)$, and the type I branched surface fully carrying laminations would reverse a possibly different sector, which could be spotted when we put $\beta$ in standard position instead.

\vspace{6pt}

\textbf{Case 3: Both $(\Sigma,\alpha,\delta,z,w)$ and $(\Sigma,\gamma,\beta,z,w)$ are simple.}

In this case the diagram is actually primitive. We recall the following definition and propositions in~\cite{lyu2024knot}:

\begin{defn}[\cite{lyu2024knot}, Definition 3.2]
  Let $(\Sigma,\alpha,\beta,z,w)$ be a reduced (1,1)-diagram. The diagram is called \textit{primitive} if  
  \begin{enumerate}[(i)]
      \item any quadrilateral has ($\beta$-)parallel $\alpha$-arcs if and only if it has ($\alpha$-)parallel $\beta$-arcs, and
      \item if there are two hexagons, then for each hexagon the two non-parallel boundary arcs (one $\alpha$-arc and one $\beta$-arc) are next to each other.
  \end{enumerate}
  \label{def:primitive}
\end{defn}

\begin{prop}[\cite{lyu2024knot}, Proposition 3.3]
  Let $(\Sigma,\alpha,\beta,z,w)$ be a reduced (1,1)-diagram where $\alpha$ is placed in standard position. Then the diagram is primitive if and only if
  \begin{enumerate}[(i)]
      \item its rainbow arcs are in alternating direction, and
      \item its vertical arcs are in the same direction.
  \end{enumerate}
  \label{prop:primitive_standard}
\end{prop}

\begin{prop}[\cite{lyu2024knot}, Proposition 3.9]
  If a non-simple reduced (1,1)-diagram $(\Sigma,\alpha,\beta,z,w)$ is not primitive, with the curves $\gamma,\delta$ carrying $\alpha,\beta$ respectively, then either $(\Sigma,\gamma,\beta,z,w)$ or $(\Sigma,\alpha,\delta,z,w)$ is non-simple.
  \label{prop:primitive_simple}
\end{prop}

We will discuss this case in subsection~\ref{subsec:5.3}, where we prove the following lemma:

\begin{lem}
  Let $(\Sigma,\alpha,\beta,z,w)$ be a primitive, incoherent (1,1) diagram. Then the type I modified Heegaard branched surface of $(\Sigma,\alpha,\beta,z,w)$ associated to $S_0$ fully carries a lamination.
  \label{lem:primitive_lami}
\end{lem}

\textbf{Case 4: The diagram $(\Sigma,\alpha,\beta,z,w)$ is not $\beta$-strongly incoherent.}

Here the $\beta$-arcs connected to $\beta_1$ must be vertical. It follows that the diagram is either $\alpha$-monotone, or all $\beta$-arcs connected to $S$ are vertical. The former case is dealt with in subsection~\ref{subsec:5.2}, while for the latter case we can do reductions and make use of Lemma~\ref{lem:typeIreds}. In particular we will prove the following lemma in subsection~\ref{subsec:5.2}:

\begin{lem}
  Let $(\Sigma,\alpha,\beta,z,w)$ be an $\alpha$-monotone, $\alpha$-strongly incoherent (1,1) diagram. Then the type I modified Heegaard branched surface of $(\Sigma,\alpha,\beta,z,w)$ associated to $S_0$ fully carries a lamination.
  \label{lem:monotone-strongly_incoherent}
\end{lem}

We now give a proof of Proposition~\ref{prop:strongly_incoherent}, modulo the Lemmas~\ref{lem:StrongIncoherentReduction}~\ref{lem:primitive_lami}~\ref{lem:monotone-strongly_incoherent} above. The reduction argument used here is similar to that in the proof of Proposition~\ref{prop:lami!} in subsection~\ref{subsec:3.5}.

\begin{proof}[Proof of Proposition~\ref{prop:strongly_incoherent}]
  Suppose first the diagram is both $\alpha$- and $\beta$- strongly incoherent. If one of $(\Sigma,\alpha,\delta,z,w)$ or $(\Sigma,\gamma,\beta,z,w)$ is non-simple, then we can apply Lemma~\ref{lem:StrongIncoherentReduction} to find a type I branched surface that fully carries a lamination. If both $(\Sigma,\alpha,\delta,z,w)$ and $(\Sigma,\gamma,\beta,z,w)$ are simple, then by Proposition~\ref{prop:primitive_simple} the diagram is primitive, and by Lemma~\ref{lem:primitive_lami} there is a type I branched surface that fully carries a lamination.

  Now suppose the diagram is $\alpha$-strongly incoherent but not $\beta$-strongly incoherent. Place $\alpha$ in standard position and recall the notations in subsection~\ref{subsec:2.3}. Then since the diagram is not $\beta$-strongly incoherent, the $\beta$-arcs connected to $\beta_1$ are vertical arcs. If the diagram is $\alpha$-monotone, then by Lemma~\ref{lem:monotone-strongly_incoherent} we can find a type I branched surface that fully carries a lamination.

  Now suppose moreover the diagram is not $\alpha$-monotone. Then all $\beta$-arcs connected to $S$ are vertical arcs. Consider $\gamma$ the carrying curve of $\alpha$. By Proposition~\ref{prop:carrying} $\gamma$ would necessarily intersect these vertical $\beta$-arcs connected to $S$. In particular, the diagram $(\Sigma,\gamma,\beta,z,w)$ must be non-simple and $\gamma$-strongly incoherent, and the reduction $(\Sigma,\alpha,\beta,z,w)\rightarrow (\Sigma,\gamma,\beta,z,w)$ is of type I.

  We can now put $\gamma=\gamma_1$ in standard position and think of the $\beta$-arcs connected to our new ``$S$'' in the $(\Sigma,\gamma_1,\beta,z,w)$ diagram. If all the $\beta$-arcs are still vertical, then we can take $\gamma_2$ the carrying curve of $\gamma_1$, and by a similar argument, $(\Sigma,\gamma_2,\beta,z,w)$ is $\gamma_2$-strongly incoherent, and the reduction $(\Sigma,\gamma_1,\beta,z,w)\rightarrow (\Sigma,\gamma_2,\beta,z,w)$ is of type I. We can repeat this procedure whenever possible. Again, by Lemma~\ref{lem:simpler} this procedure must stop after finitely many times. 
  
  Suppose it stops at some diagram $(\Sigma,\gamma_n,\beta,z,w)$. In particular this diagram is $\gamma_n$-strongly incoherent. Now the two $\beta$-arcs connected to the $\gamma_n$-sink bigon are either both rainbow arcs or both vertical arcs. In the former case the diagram is also $\beta$-strongly incoherent, and by our previous discussion there is some type I branched surface of this diagram that fully carries a lamination. In the latter case the diagram is necessarily $\gamma_n$-monotone, and by Lemma~\ref{lem:monotone-strongly_incoherent} we can still find a type I branched surface that fully carries a lamination. Since our reductions are all of type I, by Lemma~\ref{lem:typeIreds} we can go backwards along the reductions and find a type I branched surface of $(\Sigma,\alpha,\beta,z,w)$ that fully carries a lamination. This concludes the proof.
\end{proof}

\subsection{Monotone, strongly incoherent diagrams}
\label{subsec:5.2}

In this subsection we prove Lemma~\ref{lem:monotone-strongly_incoherent}. The key observation here is that for an $\alpha$-strongly incoherent diagram, the type I branched surface associated to $S_0$ and the type II branched surface associated to $\alpha$ do not essentially differ. We state this in a more general way, which could also be used in the following subsections.

\begin{figure}[!hbt]
  \begin{overpic}[scale=0.6]{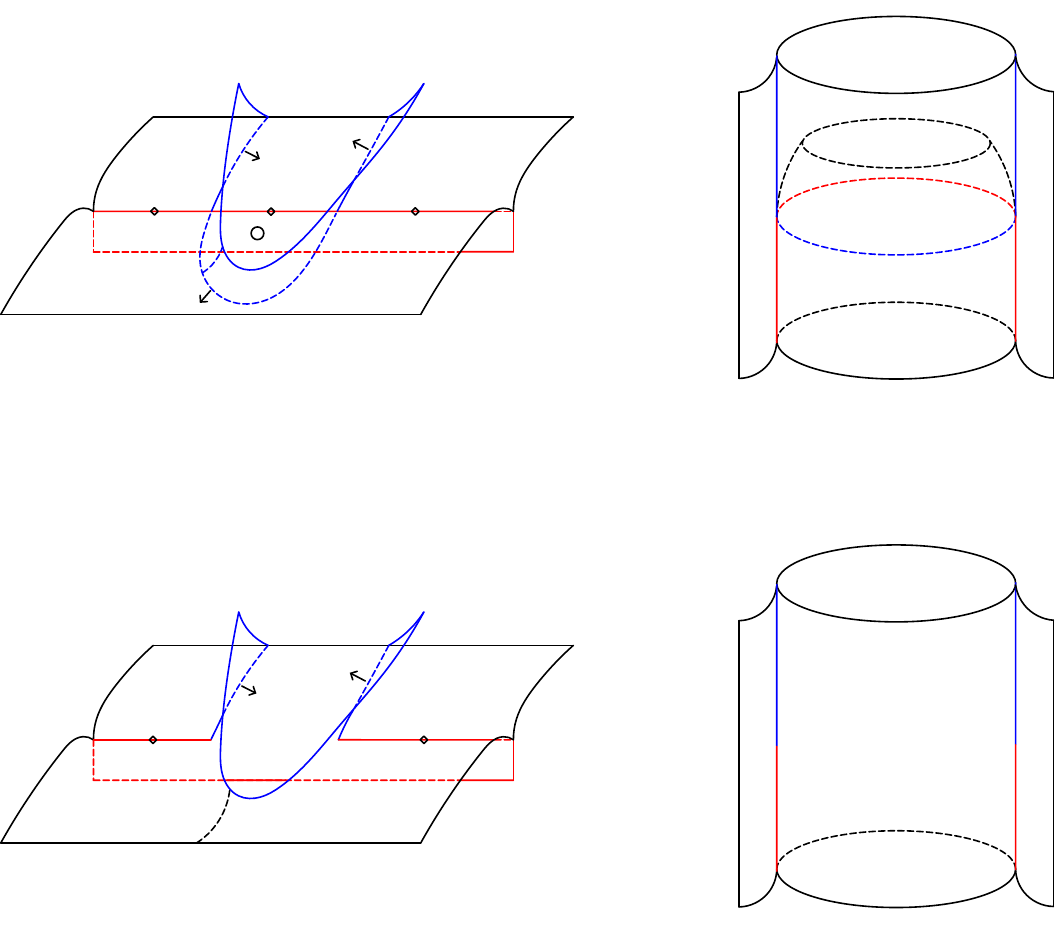}
    \put(25.5,66.6){\tiny $\partial_w$}
    \put(83,76){\tiny $\partial_w$}
    \put(28,48){$(a)$ Local pictures of $\mathcal{B}_{\alpha}$ near $w$}
    \put(28,-2){$(b)$ Local pictures of $\mathcal{B}_{\alpha}'$ near $w$}
  \end{overpic}
  \caption{Removing punctured bigon from $\mathcal{B}_{\alpha}$ to get $\mathcal{B}_{\alpha}'$}
  \label{fig:remove_ann_I}
\end{figure}

Let $(\Sigma,\alpha,\beta,z,w)$ be a reduced, incoherent (1,1) diagram. Let $\mathcal{B}_{\alpha}$ be the type II branched surface associated to $\alpha$. We can then look at the local picture of $\mathcal{B}_{\alpha}$ near the $\alpha$-source bigon of the diagram (see Figure~\ref{fig:reversing} or Figure~\ref{fig:remove_ann_I}.$(a)$). When we think of the branch locus of $\mathcal{B}_{\alpha}$ as 4 immersed curves, then the boundary of this (punctured) bigon is one entire immersed curve. Let $\mathcal{B}_{\alpha}'$ be the branched surface obtained by removing this punctured bigon (which is a source sector). It is then clear from the construction that $N(\mathcal{B}_{\alpha}')$ is homeomorphic to the knot complement with $\partial_v N(\mathcal{B}_{\alpha}')$ consisting of (four) meridional annuli. In fact, $\mathcal{B}_{\alpha}$ can be obtained from $\mathcal{B}_{\alpha}'$ by attaching an annulus to a ``meridian'' on the horizontal boundary, see Figure~\ref{fig:remove_ann_I}.

\begin{figure}[!hbt]
  \begin{overpic}[scale=0.6]{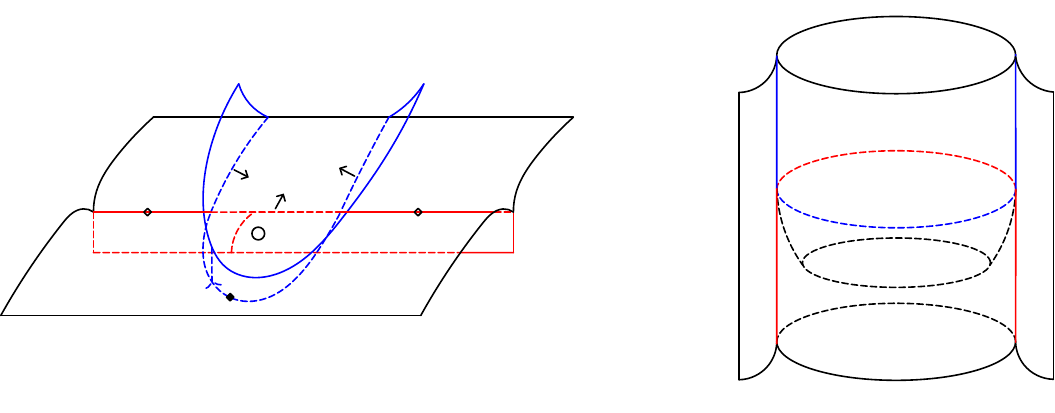}
    \put(25.5,16.6){\tiny $\partial_w$}
    \put(83,13){\tiny $\partial_w$}
  \end{overpic}
  \caption{Local pictures of $\mathcal{B}_{S_0}$ near $w$}
  \label{fig:remove_ann_II}
\end{figure}

If $(\Sigma,\alpha,\beta,z,w)$ is moreover $\alpha$-strongly incoherent, then similar things happen for $\mathcal{B}_{S_0}$ the type I branched surface reversing $S_0$. One can check that in this case the boundary of the $\alpha$-source bigon is also an entire immersed curve, see Figure~\ref{fig:remove_ann_II}. Moreover, $\mathcal{B}_{S_0}$ can also be obtained by attaching an annulus to $\mathcal{B}_{\alpha}'$ along a ``meridian'' on $\partial_h N(\mathcal{B}_{\alpha}')$, where we smooth the annulus in a different direction, see Figures~\ref{fig:remove_ann_I}~\ref{fig:remove_ann_II}.

\begin{lem}
  Let $\mathcal{B}$ be a co-oriented branched surface with circle boundary. Let $A$ be an annulus component of $\partial_h N(\mathcal{B})$. Let $\delta$ be a core curve of $A$. We can construct a branched surface $\mathcal{B'}$ by attaching an annulus to $\mathcal{B}$, such that one boundary component of the annulus is identified with $\delta$ (we smooth the annulus in any direction). Suppose $\mathcal{B'}$ has no disk of contact. Then $\mathcal{B}$ fully carries a lamination if and only if $\mathcal{B'}$ does.
  \label{lem:remove_ann}
\end{lem}

\begin{proof}
  Suppose first that $\mathcal{B}$ fully carries a lamination $\mathcal{L}$. By standard operations in the lamination theory (e.g. see~\cite{li2022taut}, discussions before Lemma 5.1), we may assume $\partial_h N(\mathcal{B})\subset \mathcal{L}$, and moreover no two cusp components of $\partial_v N(\mathcal{B})$ belong to the same $I$-bundle component of $N(\mathcal{B})-\mathcal{L}$. Let $A_1,A_2$ be the annulus components of $\partial_v N(\mathcal{B})$ next to $A$, where the branch direction along $\delta$ points to $A_1$.

  \begin{figure}[!hbt]
    \begin{overpic}{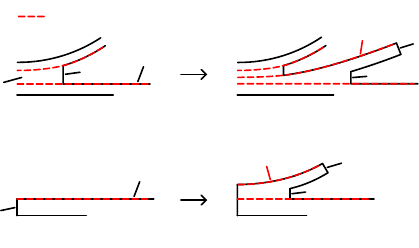}
      \put(12.3,54){\small laminations}
      \put(-13.6,38){\small $F_1\times I$}
      \put(19.3,41){\small $A_1$}
      \put(33,44){\small $A$}
      \put(14,32){\small $N(\mathcal{B})$}
      \put(78,51){\small new leaf}
      \put(99,47){\small $\delta$-boundary}
      \put(89,40.5){\small $\delta$-cusp}
      \put(72,32){\small $N(\mathcal{B'})$}
      \put(33.5,27){\small $(a)$ $A_1$ is a cusp}
      \put(-5,7){\small $A_1$}
      \put(32,16.7){\small $A$}
      \put(14,3){\small $N(\mathcal{B})$}
      \put(50,20.6){\small new leaf}
      \put(82,19){\small $\delta$-boundary}
      \put(74,12.7){\small $\delta$-cusp}
      \put(72,3){\small $N(\mathcal{B'})$}
      \put(30,-2){\small $(b)$ $A_1$ is a boundary}
    \end{overpic}
    \caption{Adding an annulus to $\mathcal{B}$}
    \label{fig:ann_lem_1}
  \end{figure}
  
  \textbf{Case 1: $A_1$ is a cusp.} Let $F_1\times I$ be the $I$-bundle component of $N(\mathcal{B})-\mathcal{L}$ whose boundary contains $A_1$. We can then take the annulus component of $\partial_h N(\mathcal{B'})$ bounded by $A_1$ and the $\delta$-boundary, and attach it to $F_1\times \frac{1}{2}$ along $A_1$; we then add this leaf to $\mathcal{L}$, see Figure~\ref{fig:ann_lem_1}.$(a)$. The new lamination is isotopic to a lamination fully carried by $\mathcal{B'}$. 

  \textbf{Case 2: $A_1$ is a boundary.} In this case we can simply take the annulus component of $\partial_h N(\mathcal{B'})$ bounded by $A_1$ and the $\delta$-boundary, and add this annulus to $\mathcal{L}$, see Figure~\ref{fig:ann_lem_1}.$(b)$. The new lamination is isotopic to one fully carried by $\mathcal{B'}$.

  \vspace{6pt}

  Now suppose $\mathcal{B'}$ fully carries a lamination $\mathcal{L'}$. Again, we may assume that $\partial_h N(\mathcal{B})\subset \mathcal{L'}$, and that no two cusp components of $\partial_v N(\mathcal{B'})$ belong to the same $I$-bundle component of $N(\mathcal{B'})-\mathcal{L'}$. Let $A_2'$ be the annulus component of $\partial_v N(\mathcal{B'})$ corresponding to $A_2$.

  \begin{figure}[!hbt]
    \begin{overpic}{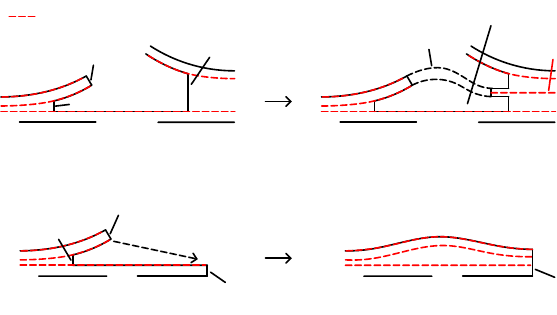}
      \put(7.5,52.7){\small laminations}
      \put(6,45.6){\small $\delta$-boundary}
      \put(13.3,37.4){\small $\delta$-cusp}
      \put(36,47){\small $A_2'$}
      \put(35.4,38.5){\small $F_2\times I$}
      \put(54,49){\small use gluing lemma}
      \put(80,54){\small bubble to be collapsed}
      \put(99,46.6){\small $F'$}
      \put(19,30){\small $N(\mathcal{B'})$}
      \put(66,30){\small $N(\mathcal{B''})$ and $N(\mathcal{B'''})$}
      \put(40,24){\small $(a)$ $A_2'$ is a cusp}
      \put(4,15){\small $\delta$-cusp}
      \put(20,19){\small $\delta$-boundary}
      \put(26,13){\small collapse}
      \put(41,5){\small $A_2'$}
      \put(18,2.5){\small $N(\mathcal{B'})$}
      \put(100,5){\small $A_2$}
      \put(75,2.5){\small $N(\mathcal{B})$}
      \put(37,-3){\small $(b)$ $A_2'$ is a boundary}
    \end{overpic}
    \caption{Eliminating the annulus from $\mathcal{B'}$}
    \label{fig:ann_lem_2}
  \end{figure}
  
  \textbf{Case 1: $A_2'$ is a cusp.} Let $F_2\times I$ be the $I$-bundle component of $N(\mathcal{B'})-\mathcal{L'}$ whose boundary contains $A_2'$. We can extend $F_2\times \frac{1}{2}$ along $A_2'$ by attaching a small annulus to $A_2'$ and add this leaf $F'$ to $\mathcal{L'}$. The resulting lamination $\mathcal{L''} = \mathcal{L'}\cup \{F'\}$ is fully carried by some branched surface $\mathcal{B''}$, which is obtained by attaching an annulus to the $A_2'$-cusp of $\mathcal{B'}$, see Figure~\ref{fig:ann_lem_2}.$(a)$. Since $\mathcal{B'}$ has no disk of contact, $\mathcal{B''}$ also has no disk of contact. In particular $\mathcal{L''}$ has no disk leaf. Now by Gabai's gluing lemma (e.g. see~\cite{li2002laminar}, Lemma 3.4) we can glue the two boundary circles corresponding to $\delta$ and the $A_2'$-cusp, and get a branched surface $\mathcal{B'''}$ fully carrying a lamination. One can then collapse the annulus bubble of $\mathcal{B'''}$ to get $\mathcal{B}$, while the lamination is collapsed to be fully carried by $\mathcal{B}$.

  \textbf{Case 2: $A_2'$ is a boundary.} In this case we can simply collapse the attached annulus to the annulus bounded by $\delta$(-cusp) and $A_2'$, see Figure~\ref{fig:ann_lem_2}.$(b)$. The lamination is then collapsed to be fully carried by $\mathcal{B}$.
\end{proof}

\begin{rem}
  One should compare this lemma with Lemma~\ref{lem:cusp}. The fact that there is an annulus \textit{branch sector} (instead of just boundary information) makes things simpler and allows us to ``drop'' the tautness condition to some extent. This lemma could also be viewed as a generalization to~\cite{lyu2024knot}, Lemma 3.12.
\end{rem}

Now Lemma~\ref{lem:monotone-strongly_incoherent} follows immediately.

\begin{proof}[Proof of Lemma~\ref{lem:monotone-strongly_incoherent}]
  Let $\mathcal{B}_{S_0}$ be the type I branched surface associated to $S_0$, and $\mathcal{B}_{\alpha}$ be the type II branched surface associated to $\alpha$. Let $\mathcal{B}_{\alpha}'$ be the branched surface obtained by removing the punctured $\alpha$-source bigon in $\mathcal{B}_{\alpha}$ (see Figure~\ref{fig:remove_ann_I}). Then as discussed before, $\mathcal{B}_{\alpha}$ and $\mathcal{B}_{S_0}$ can both be obtained by attaching an annulus along a ``meridian'' on $\partial_h N(\mathcal{B}_{\alpha}')$; in particular this meridian is necessarily a core curve of an annulus component of $\partial_h N(\mathcal{B}_{\alpha}')$. Moreover, $\mathcal{B}_{\alpha}$ and $\mathcal{B}_{S_0}$ cannot have any disk of contact (such a disk of contact would become a meridional disk in the (1,1) knot complement). Hence by Lemma~\ref{lem:remove_ann} we know $\mathcal{B}_{\alpha}$ fully carries a lamination if and only if $\mathcal{B}_{\alpha}'$ fully carries a lamination, if and only if $\mathcal{B}_{S_0}$ fully carries a lamination. $\mathcal{B}_{\alpha}$ does fully carry a lamination by Proposition~\ref{prop:monotone}. Hence $\mathcal{B}_{S_0}$ also fully carries a lamination.
\end{proof}

\subsection{Primitive diagrams, and modified sink tube push}
\label{subsec:5.3}

In this subsection we prove Lemma~\ref{lem:primitive_lami}. The strategy is to split the branched surface to make it sink disk free, and is much similar to the proof strategy of~\cite{lyu2024knot}, Proposition 3.4. We will still define a ``sink tube push'', but will slightly modify it to fit the modified Heegaard branched surface here.

Let $(\Sigma,\alpha,\beta,z,w)$ be a reduced, $\alpha$-strongly incoherent (1,1) diagram. We place $\alpha$ in standard position and consider the sector $S_0$ (recall notations in subsection~\ref{subsec:2.3}). Let $\mathcal{B}_{S_0}$ be the type I modified Heegaard branched surface reversing $S_0$, and let $\mathcal{B}_{\alpha}'$ be the branched surface obtained by removing the punctured $\alpha$-source bigon from $\mathcal{B}_{S_0}$ (again see Figure~\ref{fig:remove_ann_I}). We also recall the ``pushing arcs'' notations in~\cite{lyu2024knot} (see also subsection~\ref{subsec:4.2}).

Recall the $\beta$-sink tube of the (1,1) diagram is the subregion of $\Sigma$ consisting of all the $\beta$-sink sectors. The boundary of the $\beta$-sink tube consists of two short ($\beta$-sink) $\alpha$-arcs and two long $\beta$-arcs. In~\cite{lyu2024knot}, subsection 3.4 we defined the ``$\beta$-sink tube push'' to be pushing one boundary long $\beta$-arc onto the other, such that we do not move the boundary $\beta$-arc of the source bigon (which is removed in the branched surface in~\cite{lyu2024knot}).

\begin{figure}[!hbt]
  \begin{overpic}[scale=0.6]{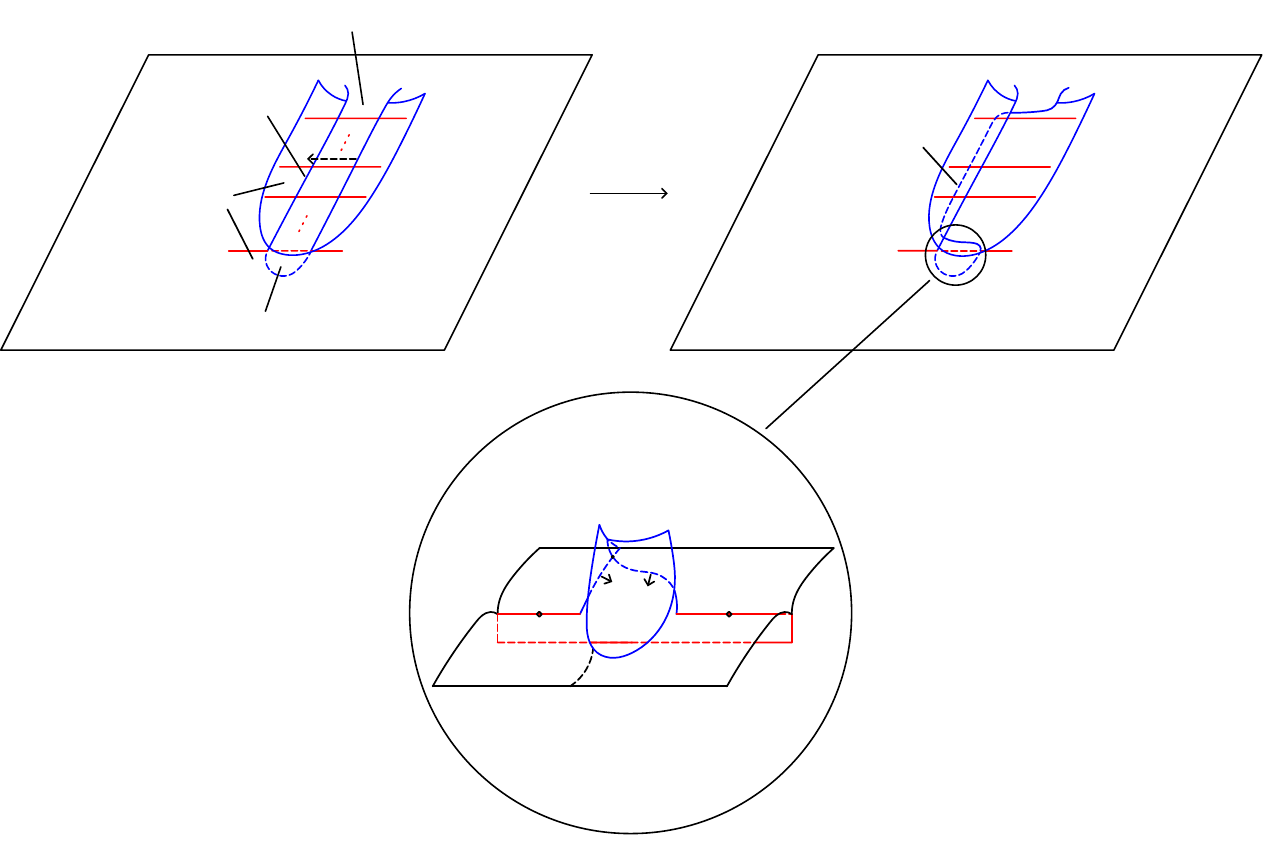}
    \put(14,66){\small the $\beta$-sink hexagon or octagon}
    \put(19,58.8){\small \color{blue} $\beta_0$}
    \put(23.7,44.7){\small \color{blue} $\beta_1$}
    \put(27.4,47){\small \color{red} $\alpha_0$}
    \put(15.7,51.4){\small $S_0$}
    \put(4,41){\small the removed $\beta$-sink bigon}
    \put(12,60.5){\small $\Sigma$}
    \put(17,36.8){\small $\mathcal{B}_{\alpha}'$}
    \put(64,59){\tiny branch locus}
    \put(64,57.55){\tiny arc pushed}
    \put(64,56){\tiny onto $\beta$-disk}
    \put(75,48.4){\scalebox{0.4}{$X$}}
    \put(80.6,58.6){\scalebox{0.4}{$Y$}}
    \put(49,22.8){\scalebox{0.4}{$X$}}
  \end{overpic}
  \caption{The modified $\beta$-sink tube push}
  \label{fig:modified_sink_tube_push}
\end{figure}

A \textbf{modified ($\beta$-)sink tube push} (with respect to the diagram $(\Sigma,\alpha,\beta,z,w)$) is a splitting of the branched surface $\mathcal{B}_{\alpha}'$ described by pushing the $\beta$-arcs. Since we reversed $S_0$ and removed the $\alpha$-source ($\beta$-sink) bigon, when pushing arcs for $\mathcal{B}_{\alpha}'$ we cannot move $\beta_0$ and $\beta_1$. We \textbf{suppose $\beta_0$ is on the boundary of $S_0$ and a $\beta$-sink sector}. Then $\beta_0$ is contained in one of the boundary long $\beta$-arcs of the $\beta$-sink tube. Our ``modified ($\beta$-)sink tube push'' is to push the other boundary long $\beta$-arc onto this one (to keep $\beta_0$ fixed), see Figure~\ref{fig:modified_sink_tube_push}. Moreover, $\beta_1$ is the boundary $\beta$-arc of the $\beta$-sink bigon; in order to keep $\beta_1$ fixed, we stop right before the $\beta$-sink bigon, not pushing the arcs over $\alpha_0$, see the bottom local picture of Figure~\ref{fig:modified_sink_tube_push}.

We can now prove Lemma~\ref{lem:primitive_lami}. Again, the proof is similar to ~\cite{lyu2024knot}, Proposition 3.4. We only replace the original sink tube push with the modified one above.

\begin{proof}[Proof of Lemma~\ref{lem:primitive_lami}]
  We place $\alpha$ in standard position. Since the diagram is incoherent, there are at least 2 rainbow arcs on each side of $\alpha$ and they are in alternating directions. In particular the diagram is automatically $\alpha$-strongly incoherent.

  We consider $\mathcal{B}_{S_0}$ the type I branched surface associated to $S_0$, and $\mathcal{B}_{\alpha}'$ the branched surface obtained by removing the punctured $\alpha$-source ($\beta$-sink) bigon from $\mathcal{B}_{S_0}$ (see subsection~\ref{subsec:5.2}). We prove that $\mathcal{B}_{\alpha}'$ fully carries a lamination. By the construction of $\mathcal{B}_{\alpha}'$, all double points are on the torus $\Sigma$. Moreover, since the diagram is $\alpha$-strongly incoherent and we removed the $\alpha$-source bigon, there are no double points on the interior of $\alpha_0$. Hence sink corners of all the double points are regions on $\Sigma$. In particular, any possible sink disk of $\mathcal{B}_{\alpha}'$ is some sector in the $\beta$-sink tube. See Figure~\ref{fig:6_1} for an example of sink disk here. We can then perform the modified $\beta$-sink tube push to $\mathcal{B}_{\alpha}'$ to eliminate these sink disks. We henceforth show that the new branched surface $\mathcal{B}_{\alpha}''$ after the splittings has no sink disk.

  \begin{figure}[!hbt]
    \begin{overpic}[scale=0.8]{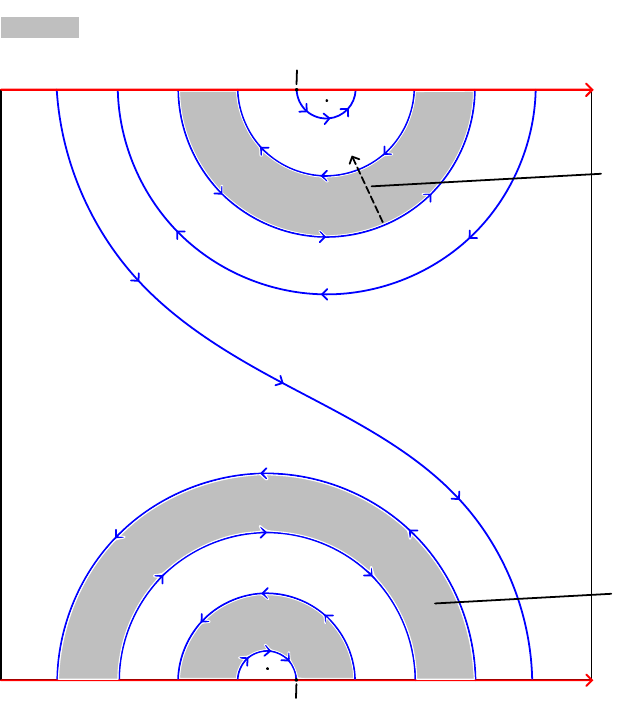}
      \put(13,95.2){sink tube}
      \put(39,80){$S_0$}
      \put(86.5,16.5){sink disk}
      \put(85,75){sink tube push}
      \put(28.8,-0.6){shared endpoint}
      \put(33,91){shared endpoint}
      \put(82,2){\color{red} $\alpha$}
      \put(58,38){\color{blue} $\beta$}
      \put(35.2,5.2){\small $z$}
      \put(43,85){\small $w$}
    \end{overpic}
    \caption{A primitive diagram for the twist knot $6_1$}
    \label{fig:6_1}
  \end{figure}

  In the proof of~\cite{lyu2024knot}, Proposition 3.4 we showed that for a primitive diagram with $\alpha$ in standard position, the two bigons share an endpoint. In short, we can look at the directions of the $\beta$-strands on the top and bottom (in other words, from different sides of $\alpha$). By Proposition~\ref{prop:primitive_standard}, at the ends of rainbow arcs the strands are going upwards and downwards alternatively, while at the ends of vertical arcs the strands are all in the same direction. This information provides a unique way to glue the top and bottom together; in particular the two ``rainbows'' are almost glued together, and the two bigons at the center of the rainbows share an endpoint. Again see Figure~\ref{fig:6_1} for an example.

  It follows that the sink tube is a spiral around the shared endpoint of the bigons, see Figure~\ref{fig:spiral}. When performing the modified sink tube push, every $\beta$-sink sector has its ``outer'' boundary $\beta$-arc moved onto its ``inner'' boundary $\beta$-arc. It follows that after the modified sink tube push, remaining (top or bottom) rainbow arcs on $\Sigma$ are all in the same direction. 

  \begin{figure}[!hbt]
    \begin{overpic}[scale=0.8]{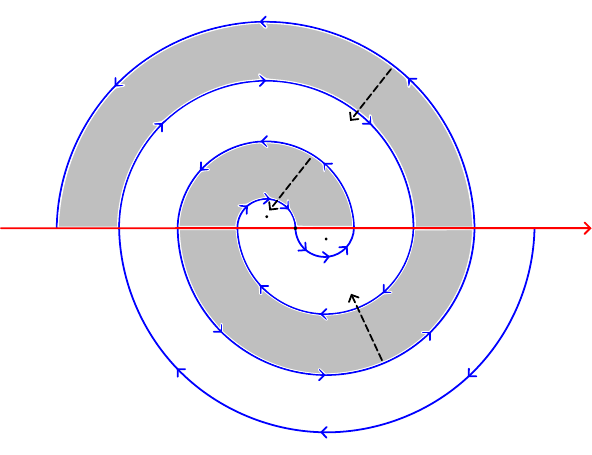}
      \put(42.2,39){\small $z$}
      \put(51.7,35.4){\small $w$}
      \put(48,29){\small $S_0$}
      \put(58,31){\small \color{blue} $\beta_1$}
      \put(67,26){\small \color{blue} $\beta_0$}
    \end{overpic}
    \caption{The sink tube spiral}
    \label{fig:spiral}
  \end{figure}

  The modified sink tube push eliminates some double points, and generates two new double points: $X$ near the $\beta$-sink bigon and $Y$ in the $\beta$-sink hexagon or octagon, see Figure~\ref{fig:modified_sink_tube_push}. The sink corner $D_X$ of $X$ contains the $\alpha$-disk (see the local picture of Figure~\ref{fig:modified_sink_tube_push}), hence is not a sink disk. The sink corner $D_Y$ of $Y$ is some region of the torus $\Sigma$ after pushing the arcs. Since the modified sink tube push does not move an entire vertical arc, $\partial D_Y$ contains a $\beta$-arc that comes from a vertical arc of the original (1,1) diagram. Moreover, $\partial D_Y$ contains multiple such $\beta$-arcs, since one needs to return to the same side of $\alpha$ (i.e. top or bottom) when traveling along $\partial D_Y$. Since our original (1,1) diagram is primitive, all these arcs are in the same direction; hence $D_Y$ cannot be a sink disk.

  The rest of the double points of $\mathcal{B}_{\alpha}''$ all come from double points of $\mathcal{B}_{\alpha}'$. The sink corners of these double points are all on the torus $\Sigma$. Let $D$ be any such sector. $D$ cannot be a bigon since the bigon is either punctured or removed. If $\partial D$ contains a $\beta$-arc coming from the vertical arcs of the diagram, then similar to the discussion before there are multiple such arcs and they are in the same direction; in particular $D$ is not a sink disk. If $\partial D$ does not contain such a $\beta$-arc, then since $D$ is not a bigon, $\partial D$ contains multiple (top or bottom) rainbow $\beta$-arcs. These arcs are again in the same direction as discussed before, and hence $D$ cannot be a sink disk.

  The branch locus of $\mathcal{B}_{\alpha}'$ contains 3 immersed circles. Since our splittings do not change the topology of $\partial_v N(\mathcal{B}_{\alpha}')$, the branch locus of $\mathcal{B}_{\alpha}''$ also consists of 3 immersed circles. These circles intersect at vertices of $S_0$ and are not isolated. Hence by Lemma~\ref{lem:sk_corner}, $\mathcal{B}_{\alpha}''$ is sink disk free. Again, our splittings do not change the topology of $\partial_h N(\mathcal{B}_{\alpha}')$, and $\partial_h N(\mathcal{B}_{\alpha}'')$ also consists of annuli. We can then apply Lemma~\ref{lem:sk_disk_free} and conclude that $\mathcal{B}_{\alpha}''$ fully carries a lamination. Hence $\mathcal{B}_{\alpha}'$ also fully carries a lamination. Finally, by Lemma~\ref{lem:remove_ann}, $\mathcal{B}_{S_0}$ also fully carries a lamination.
\end{proof}

\subsection{Type II reductions, associated branched surfaces in~\cite{lyu2024knot}, and blowing up branched surfaces}
\label{subsec:5.4}

In this subsection we prove Lemma~\ref{lem:StrongIncoherentReduction}. The proof strategy is much similar to those in~\cite{lyu2024knot}, and actually makes use of the results there. We first quote the relevant definition and results below:

\begin{defn}[\cite{lyu2024knot}, Definition 2.13]
  Let $(\Sigma,\alpha,\beta,z,w)$ be a non-simple reduced (1,1) diagram. We orient the two curves so that the boundaries of the two innermost bigons appear clockwise and counter-clockwise. We call the clockwise bigon \textit{source bigon} and the counter-clockwise one \textit{sink bigon}. Notice here the source bigon is both $\alpha$-source and $\beta$-source, and the sink bigon both $\alpha$-sink and $\beta$-sink.
  
  We construct a branched surface from the torus $\Sigma$ and the two compression disks bounded by $\alpha,\beta$ on different sides of $\Sigma$. We always assume that the $\alpha$-disk is on the negative side of $\Sigma$ and the $\beta$-disk is on the positive side. Smooth the disks so that the branch directions point to the left-hand side of the oriented curves (when observing on the oriented torus from the positive side). Then remove (the interior of) the source bigon and a small open disk in the interior of the sink bigon (i.e. puncture the sink bigon). The resulting branched surface is called the \textit{branched surface associated to the (1,1) diagram}. 
  \label{def:associated}
\end{defn}

\begin{rem}
  Unlike the modified Heegaard branched surfaces defined in this paper, the associated branched surface to a non-simple, reduced (1,1) diagram is essentially unique. The only choice involved is the orientations of $(\alpha,\beta)$, where the two different choices differ by a hyperelliptic involution.
\end{rem}

\begin{prop}[\cite{lyu2024knot}, Proposition 3.1]
  Every branched surface associated to a non-simple reduced (1,1) diagram fully carries a lamination.
  \label{prop:associated}
\end{prop}

The proof of Proposition~\ref{prop:associated} relies on an inductive argument; more precisely, we used the following proposition:

\begin{prop}[\cite{lyu2024knot}, Proposition 3.11]
  Let $(\Sigma,\alpha,\beta,z,w)$ be a reduced, non-simple (1,1) diagram, with $\delta$ carrying $\beta$. Suppose $(\Sigma,\alpha,\delta,z,w)$ is non-simple and its associated branched surface fully carries a lamination. Then the branched surface associated to $(\Sigma,\alpha,\beta,z,w)$ also fully carries a lamination.
  \label{prop:associated_red}
\end{prop}

In the proof of Proposition~\ref{prop:associated_red} we made a series of constructions that relates the branched surface associated to $(\Sigma,\alpha,\beta,z,w)$ to the branched surface associated to $(\Sigma,\alpha,\delta,z,w)$. We will do similar things here to prove Lemma~\ref{lem:StrongIncoherentReduction}, but this time relating some modified Heegaard branched surface of $(\Sigma,\alpha,\beta,z,w)$ to the branched surface associated to $(\Sigma,\alpha,\delta,z,w)$. The constructions below are actually parallel to those in~\cite{lyu2024knot}, subsection 3.3.

Let $\mathcal{B}_{S_0}$ be the type I modified Heegaard branched surface of $(\Sigma,\alpha,\beta,z,w)$ associated to $S_0$. We want to prove that $\mathcal{B}_{S_0}$ fully carries a lamination. Let $\mathcal{B}_{\alpha}'$ be the branched surface obtained by removing the punctured $\alpha$-source ($\beta$-sink) bigon from $\mathcal{B}_{S_0}$ (see Figures~\ref{fig:remove_ann_I}~\ref{fig:remove_ann_II}). By Lemma~\ref{lem:remove_ann}, $\mathcal{B}_{S_0}$ fully carries a lamination if and only if $\mathcal{B}_{\alpha}'$ does.

We now perform the constructions.

\vspace{6pt}

\textbf{Step 1: Attach the $\delta$-annulus.} Let $\delta$ be the carrying curve of $\beta$ on the (1,1) diagram. It passes through all the $\beta$-parallel sectors and the hexagons or octagon. In particular, $\delta$ is disjoint from the reversed sector $S_0$ and the removed bigon of $\mathcal{B}_{\alpha}'$. We pick the orientation of $\delta$ such that in the $(\alpha,\delta)$ diagram the $\alpha$-source bigon is also $\delta$-source. Then we can attach an annulus to $\mathcal{B}_{\alpha}'$ by identifying one of its boundary components with the $\delta$-curve, so that it is attached to the same side of $\Sigma$ as the $\beta$-disk, and we moreover smooth it so that the branch direction always points to the left of the $\delta$-curve (when observed on $\Sigma$ from the positive side). We call this new branched surface $\mathcal{B''}$. 

\begin{lem}
  $\mathcal{B}_{\alpha}'$ fully carries a lamination if $\mathcal{B''}$ does.
  \label{lem:remove_ann_prime}
\end{lem}

\begin{figure}[!hbt]
  \begin{overpic}{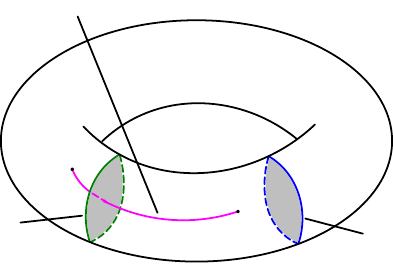}
      \put(1,13){$\color{green} \delta$}
      \put(93,10){$\color{ao} \beta$}
      \put(1,70){(1,1)-knot arc inside the solid torus}
      \put(14.5,28){$z$}
      \put(57,19){$w$}
      \put(100,25){$\Sigma$}
  \end{overpic}
  \caption{The $\delta$-disk and the (1,1)-knot}
  \label{fig:delta_disk}
\end{figure}

\begin{proof}
  The annulus can actually be seen as attached to a horizontal boundary component of $\mathcal{B}_{\alpha}'$. By Lemma~\ref{lem:remove_ann}, we only need to check that $\delta$ is a core curve of this horizontal boundary annulus. In fact, we show that $\delta$ is also a knot meridian when we identify $N(\mathcal{B}_{\alpha}')$ with the (1,1) knot complement. By definition the carrying curve $\delta$ is isotopic to $\beta$ on $\Sigma$ (where we forget the basepoints $z,w$), hence it also bounds a disk in the solid torus of the Heegaard splitting (of the ambient space). However, $\delta$ is not isotopic to $\beta$ on $(\Sigma,z,w)$ (having fewer intersection with $\alpha$), hence the $\delta$-disk would intersect the (1,1) knot essentially once, see Figure~\ref{fig:delta_disk}. It follows that the $\delta$-curve is indeed meridional in the (1,1) knot complement. Now since $\partial_v N(\mathcal{B}_{\alpha}')$ consists of meridional annuli, this $\delta$-curve is necessarily a core curve of some horizontal boundary annulus.
\end{proof}

\vspace{6pt}

\textbf{Step 2: Take the difference.} Let $\mathcal{C}$ be the branched surface associated to $(\Sigma,\alpha,\delta,z,w)$. By Proposition~\ref{prop:associated}, $\mathcal{C}$ fully carries a lamination. Let $\mathcal{C'}$ be the branched surface obtained by removing a disk in the interior of the $\delta$-disk of $\mathcal{C}$. Then $\mathcal{C'}$ also fully carries a lamination.

Since the $\delta$ curve passes and only passes through the $\beta$-parallel sectors and the hexagons or octagon, the reversed sector $S_0$ and the removed $\alpha$-source ($\beta$-sink) bigon of $\mathcal{B''}$ are contained in the $(\alpha,\delta)$-source bigon. Since we removed the $(\alpha,\delta)$-source bigon in the construction of $\mathcal{C}$, we can actually regard $\mathcal{C'}$ as a sub-branched surface of $\mathcal{B''}$. We can then take the difference $\mathcal{A}=\mathcal{B''}-\mathcal{C'}$, and attach $\mathcal{A}$ to the two horizontal boundary components $\partial_h^{\pm}$ of $\mathcal{C'}$ to form the branched surface $\mathcal{A'}$ (we assume that the $\beta$-disk is attached to $\partial_h^+$).

\begin{rem}
  It is important here that the diagram is ($\alpha$-)strongly incoherent. In fact, if it is not, then the reversed region cannot be contained in the source bigon of the $(\alpha,\delta)$ diagram (here the reduction is necessarily of type I, see Figure~\ref{fig:reductions}.$(a)$).
\end{rem}

\begin{lem}
  $\mathcal{B''}$ fully carries a lamination if $\mathcal{A'}$ does.
  \label{lem:attaching_to_subsurface}
\end{lem}

\begin{proof}
  The proof is essentially the same as that of~\cite{lyu2024knot}, Lemma 3.13. In fact, we can still regard the lamination $\mathcal{L}$ carried by $\mathcal{A'}$ as a closed subset in $N(\mathcal{B''})$. $\mathcal{L}$ is almost a lamination of $N(\mathcal{B''})$, except that at the cusp(s) of $\mathcal{C'}$ it is not properly embedded. Hence we need to extend $\mathcal{L}$ into the cusp(s) of $\mathcal{C'}$, and this is always achievable if $\mathcal{C'}$ fully carries a lamination and does not carry any disk of contact. Fortunately, this is true for our $\mathcal{C'}$, since any disk of contact would be a meridional disk in the complement of the (1,1) knot represented by $(\Sigma,\alpha,\delta,z,w)$. The union of the extended lamination and the lamination fully carried by $\mathcal{C'}$ is then fully carried by $\mathcal{B''}$.
\end{proof}

\textbf{Step 3: Collapse the boundary train track.} $\partial_h^{\pm}$ are two pairs of pants, each with 3 boundary circles corresponding respectively to the $\mathcal{C'}$-cusp, the puncture in the $\delta$-disk, and the puncture in the $(\alpha,\delta)$-sink bigon. When attaching $\mathcal{A}$ to $\partial_h^{\pm}$ to get $\mathcal{A'}$, we see that $\mathcal{A}$ only intersects the boundary circles of $\partial_h^{\pm}$ that correspond to the cusp of $\mathcal{C'}$ (we call them $l_{C'}^{\pm}$). It follows that $\mathcal{A'}$ has a bunch of boundary circles, and a boundary train track at this $\mathcal{C'}$-cusp. For technical reasons (in order to use Lemma~\ref{lem:sk_disk_free}) we need to collapse this boundary train track. We remark that the operations below are local (i.e. happens in a small neighborhood of the boundary train track), and hence are identical to those in Step 3 in~\cite{lyu2024knot}, subsection 3.3.

Let $\mathcal{T}$ be the boundary train track at the $\mathcal{C'}$-cusp. We notice that $\mathcal{A}$ intersects $l_{C'}^{\pm}$ each in two points, corresponding to the two vertices of the $(\alpha,\delta)$-source bigon. In particular $\mathcal{T}$ looks like Figure~\ref{fig:cusp_collapsing} left.

\begin{figure}[!hbt]
  \begin{overpic}[scale=0.6]{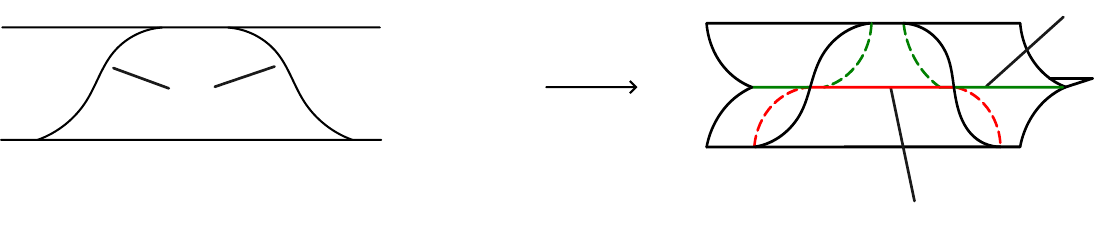}
      \put(40,12){$\times \;I$}
      \put(35.5,7){$l_{\mathcal{C'}}^+$}
      \put(35.5,18){$l_{\mathcal{C'}}^-$}
      \put(16,11.5){$\mathcal{A}$}
      \put(83,1){$\color{red}\alpha$}
      \put(97.5,20){$\color{green}\delta$}
      \put(72,14){\tiny$M$}
      \put(87,14){\tiny$N$}
      \put(6,3){\small boundary train track}
  \end{overpic}
\caption{Collapsing the train track $\mathcal{T}$ at the $\mathcal{C'}$-cusp}
\label{fig:cusp_collapsing}
\end{figure}

Now we collapse/pinch $\mathcal{T}$ into a single circle as shown in Figure~\ref{fig:cusp_collapsing}. We map the two $\mathcal{A}$-arcs of the boundary train track to two points, and identify the corresponding $l^{\pm}_{\mathcal{C'}}$-arcs that are cut out by the $\mathcal{A}$-arcs. We note that after collapsing the two double points $M,N$ (obtained by collapsing the $\mathcal{A}$-arcs) actually correspond to the vertices of the $(\alpha,\delta)$-source bigon. 

The resulting branched surface $\mathcal{A''}$ has a circle boundary at the $\mathcal{C'}$-cusp instead of a train track. The branch locus of $\mathcal{A''}$ consists of 4 immersed circles, and they correspond to the branch locus of $\mathcal{B''}$.

\begin{lem}
  $\mathcal{A'}$ fully carries a lamination if $\mathcal{A''}$ does.
  \label{lem:cusp_collapsing}
\end{lem}

\begin{proof}
  The proof is the same as that of~\cite{lyu2024knot}, Lemma 3.14 (recall all the operations are local). The point is that by removing a small neighborhood of the collapsed boundary in $\mathcal{A''}$ we regain $\mathcal{A'}$, thus we can simply restrict the lamination carried by $\mathcal{A''}$ to this $\mathcal{A'}$.
\end{proof}

\begin{lem}
  $\mathcal{A''}$ fully carries a lamination.
  \label{lem:the_hard_one}
\end{lem}

It is clear that now Lemma~\ref{lem:the_hard_one} would imply Lemma~\ref{lem:StrongIncoherentReduction}. We nevertheless formally state the proof:

\begin{proof}[Proof of Lemma~\ref{lem:StrongIncoherentReduction}]
  By Lemma~\ref{lem:the_hard_one} $\mathcal{A''}$ fully carries a lamination. Then by Lemmas~\ref{lem:cusp_collapsing}~\ref{lem:attaching_to_subsurface}~\ref{lem:remove_ann_prime}, we know $\mathcal{B}_{\alpha}'$ fully carries a lamination. Finally by Lemma~\ref{lem:remove_ann} $\mathcal{B}_{S_0}$ fully carries a lamination, as desired.
\end{proof}

It remains for us to prove Lemma~\ref{lem:the_hard_one}. Before we proceed, we first briefly visualize $\mathcal{A''}$ and justify our constructions.

As we stated before, the branch locus of $\mathcal{A''}$ consists of 4 immersed circles, and they correspond to the 4 immersed circles of the branch locus of $\mathcal{B''}$ (3 from $\mathcal{B}_{\alpha}'$, 1 from the attached $\delta$-annulus). The difference here is that, by attaching $\mathcal{A}$ to the horizontal boundary of $N(\mathcal{C'})$ instead of just $\mathcal{C'}$, we actually ``blow up'' the branched surface. This simplifies the intersections of the branch locus as immersed circles; in other words, we get fewer double points.

\begin{figure}[!hbt]
  \begin{overpic}{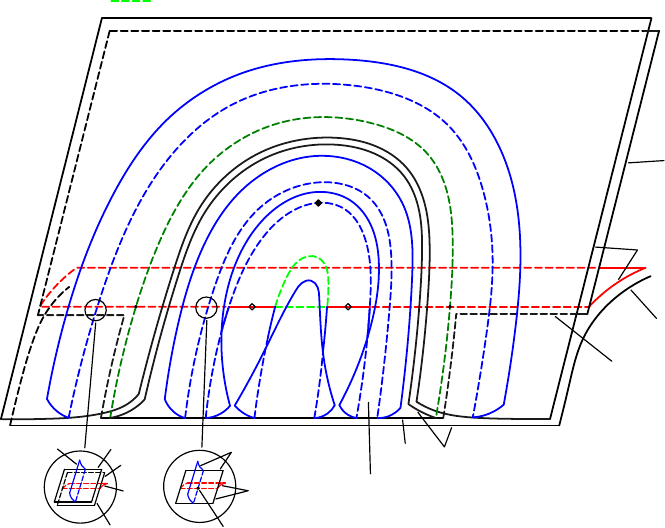}
    \put(23.5,78.5){\small boundary of the removed bigon (not part of the branch locus)}
    \put(46,43){\small $S_0$}
    \put(100.5,54.5){\small $\partial_h^-$}
    \put(96.4,41.5){\small $\partial_h^+$}
    \put(99,30){\small $\partial_h^-$}
    \put(92,23){\small $\mathcal{C'}$-cusp}
    \put(66,9){\small $\partial_h^-$}
    \put(59.5,9.5){\small $\partial_h^+$}
    \put(49,5){\small sink disk candidate}
    \put(17.5,33.8){\small $M$}
    \put(68,33.8){\small $N$}
    \put(16,12){\tiny $\partial_h^+$}
    \put(18.4,9){\tiny $\partial_h^-$}
    \put(18.6,4.5){\tiny $\partial_h^+$}
    \put(6,-1.6){\tiny collapsed train track}
    \put(6.6,12){\tiny $\mathcal{A}$}
    \put(34,11.7){\tiny $\mathcal{A}$}
    \put(38,5){\tiny $\partial_h^+$}
    \put(32,-1.6){\tiny double point}
  \end{overpic}
  \caption{A local picture of $\mathcal{A''}$ near the $(\alpha,\delta)$-source bigon}
  \label{fig:local_pic_App}
\end{figure}

Figure~\ref{fig:local_pic_App} depicts a local picture of $\mathcal{A''}$ near the $(\alpha,\delta)$-source bigon. Recall $\mathcal{A}=\mathcal{B''}-\mathcal{C'}$ consists of the $\beta$-disk and sectors in the $(\alpha,\delta)$-source bigon. As we attach $\mathcal{A}$ to $\partial_h^{\pm}$ and collapse the boundary train track to get $\mathcal{A''}$, we can still regard the $\beta$-disk as attached to ``the torus $\Sigma$'', where outside the $(\alpha,\delta)$-source bigon it is attached to $\partial_h^+$, and inside the $(\alpha,\delta)$-source bigon it is attached to the rest of $\mathcal{A}$.

Double points of $\mathcal{A''}$, being the intersections of its branch locus, still come from the intersections of the $\alpha$-curve and the $\beta$- and $\delta$-curves. The point here is that, outside the $(\alpha,\delta)$-source bigon $\alpha$ does not really intersect $\beta$ or $\delta$, when regarded as parts of the branch locus of $\mathcal{A''}$. Recall the cusp of $\mathcal{C'}$ comes from the cusp of the branched surface $\mathcal{C}$ associated to the $(\alpha,\delta)$ diagram. In particular this cusp can be described as ($\alpha\cup\delta-\partial$(the $(\alpha,\delta)$-source bigon)). Hence outside the $(\alpha,\delta)$-source bigon, $\alpha$ as part of the branch locus of $\mathcal{A''}$ actually comes from collapsing the boundary train track at the $\mathcal{C'}$-cusp, see Figure~\ref{fig:cusp_collapsing} and Figure~\ref{fig:local_pic_App}. As can be seen in Figure~\ref{fig:cusp_collapsing}, there are only two double points on the branch locus obtained by collapsing the boundary train track at the $\mathcal{C'}$-cusp, and these two points correspond to the vertices $M,N$ of the $(\alpha,\delta)$-source bigon. In particular outside the $(\alpha,\delta)$-source bigon $\alpha$ does not intersect other parts of the branch locus of $\mathcal{A''}$, see the left local picture at the bottom-left of Figure~\ref{fig:local_pic_App}.

On the other hand, inside the $(\alpha,\delta)$-source bigon $\alpha$ and $\beta$ normally intersect as parts of the branch locus of $\mathcal{A''}$, and we can spot the double points along the boundary $\alpha$-arc of the $(\alpha,\delta)$-source bigon. See the right local picture at the bottom-left of Figure~\ref{fig:local_pic_App}. In particular, sink corners of such double points could be sink disks, again see Figure~\ref{fig:local_pic_App}.

On reflections, this blowing up technique works more generally. We summarize it formally as below for possible future use. The proof is essentially the same.

\begin{defn}
  Let $\mathcal{B}$ be a branched surface possibly with circle boundary. A branched surface $\mathcal{C}\subset \mathcal{B}$ is called a \textit{sub-branched surface} of $\mathcal{B}$, if $\partial \mathcal{C}\subset \partial \mathcal{B}$, and each branch sector of $\mathcal{C}$ is a union of branch sectors of $\mathcal{B}$.

  Let $\mathcal{A'}$ be the setwise difference $\mathcal{A'}=\mathcal{B}-\mathcal{C}$. Let $\mathcal{A}$ be the branched surface obtained by first attaching $\mathcal{A'}$ to the horizontal boundary of $N(\mathcal{C})$ (as if attached to $\mathcal{C}$), and then collapsing the boundary train tracks. We call $\mathcal{A}$ the \textit{blow up branched surface of $\mathcal{B}$ with respect to $\mathcal{C}$}.
  \label{defn:blowup}
\end{defn}

\begin{prop}
  Let $\mathcal{B}$ be a branched surface possibly with circle boundary, and $\mathcal{C}\subset \mathcal{B}$ a sub-branched surface. Let $\mathcal{A}$ be the blow up of $\mathcal{B}$ with repsect to $\mathcal{C}$. Suppose $\mathcal{C}$ fully carries a lamination and has no disk of contact. Suppose moreover that $\mathcal{A}$ fully carries a lamination. Then $\mathcal{B}$ fully carries a lamination.
  \label{prop:blowup}
\end{prop}

\subsection{Rational tangle strands again, and modified splittings}
\label{subsec:5.5}

We are now ready to prove Lemma~\ref{lem:the_hard_one}. The idea is to split the branched surface by pushing arcs to get a new branched surface that is sink disk free. The techniques used here are similar to those in~\cite{lyu2024knot}, subsection 3.5.

\begin{proof}[Proof of Lemma~\ref{lem:the_hard_one}]
  We divide this long proof into several parts. In parts 1$\sim$5 we show that we can always split $\mathcal{A''}$ to obtain a branched surface with all double points safe. In part 6 we show that this implies that $\mathcal{A''}$ fully carries a lamination.

  \vspace{6pt}

  \textbf{Part 1: modified sink tube push on $\mathcal{A''}$.}

  The double points of $\mathcal{A''}$ appear on the boundary $\alpha$-arc of the $(\alpha,\delta)$-source bigon as its end points $M,N$ and its intersections with the $\beta$-arc. The sink corner of $M,N$ is the annulus obtained by collapsing the boundary train track, see Figure~\ref{fig:cusp_collapsing}. In particular $M,N$ are safe.
  
  Since we removed the $(\alpha,\beta)$-bigon inside the $(\alpha,\delta)$-source bigon, its vertices are not double points (see the light green curves in Figure~\ref{fig:local_pic_App}). If there are only two $\beta$-arcs in the $(\alpha,\delta)$-source bigon, then there are only two double points that come from the $(\alpha,\beta)$ intersections. Moreover, the sink corners of these two double points have some parts of the $\delta$-curve on their boundaries, along which the branch directions point outwards; in particular they are not sink disks. Hence in this case $\mathcal{A''}$ already has all double points safe. From now on we suppose there are at least 3 $\beta$-arcs inside the $(\alpha,\delta)$-source bigon.

  As stated before, we can still regard the $\beta$-disk as attached to some torus $\Sigma^+$\footnote{Here we use the $+$ sign to emphasize that this ``torus'' is slightly different from the original torus $\Sigma$, and that we are mostly attaching to $\partial_h^+$}, where outside the $(\alpha,\delta)$-source bigon it is attached to $\partial_h^+$, and inside the $(\alpha,\delta)$-source bigon it is attached to the rest of $\mathcal{A}$. We can then describe splittings of $\mathcal{A''}$ as pushing $\beta$-arcs on $\Sigma^+$.

  Notice that our ``torus'' $\Sigma^+$ is actually cut along the $\delta$-curve. That being said, $\Sigma^+$ actually has 2 boundary circles corresponding to $\delta$. One of them comes from the cusp of $\mathcal{A''}$, which we will denote as $\delta_{cusp}$ (see the dark green curve in Figure~\ref{fig:local_pic_App}); the other is the $\delta$-boundary circle of $\partial_h^+$, which we denote as $\delta_{\partial_h}$. Moreover, we reversed the sector $S_0$ and removed the $(\alpha,\beta)$-bigon inside the $(\alpha,\delta)$-source bigon in the construction of $\mathcal{B''}$. Hence when pushing $\beta$-arcs on $\Sigma^+$, we need to meet the following conditions:
  
  \begin{enumerate}
    \item we do not move the boundary arcs $\beta_0,\beta_1$ of $S_0$, and
    \item we do not push across $\delta$.
  \end{enumerate}

  Now that the $(\alpha,\delta)$-source bigon contains at least 3 $\beta$-arcs, $\beta_0$ is not the outermost one. Since $\beta$-arcs in the $(\alpha,\delta)$-source bigon are necessarily in alternating directions, $\beta_0$ is bounded by $S_0$ and a $\beta$-sink sector. It follows that we can perform a \textit{modified $\beta$-sink tube push} on $\Sigma^+$, using the $\beta$-sink tube defined by the $(\alpha,\beta)$ diagram. We can check that this operation meets the conditions above. In fact, by definition this operation is not moving $\beta_0$ and $\beta_1$. Moreover, since the $\delta$-curve only passes through $\beta$-parallel sectors and the hexagons or octagon, this operation also does not push across $\delta$.

  \begin{figure}[!hbt]
    \begin{overpic}[scale=0.33]{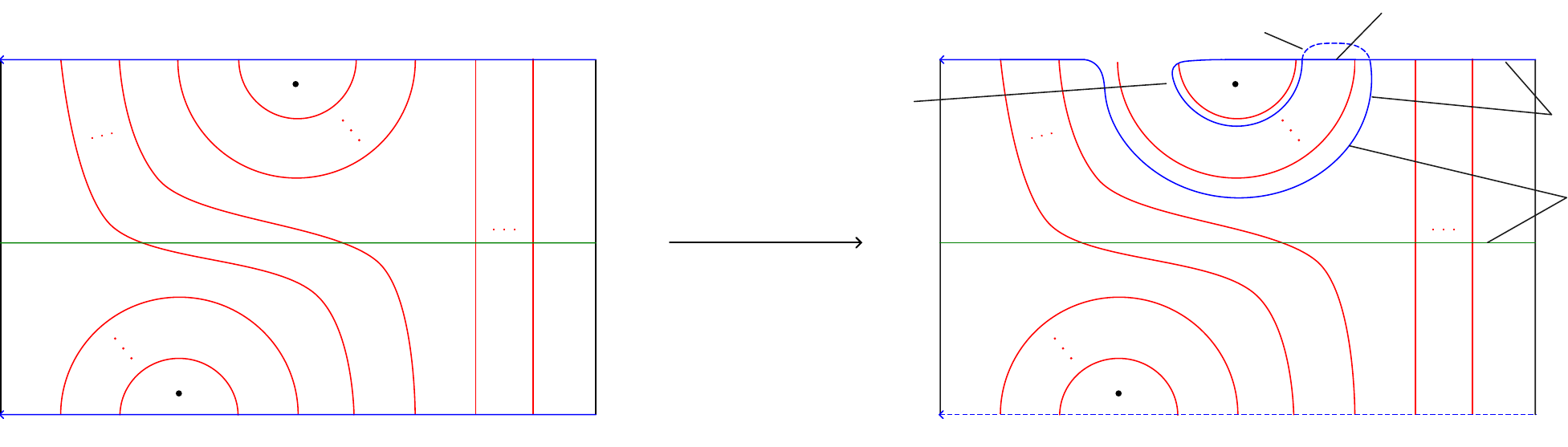}
        \put(10,3){\tiny$z$}
        \put(19,23){\tiny$w$}
        \put(36,0){\tiny $\color{ao}\beta$}
        \put(34.3,20){\tiny $\color{red}\alpha$}
        \put(36,11){\tiny $\color{green}\delta$}
        \put(100,16){\tiny parallel}
        \put(100,14){\tiny $\delta$-circles}
        \put(88.4,28){\tiny $\beta_c$}
        \put(99,20){\tiny $\beta_{\delta}$}
        \put(56,21){\tiny $\beta_s$}
        \put(18,27){\tiny \color{blue}$\beta_1$}
        \put(15.2,24.8){\tiny $\overbrace{\qquad\quad\:}$}
        \put(54,27){\tiny arc pushed onto the $\beta$-disk}
    \end{overpic}
    \caption{Modified ($\beta$-)sink tube push for $\beta$ in standard position}
    \label{fig:connecting_arc}
\end{figure}

  After this operation the $\beta$-curve that remains on $\Sigma^+$ (including $\beta_1$) consists of two loops and an arc connecting them, see Figure~\ref{fig:connecting_arc}, where we put $\beta$ in standard position instead. There is a small loop containing $\beta_1$, which we denote as $\beta_s$ ($s$ for \textit{sink}, since this loop bounds a disk containing the $\beta$-sink bigon). We denote the other loop $\beta_{\delta}$. Notice that as shown in Figure~\ref{fig:connecting_arc}, $\beta_{\delta}$ is isotopic to $\delta$ in the 2-pointed torus $(\Sigma,z,w)$, as they both intersect each vertical $\alpha$-arc exactly once. Hence on $\Sigma^+$ our $\beta_{\delta}$ is actually isotopic to one of the $\delta$-boundary components of $\Sigma^+$. We denote the arc connecting the loops $\beta_c$.
  
  The modified sink tube push on $\Sigma^+$ creates two new double points on the two loops $\beta_s$ and $\beta_{\delta}$ respectively. The new double point on $\beta_{\delta}$ is safe, as its sink corner has boundary on the $\delta$-boundary or cusp. The new double point on $\beta_s$ is also safe, since its sink corner would contain the $\alpha$-disk (recall the local picture in Figure~\ref{fig:modified_sink_tube_push}). Remaining double points that may not be safe are the intersections of $\beta_c$ and the boundary $\alpha$-arc of the $(\alpha,\delta)$-source bigon.
  
  \vspace{6pt}

  \textbf{Part 2: reduce $\Sigma^+$ to a 4-basepoint sphere $\tilde{S}$.}

  Again, it is possible to reduce our ``pushing arcs'' models and regard $\beta_c$ as a rational tangle strand on a 4-punctured sphere. The argument here is similar to that in~\cite{lyu2024knot}, subsection 3.5, but also slightly different. We restate the procedure in details below. We again remark that the following ``collapsing'' operations are purely for simplifying our characterization of ``pushing arcs'', and do not actually happen on our branched surface. One may compare the following Figure~\ref{fig:collapsing_tangle} to~\cite{lyu2024knot}, Figure 24.

  \begin{figure}[!hbt]
    \begin{overpic}[scale=0.65]{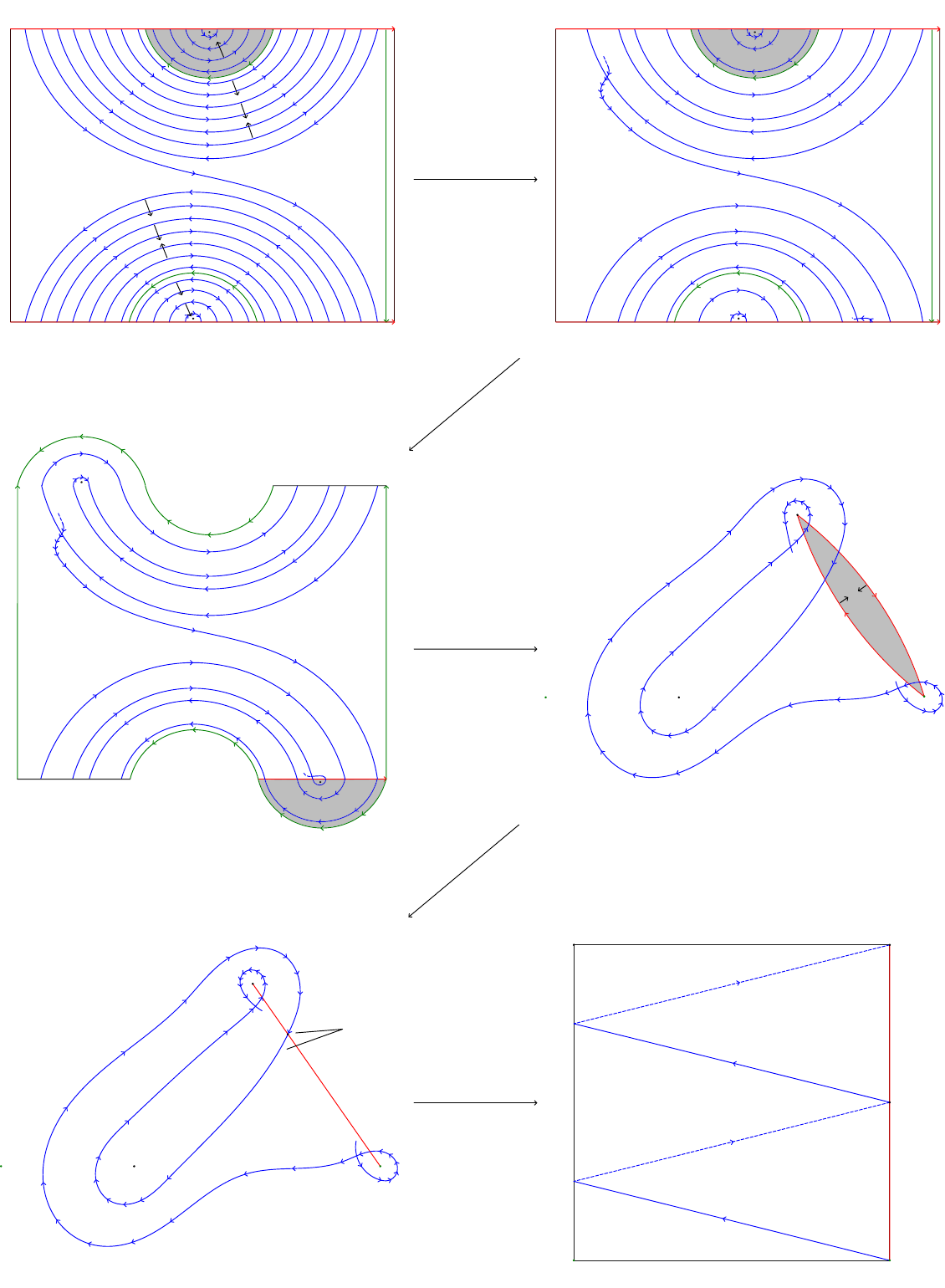}
        \put(9.5,72){\tiny knot (23,11,1,7)}
        \put(29,74.2){\Tiny $\color{red} \alpha$}
        \put(25,83){\Tiny $\color{ao} \beta$}
        \put(29,90){\Tiny $\color{green} \delta$}
        \put(9.4,74){\Tiny $M$}
        \put(19.3,74){\Tiny $N$}
        \put(29.3,98){\Tiny $M$}
        \put(10.6,98){\Tiny $N$}
        \put(14.6,74.2){\Tiny $z$}
        \put(15.6,97.9){\Tiny $w$}
        \put(31.1,87){\tiny $\beta$-sink tube push}
        \put(51.7,74){\Tiny $M$}
        \put(61.6,74){\Tiny $N$}
        \put(71.7,98){\Tiny $M$}
        \put(53.7,98){\Tiny $N$}
        \put(56.9,74.2){\Tiny $z$}
        \put(58,97.9){\Tiny $w$}
        \put(37,68){\tiny cut along $\delta$}
        \put(5.8,63.2){\Tiny $z(d)$}
        \put(25,39.7){\Tiny $w(a)$}
        \put(1.5,50){\Tiny $\color{green} \delta_{\partial_h}(c)$}
        \put(30.1,45){\Tiny $\color{green} \delta_{cusp}(b)$}
        \put(30.8,51){\tiny collapse circle boundaries}
        \put(61,59.6){\Tiny $a$}
        \put(53,46){\Tiny $d$}
        \put(41.4,45.3){\Tiny $\color{green} c$}
        \put(70.7,45.6){\Tiny $\color{green} b$}
        \put(37,32){\tiny collapse $\alpha$-bigon}
        \put(21.4,19.5){\Tiny $e$}
        \put(27,20){\tiny sink corners of $e$}
        \put(18.8,23.3){\Tiny $a$}
        \put(10.5,9.5){\Tiny $d$}
        \put(0,9.7){\Tiny $\color{green} c$}
        \put(28.7,9.1){\Tiny $\color{green} b$}
        \put(27,16.8){\tiny collapsing $\beta$-loops and}
        \put(27,15.5){\tiny put in standard position}
        \put(69.2,14.3){\Tiny $e$}
        \put(43.6,27){\Tiny $d$}
        \put(69,27){\Tiny $a$}
        \put(69.3,2){\Tiny $\color{green} b$}
        \put(43.5,2){\Tiny $\color{green} c$}
        \put(47.5,0){\tiny strand of rational tangle $\frac{1}{4}$}
    \end{overpic}
    \caption{Rational tangle strand from knot (23,11,1,7)}
    \label{fig:collapsing_tangle}
  \end{figure}
  
  Recall that our $\Sigma^+$ is actually cut along $\delta$, with two boundary $\delta$-circles $\delta_{cusp}$ and $\delta_{\partial_h}$. Moreover, there is also a boundary circle at the punctured basepoint $z$. Since we can never push arcs across these boundary circles, we can collapse(quotient) them to points. If we further mark the midpoint of the boundary $\alpha$-arc of the $\beta$-sink bigon (which also cannot be moved over by our pushing arcs operations), we get a sphere with 4 basepoints marked. We denote it by $(\tilde{S},a,b,c,d)$, where $\tilde{S}$ is the sphere we get, $a$ is the midpoint of the boundary $\alpha$-arc of the $\beta$-sink bigon, $b$ is the $\delta_{cusp}$ boundary circle, $c$ is the $\delta_{\partial_h}$ circle, and $d$ is the boundary circle at $z$, see the first four pictures of Figure~\ref{fig:collapsing_tangle}, where we use the knot (23,11,1,7) as an example and $\alpha$ is placed in standard position. We remark that since double points are all on the boundary $\alpha$-arc of the $(\alpha,\delta)$-source bigon, starting from the third picture in Figure~\ref{fig:collapsing_tangle} we only draw this $\alpha$-arc instead of the whole $\alpha$-curve.
    
  Now as in the fourth picture of Figure~\ref{fig:collapsing_tangle}, the $(\alpha,\delta)$-source bigon (see the shaded regions) becomes a bigon connecting two basepoints on $\tilde{S}$, bounded only by $\alpha$-arcs (that come from the boundary $\alpha$-arc of the source bigon). We can then collapse it to a \textit{single} arc connecting the basepoints $a,b$. More precisely, the bigon can be parametrized by the unit disk, so that $(\pm 1,0)$ correspond to the vertices $a,b$ and the $\beta$ arcs intersecting the bigon are vertical; by collapsing we mean a projection of the unit disk to the $x$-axis. See this procedure in Figure~\ref{fig:collapsing_tangle} from the fourth picture to the fifth picture.
    
  Recall again that after the modified sink tube push the $\beta$-arcs left on ``$\Sigma$'' are two loops $\beta_s$, $\beta_{\delta}$ and a connecting arc $\beta_c$. Now in our 4-point sphere $(\tilde{S},a,b,c,d)$, each of the two $\beta$-loops bounds a small disk with a basepoint inside. We then further collapse(quotient) the loops and the disks they bound to the basepoints inside. Then $\beta_c$ becomes an arc connecting the basepoint $a$ to either $b$ or $c$, and thus may be viewed as a strand of a rational tangle. See the last two pictures in Figure~\ref{fig:collapsing_tangle}.

  We claim that our collapsings do not collapse possible sink disks. Notice that the branch direction along each of the two $\beta$-loops points to the small disk with basepoint inside. When collapsing the $\beta$-loops to the corresponding basepoints, we only ``collapsed'' the sink corners of the double points on the two $\beta$-loops. But we have previously argued that these two double points are safe, hence no sink disk is collapsed. On the other hand, the bigon in the fourth picture of Figure~\ref{fig:collapsing_tangle} is bounded by $\alpha$ opposite its branch direction, so when collapsing it we do not collapse any sink corner. However, we do collapse the double points in pairs (that are originally connected by $\beta$-arcs in the $(\alpha,\delta)$-source bigon), so in the 4-point sphere after collapsings each intersection of $(\alpha,\beta)$ has two sink corners, see the fifth picture of Figure~\ref{fig:collapsing_tangle}.

  We denote the two arcs connecting basepoints on $\tilde{S}$ as $\tilde{\alpha},\tilde{\beta}$. We fix a frame so that $\tilde{\alpha}$ represents a strand of the rational tangle $\frac{1}{0}$. Since $\tilde{\alpha},\tilde{\beta}$ come from a reduced (1,1)-diagram, their position on $\tilde{S}$ is tight, i.e. having no trivial bigons. Then $\tilde{\beta}$ is almost\footnote{To completely determine $\tilde{\beta}$ we need a bit more information, such as the starting point; for the purpose of this paper we only need the ``standard position'' characterization described below, and the fact that every $\tilde{\beta}$ has such a characterization.} \textit{characterized by} a rational number $r=\frac{p}{q}$, where $p,q\geq 0$, $(p,q)=1$. In fact, we can place $\tilde{\beta}$ in certain ``standard position'' as in Figure~\ref{fig:collapsing_tangle} bottom-right, where the rectangle is bubbled to represent the sphere $S$. The boundary segments of the bubbled rectangle are chosen to represent rational tangles $\frac{1}{0}$ and $\frac{0}{1}$. The solid $\tilde{\beta}$-segments are in the front of the bubbled rectangle, while the dashed segments are on the back of it. If we regard the rectangle as the unit rectangle, then these segments are all of slopes $\pm \frac{p}{q}$. We also recall that the numbers $p,q$ actually count the total intersections of $\tilde{\beta}$ with the $\frac{0}{1}$ and $\frac{1}{0}$ arcs $bc$ and $ab$, respectively, where we count 2 for each intersection other than the basepoints, and 1 for each intersection at the basepoints. 

  Since $\tilde{\alpha}$ does intersect $\tilde{\beta}$ at the non-basepoint $e$ and basepoint $a$, we know $q\geq 3$, and thus $p\geq 1$. \textbf{We claim that if $p=1$, then all double points are safe.} In fact, suppose we start tracing the $\tilde{\beta}$-arc from $a$; then since it will not intersect horizontal segments $ad$ and $bc$ again before getting to the other endpoint, it will necessarily wind from $a$ all the way down to the bottom, see Figure~\ref{fig:collapsing_tangle} bottom-right. Hence $\tilde{\beta}$ always intersects (the interior of) $\tilde{\alpha}$ in the same direction. Now consider the segments of the $\tilde{\alpha}$-arc cut out by $\tilde{\beta}$. We claim that none of them is the boundary arc of a sink disk. In fact, except the segments with an endpoint $a$ or $b$, all other segments are ``$\beta$-parallel'' back on $\Sigma^+$, thus cannot appear on the boundary of a sink disk. For the segment with an endpoint $b$, we notice that back on $\Sigma^+$ this $b$ is a branch locus circle, whose branch direction points to the other side (i.e. not the side our $\alpha$-segments are attached to)\footnote{Here $b$ is just $\delta_{cusp}$ by definition. We state the argument in a slightly more general sense to also care for later situations, where we might have collapsed a few more loops to $b$.}. In particular our $\alpha$-segments cannot be boundary arcs of sink disks. The segment with endpoint $a$ is $ae$, and back on $\Sigma^+$ correspond to the boundary $\alpha$-arcs of the reversed sector $S_0$, where the branch direction points to the $\alpha$-disk. Hence they cannot be boundary arcs of sink disks on $\Sigma^+$. It follows that no sink corners at the $(\alpha,\beta)$-intersections are actually sink disks, and hence our branched surface (obtained by doing a modified $\beta$-sink tube push to $\mathcal{A''}$) already has all double points safe. \textbf{From now on we suppose $p\geq 2$.} 

  \vspace{6pt}

  \textbf{Part 3: set up the $\tilde{\beta}$-sink tube on $\tilde{S}$.}

  We will henceforth describe our splittings as pushing the $\tilde{\beta}$-arcs on $\tilde{S}$. Recall that when pushing $\beta$ on $\Sigma^+$ we need to fix arcs $\beta_0,\beta_1$ and cannot move across $\delta$. Here on $\tilde{S}$ since the $\delta$-boundary circles are already collapsed to basepoints, we automatically will not push $\tilde{\beta}$ across them. The arc $\beta_1$ is formally collapsed to the basepoint $a$, and the arc $\beta_0$ is collapsed to the point $e$ that appears on $\tilde{\alpha}$ as the $(\tilde{\alpha},\tilde{\beta})$-intersection next to $a$, see the last two pictures in Figure~\ref{fig:collapsing_tangle}. Hence when pushing $\tilde{\beta}$ on $\tilde{S}$, \textbf{besides fixing the basepoints, we also need to fix $e$}.

  We now introduce the $\tilde{\beta}$-sink tube on $\tilde{S}$. This part is identical to the first half of Step 3 of~\cite{lyu2024knot}, proof of Lemma 3.15. We briefly summarize it here and refer the readers to~\cite{lyu2024knot} for details. We can put $\tilde{\beta}$ in ``standard position'' as described before, and take the boundary of the rectangle $abcd$ ($a$-$b$-$c$-$d$-$a$) as a simple closed curve $\epsilon$ on $\tilde{S}$. Now $\tilde{\beta}$ and $\epsilon$ cut $\tilde{S}$ into several regions. If we take the intersections of $\tilde{\beta}$ and $\epsilon$ as the vertices of these regions (notice here a basepoint is not necessarily a vertex), then when $p,q\geq 2$, we obtain 2 bigons (each contains a non-vertex basepoint on the boundary), 2 triangles (each has one basepoint vertex), and many quadrilaterals. See Figure~\ref{fig:bigons_n_triangles} for an example of local pictures at the basepoints. We note that when $p=2$ the two triangles share a boundary $\tilde{\beta}$-arc, as depicted in Figure~\ref{fig:bigons_n_triangles}.$(c)$.

  \begin{figure}[!hbt]
    \begin{overpic}[scale=0.9]{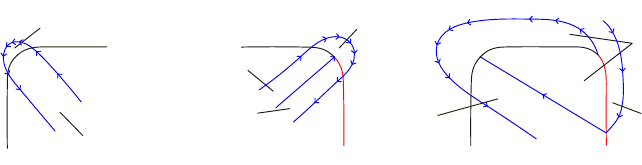}
        \put(5,-2){$(a)$}
        \put(37,-2){$(b)$ $p,q\geq 3$}
        \put(78,-2){$(c)$ $p=2$}
        \put(2.5,15.5){\Tiny $d$}
        \put(12,2.2){\Small sink sector}
        \put(7,21.5){\Small (sink) bigon}
        \put(52.5,16.5){\Tiny $a$}
        \put(25,15.3){\Small parallel sector}
        \put(22,7.5){\Small source sector}
        \put(44,22){\Small (source) triangle}
        \put(73,16.5){\Tiny $d$}
        \put(92,15.8){\Tiny $a$}
        \put(62,6){\Small sink}
        \put(62,3.3){\Small sector}
        \put(100.5,7){\Small source}
        \put(100.5,4.3){\Small sector}
        \put(99.5,18){\Small triangles}
    \end{overpic}
    \caption{Bigons and Triangles at the basepoints}
    \label{fig:bigons_n_triangles}
\end{figure}

  Since $\tilde{S}$ succeeds a fixed orientation from $\Sigma^+$, and $\tilde{\beta}$ is oriented, we can define ($\tilde{\beta}$-)sink, source, and parallel $\epsilon$-arcs and sectors as in Definition~\ref{def:sink_source_parallel}. In particular the definition also makes sense when an endpoint of the $\epsilon$-arc is an endpoint of $\tilde{\beta}$. Moreover, as we have the hyperelliptic involution for (1,1) diagrams, for $(\tilde{S},\tilde{\beta},\epsilon)$ we have a set of involutions $\{\tau_1,\tau_2,\tau_3\}$, which gives the symmetry similar to Lemma~\ref{lem:pre_sink_tube} (see the ``Claim'' in the proof of~\cite{lyu2024knot}, Lemma 3.15). For example, there is exactly one sink bigon and one source bigon. Moreover, since each triangle has exactly one $\tilde{\beta}$-parallel boundary arc (see Figure~\ref{fig:bigons_n_triangles}.$(b)(c)$), there is exactly one $\tilde{\beta}$-sink arc and one $\tilde{\beta}$-source arc among the boundary $\epsilon$-arcs of the triangles, and they belong to different triangles. We further call the triangle with $\tilde{\beta}$-sink boundary arc the \textbf{$\tilde{\beta}$-sink triangle}, and the triangle with $\tilde{\beta}$-source boundary arc the \textbf{$\tilde{\beta}$-source triangle}. 
  
  In particular, $ae$ is some boundary arc of a triangle. Since $ae$ is to the right (i.e. source direction) of $\tilde{\beta}$ at $e$ (recall it comes from the boundary $\alpha$-arcs of the \textit{source} sector $S_0$), it is actually a boundary arc of the $\tilde{\beta}$-source triangle, again see Figure~\ref{fig:bigons_n_triangles}.$(b)(c)$. We remark that $ae$ is not necessarily $\tilde{\beta}$-source because of the collapsing of the loop, see Figure~\ref{fig:collapsing_tangle} for an example.
  
  This gives us enough information to define a $\tilde{\beta}$-sink tube as in Definition~\ref{def:sink_tube}, since there are only 2 $\tilde{\beta}$-sink arcs among the boundary $\epsilon$-arcs of the non-quadrilateral regions. Similar to Lemma~\ref{lem:char_sink_tube}, this sink tube contains all sink sectors.

  \vspace{6pt}

  \textbf{Part 4: perform the old sink tube push for rational tangle strands whenever possible.}

  In~\cite{lyu2024knot} we defined the sink tube push for rational tangle strands. The idea was to find a systematic way to simplify $\tilde{\beta}$ while fixing the boundary $\tilde{\beta}$-arc of the $\tilde{\beta}$-source bigon. The situation is different here. Instead of fixing the boundary $\tilde{\beta}$-arc of the $\tilde{\beta}$-source bigon, we now need to fix the boundary $\tilde{\beta}$-arc of the $\tilde{\beta}$-source triangle, which contains $e$ as a vertex.

  However, we still need the old sink tube push. We will perform it whenever it does not move $e$. We now briefly recall the arguments in~\cite{lyu2024knot}.

  \begin{figure}[!hbt]
    \begin{overpic}[scale=0.5]{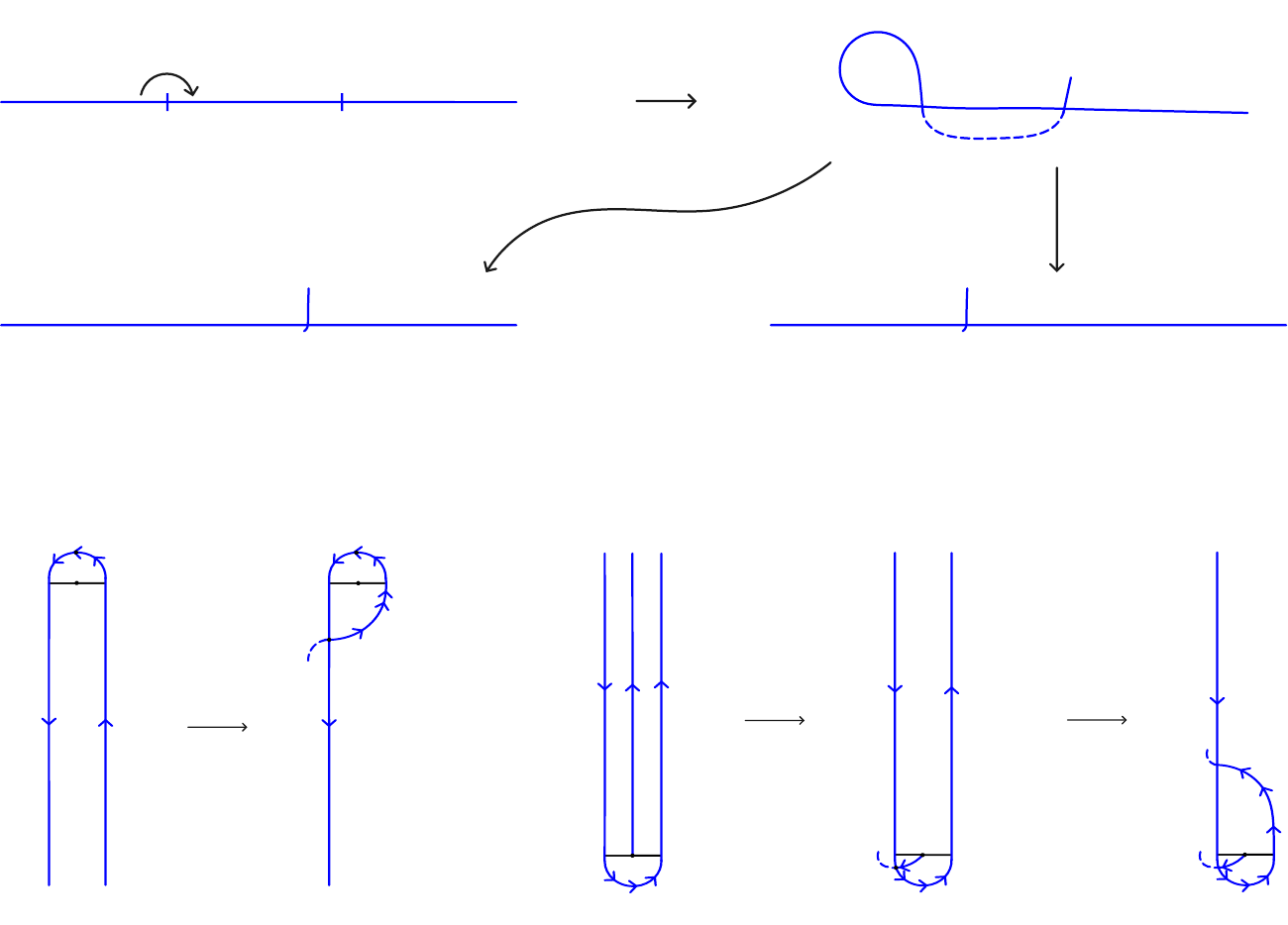}
        \put(50,36){$(a)$}
        \put(15,-1){$(b)$}
        \put(70,-1){$(c)$}
        \put(-1,62){\tiny$P$}
        \put(12,62){\tiny$Q$}
        \put(25.5,62){\tiny$R$}
        \put(39,62){\tiny$S$}
        \put(83,66){\tiny$P$}
        \put(66,62){\tiny$Q$}
        \put(97,62){\tiny$S$}
        \put(25,55){\tiny$|PQ|>|QR|$}
        \put(84,55){\tiny$|PQ|<|QR|$}
        \put(23,50){\tiny $P$}
        \put(-2,44.5){\tiny $Q(S')$}
        \put(10,44.5){\tiny $R(R')$}
        \put(21,44.5){\tiny $X_p(Q')$}
        \put(38,44.5){\tiny $S(P')$}
        \put(74,50){\tiny $P$}
        \put(57,44.5){\tiny $Q(P')$}
        \put(72,44.5){\tiny $X_p(Q')$}
        \put(84,44.5){\tiny $R(R')$}
        \put(97,44.5){\tiny $S(S')$}
        \put(49,6){\tiny $P$}
        \put(71,6){\tiny $P$}
        \put(67,3){\tiny $X_p$}
        \put(5,30){\tiny $Q$}
        \put(27,30){\tiny $Q$}
        \put(22.3,23){\tiny $X_q$}
    \end{overpic}
    \caption{The old sink tube push for the rational tangle strand}
    \label{fig:old_tangle_push}
  \end{figure}

  One can define each $\tilde{\beta}$-segment cut out by $\epsilon$ to be of length 1. By the symmetry of $\tilde{\beta}$ we can describe it as Figure~\ref{fig:old_tangle_push}.$(a)$ upper-left, where $P,S$ are endpoints of $\tilde{\beta}$, $Q$ is the midpoint of the boundary $\tilde{\beta}$-arc of the sink bigon, and $R$ the midpoint of the boundary $\tilde{\beta}$-arc of the source bigon, such that $P,Q,R,S$ appear in order and $\mathrm{d}(P,Q)=\mathrm{d}(R,S)$. Then $P$ is on the boundary of the $\tilde{\beta}$-sink triangle (see Figure~\ref{fig:old_tangle_push}.$(c)$ left), and $S$ is on the boundary of the $\tilde{\beta}$-source triangle; in particular $S$ is the basepoint $a$. 
  
  The old ``sink tube push'' is then to push the long $\tilde{\beta}$-arc of the sink tube in $PQ$ onto the other long $\tilde{\beta}$-arc in $QS$. Moreover, we modify our operation at the ends of the tube, so that near the $\tilde{\beta}$-sink bigon we obtain a $\tilde{\beta}$-loop bounding a basepoint (see Figure~\ref{fig:old_tangle_push}.$(b)$), and near the $\tilde{\beta}$-sink triangle we push the $\tilde{\beta}$-arc connected to $P$ onto the $\tilde{\beta}$-edge of the triangle (see the first arrow in Figure~\ref{fig:old_tangle_push}.$(c)$). See also the first arrow of Figure~\ref{fig:tangle_example}.$(a)$.

  \begin{figure}[!hbt]
    \begin{overpic}[scale=0.55]{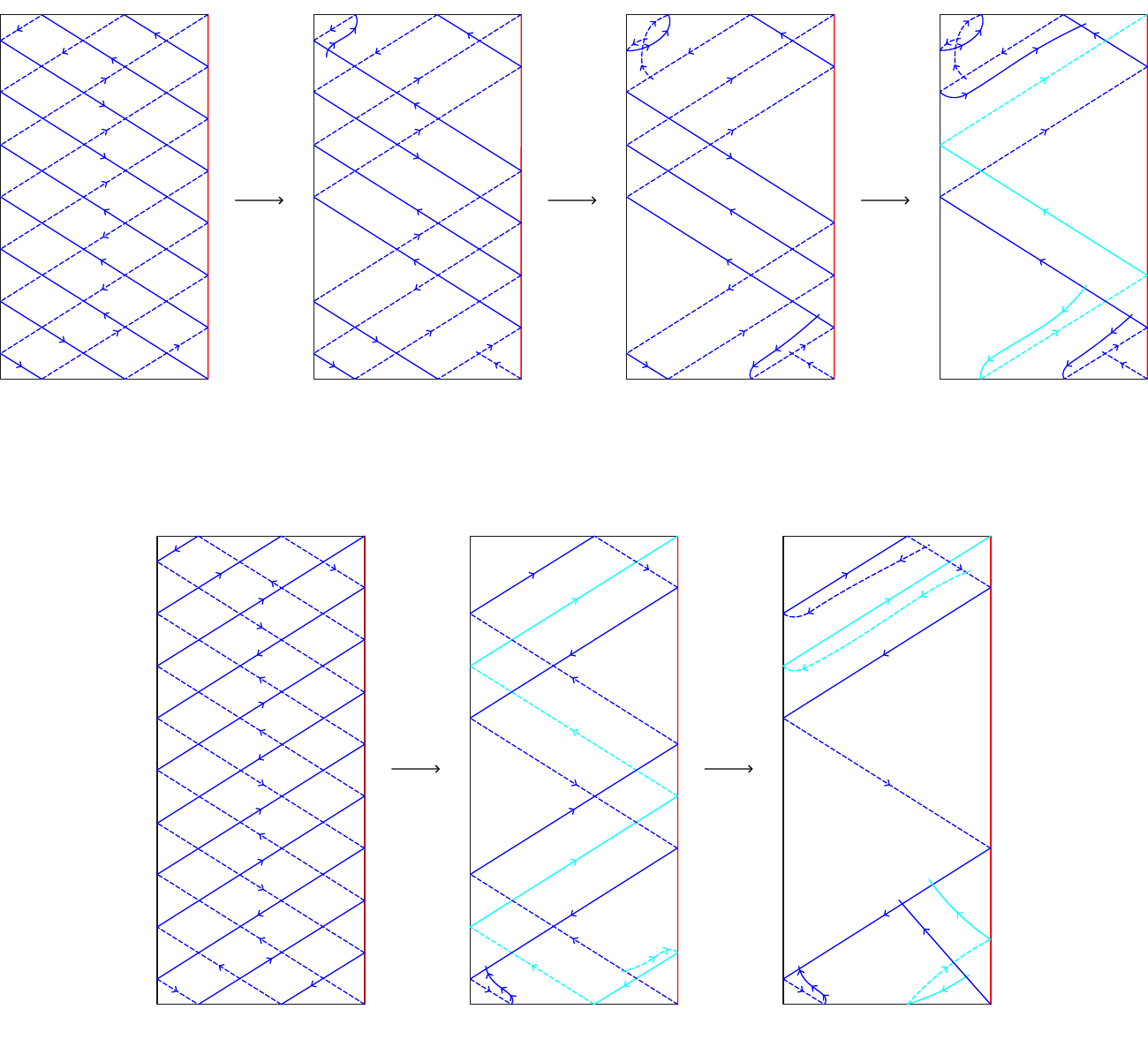}
      \put(48.5,49){$(a)$}
      \put(7.5,55){$\frac{5}{14}$}
      \put(35,55){$\frac{4}{11}$}
      \put(62.5,55){$\frac{3}{8}$}
      \put(90,55){$\frac{1}{4}$}
      \put(18,90){\tiny $a$}
      \put(18.3,84.9){\tiny $e$}
      \put(18.3,58){\tiny $b$}
      \put(-1,58){\tiny $c$}
      \put(-1,90){\tiny $d$}
      \put(48.5,-3.5){$(b)$}
      \put(21,0.5){$\frac{5}{18}$}
      \put(49,0.5){$\frac{2}{7}$}
      \put(77,0.5){$\frac{1}{4}$}
      \put(32,44){\tiny $a$}
      \put(32,39.8){\tiny $e$}
      \put(32,3){\tiny $b$}
      \put(12.5,3){\tiny $c$}
      \put(12.2,44){\tiny $d$}
    \end{overpic}
    \caption{Examples of reducing rational tangle strands}
    \label{fig:tangle_example}
  \end{figure}

  We can check that the new double points created by this sink tube push for $\tilde{\beta}$ are all safe. There are two new double points: $X_q$ on the new $\tilde{\beta}$-loop $\tilde{\beta}_q$ containing $Q$, and $X_p$ on the $\tilde{\beta}$-edge of the $\beta$-sink triangle, which is connected to $P$ by a short arc in the triangle, see Figure~\ref{fig:old_tangle_push}.$(b)(c)$. Every sink corner of possible double points on $\tilde{\beta}_q$ has the basepoint inside on its boundary (there might be double points other than $X_q$ if the basepoint is $b$, where $\alpha$ intersects $\tilde{\beta}_q$). Since the basepoint is not $a$ (which is an endpoint of $\tilde{\beta}$), it is obtained by collapsing a circle. Back on $\Sigma^+$ this circle before collapsing is either a boundary circle of the branched surface, or a branch locus circle whose branch direction points to the other side. Sink corners of double points on $\tilde{\beta}_q$ all have parts of this circle on their boundaries, thus cannot be sink disks. The sink corner of $X_p$ has a boundary arc $PX_p$, and thus has the basepoint $P$ on its boundary. $P$ is not the basepoint $a$ (recall that $S$ is $a$), hence is obtained by collapsing a $\beta$-loop. Again, back on $\Sigma^+$ the sink corner of $X_p$ has a boundary arc belonging to this collapsed loop, where the branch direction points to the other side. It follows that $X_p$ is also safe.

  Since double points on $\tilde{\beta}_q$ are all safe, we can further collapse (or quotient) $\tilde{\beta}_q$ and the disk it bounds in its branch direction on $\tilde{S}$ to the basepoint inside, without destroying possible sink disks. See Figure~\ref{fig:old_tangle_push}.$(a)$. We describe a \textbf{complete} ``sink tube push'' procedure for a rational tangle strand to include both the arc-pushing and the collapsing of the loop afterwards.

  Now if we temporarily ignore the short $\tilde{\beta}$-segment $PX_p$ (we call it a \textbf{tail}), we get a new $\tilde{\beta}$-arc $\tilde{\beta}'$ connecting $Q$ and $S$ (see Figure~\ref{fig:old_tangle_push}.$(a)$ upper-right, notice the loop at $Q$ is collapsed). This $\tilde{\beta}'$ is represented by some rational number $\frac{s}{t}$, where $s,t\geq 0$, $(s,t)=1$ and $s+t<p+q$ (the number of intersections is decreasing at least by 1 by collapsing the $\tilde{\beta}$-loop at $Q$ to the basepoint). Moreover, if $s,t\geq 2$, we still can define the sink tube for $(S,\tilde{\beta}',\epsilon)$. Since we ignore the tail $PX_p$, $P$ is no longer a vertex, and thus the triangle of $(S,\tilde{\beta},\epsilon)$ at $P$ turns to a bigon of $(S,\tilde{\beta}',\epsilon)$; in fact, the \textbf{$\tilde{\beta}$-sink} triangle of $(S,\tilde{\beta},\epsilon)$ containing $P$ becomes the \textbf{sink} bigon of $(S,\tilde{\beta}',\epsilon)$, as shown in Figure~\ref{fig:old_tangle_push}.$(c)$. On the other hand, the source bigon of $(S,\tilde{\beta}',\epsilon)$ is the same as the source bigon of $(S,\tilde{\beta},\epsilon)$ since it is not moved. Depending on whether $\mathrm{d}(P,Q)>\mathrm{d}(Q,R)$ or $\mathrm{d}(P,Q)<\mathrm{d}(Q,R)$ (they cannot be equal since $\mathrm{d}(P,Q)$ contains a half length at $Q$ while $\mathrm{d}(Q,R)$ contains two halves, one at each end, and is thus an integer), we can set new $P',Q',R',S'$ as in the bottom two pictures of Figure~\ref{fig:old_tangle_push}.$(a)$ (we remark that one can identify $Q'$, the midpoint of the boundary $\beta$-arc of the new sink bigon, with $X_p$). We can then do a ``sink tube push'' for $\tilde{\beta}'$, pushing our new $P'Q'$ onto $Q'S'$, in the same sense as before. Notice that since the tail $PX_p$ is inside the new sink bigon, it is not moved or moved over by our new push, see the third picture of Figure~\ref{fig:old_tangle_push}.$(c)$ (hence we are good to first ignore the tail and define a sink tube push with $\beta_c'$). In fact, this tail is collapsed to the basepoint when we collapse the new loop at the sink bigon after doing ``sink tube push'' for $\tilde{\beta}'$. See also the second arrow of Figure~\ref{fig:tangle_example}.$(a)$.

  It follows that we can perform such ``sink tube push'' repeatedly, whenever the characterizing rational number $\frac{p'}{q'}$ has $p',q'\geq 2$, and the operation does not move $e$. Since $p'+q'$ is strictly decreasing along this procedure, it stops after finitely many steps. It stops either with $p'=1$ (recall $q'\geq 3$ since $a$ and $e$ are never moved), or at a place where the ``sink tube push'' would move $e$. In the former case we already obtain a branched surface with all double points safe by the ``$p=1$'' argument at the end of Part 2. In the latter case we need to introduce more operations.

  \vspace{6pt}

  \textbf{Part 5: perform the alternative sink tube push.}

  Now suppose $p',q'\geq 2$ and the old sink tube push would move $e$. We first analyze when this will happen. Recall that among the boundary $\epsilon$-arcs of the non-quadrilateral regions, there are two $\tilde{\beta}$-sink arcs, two $\tilde{\beta}$-source arcs, and two $\tilde{\beta}$-parallel arcs. It follows that besides the $\tilde{\beta}$-sink tube, there is also a $\tilde{\beta}$-source tube and a $\tilde{\beta}$-parallel tube, containing respectively all the $\tilde{\beta}$-source sectors and all the $\tilde{\beta}$-parallel sectors. By our symmetry, the sink tube and the source tube are of the same length (we define the length of a tube to be the length of its boundary $\tilde{\beta}$-arcs, or equivalently, the number of quadrilateral sectors it contains).

  \begin{figure}[!hbt]
    \begin{overpic}{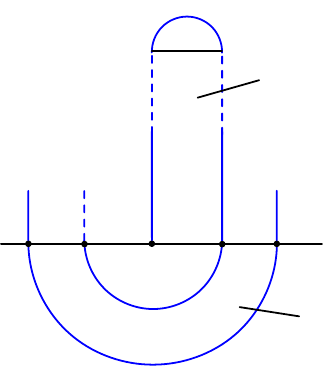}
      \put(39,32){$a$}
      \put(60,37.5){$g$}
      \put(73.8,37.5){$h$}
      \put(23,38.1){$f$}
      \put(8.7,37.5){$k$}
      \put(70,80){source tube}
      \put(81,16){sink sector}
    \end{overpic}
    \caption{A local picture near the $\tilde{\beta}$-source triangle}
    \label{fig:source_triangle}
  \end{figure}

  Now since the sink tube push would move $e$, we know $e$ is on the boundary of the sink tube. Hence the boundary $\tilde{\beta}$-arc of the source triangle is on the boundary of the sink tube. Let $P,Q,R,S$ be as before marking the $\tilde{\beta}$-rational tangle strand. Let $a,f,g$ be the vertices of the source triangle, such that $ag$ is the $\tilde{\beta}$-source arc. Recall $a=S$ and $e\in \{f,g\}$. Let $f,g,h,k$ be the vertices of the sink sector that shares a boundary $\tilde{\beta}$-arc with the source triangle, such that they appear on the boundary of the sector in clockwise order. See Figure~\ref{fig:source_triangle}.

  Let $\tilde{\beta}_1,\tilde{\beta}_2$ be the two boundary $\tilde{\beta}$-arcs of the sink tube, such that $\tilde{\beta}_1$ has an endpoint $P$. Then the old sink tube push is to push $\tilde{\beta}_1$ onto $\tilde{\beta}_2$ (recall Figure~\ref{fig:old_tangle_push}). Now since this old sink tube push would move $e$, the arc $fg$ is on $\tilde{\beta}_1$ and $hk$ on $\tilde{\beta}_2$. We claim that $f$ is on $gP$ on $\tilde{\beta}_1$. Suppose not, then $g$ is on $fP$. However, we know $ag$ is one end of the source tube. That being said, $gP$ must contain an entire boundary $\tilde{\beta}$-arc of the source tube. But $gP$ is a proper subarc of $\tilde{\beta}_1$, contradicting the fact that the sink and source tubes are of the same length. Hence $f$ is on $gP$. We remark that we allow $f=P$.

  Now notice $af$ is a boundary $\epsilon$-arc of the parallel tube. Also notice that the other end of the parallel tube is at the sink triangle, i.e. at $P$. Hence the $\tilde{\beta}$-arc $fP$ is a boundary $\tilde{\beta}$-arc of the parallel tube. In particular the parallel tube is shorter than the sink/source tubes. Let $P,f,a,l$ be the vertices of the parallel tube, appearing in counter-clockwise order on the boundary of the tube. Then $al$ is a proper subarc of the boundary $\beta$-arc of the source tube containing $a$. See Figure~\ref{fig:alternative_push} left.

  \begin{figure}[!hbt]
    \begin{overpic}[scale=0.8]{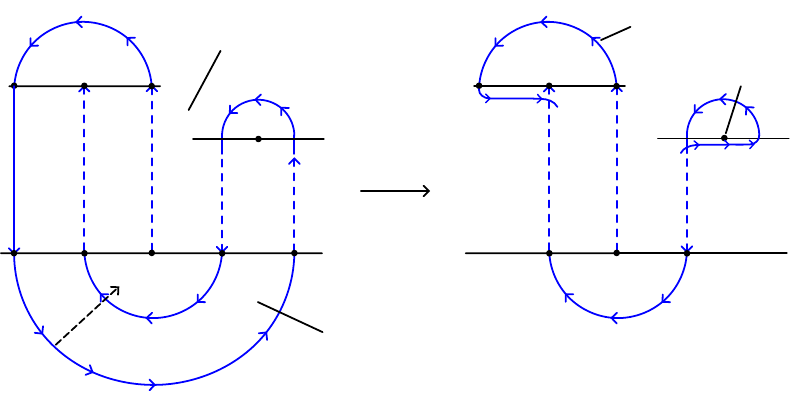}
      \put(19.7,19.4){$a(S)$}
      \put(28.6,19.4){$g$}
      \put(37.8,19.4){$h$}
      \put(11,20){$f$}
      \put(2,20){$k$}
      \put(11,40.6){$P$}
      \put(19.7,40.6){$l$}
      \put(70,20){$f$}
      \put(78.7,19.4){$a$}
      \put(87.6,19.4){$g$}
      \put(78.7,40.6){$l$}
      \put(70,40.6){$P$}
      \put(27,45){source tube}
      \put(41.8,8){sink tube}
      \put(91,41){basepoint $T$}
      \put(80.1,47.4){long tail}
    \end{overpic}
    \caption{The parallel tube and the alternative push}
    \label{fig:alternative_push}
  \end{figure}

  We now consider how to push arcs to reduce $\tilde{\beta}$ here. Since $e$ is on $\tilde{\beta}_1$, here we need to push $\tilde{\beta}_2$ onto $\tilde{\beta}_1$ instead. This time we stop before the endpoints of $\tilde{\beta}_2$ and do not move the endpoints of $\tilde{\beta}_2$, again see Figure~\ref{fig:alternative_push}. We call this the \textbf{alternative sink tube push}. We notice that the boundary $\tilde{\beta}$-arcs of the parallel tube are not moved. In fact, $fP$ is on $\tilde{\beta}_1$ and $al$ is not on the boundary of the sink tube. See also the the last arrow of Figure~\ref{fig:tangle_example}.$(a)$, or the first arrow of Figure~\ref{fig:tangle_example}.$(b)$.

  We still get a $\tilde{\beta}$-loop at the sink bigon, bounding a disk with a basepoint, see Figure~\ref{fig:alternative_push} right. Call this basepoint $T$. Similar to before, sink corners of double points on the boundary of this loop all have $T$ on their boundaries. Since $T$ is not $a=S$, back on $\Sigma^+$ it is either a boundary circle, or a cusp circle along which the branch direction points to the other side (different from the sink corners). Hence these sink corners cannot be sink disks, and these double points are safe. We can then collapse this new $\tilde{\beta}$-loop and the disk it bounds in its branch direction to $T$, as before. We remark that possible remaining tail from the old sink tube push operations is collapsed to $T$ here.

  Now after the alternative push and the collapsing of the $\tilde{\beta}$-loop, we obtain a $\tilde{\beta}$-arc $TP$ connecting the basepoints, and a \textit{long} tail containing $al$, again see Figure~\ref{fig:alternative_push} right. See also Figure~\ref{fig:tangle_example}, where the long tails are marked with light blue arcs. In particular the sink corner of the new double point at the end of the tail has $P$ on its boundary, and (similar to before) cannot be a sink disk.

  The key observation here is that, if we neglect the long tail and consider the rational tangle strand $TP$, then the source triangle $afg$ for the old strand $PS$ becomes the source bigon $fg$ for the new strand $TP$, again see Figure~\ref{fig:alternative_push} right. Hence if we could make sense of the old sink tube push for our ``rational tangle strand $TP$ with long tail'', then we get a systematic way to reduce the configuration without moving $e$. We henceforth show that we can indeed do this.
  
  The long tail can be seen as \textit{parallel} to the strand $TP$, or more precisely the subarc $fP$, in the sense that $alPf$ is the parallel tube of the old rational tangle strand. Since the parallel tube cannot contain any sink disk, we can \textit{formally} collapse the tail to $fP$ along the parallel tube (of course we cannot actually move $a$), and regard $TP$ as a \textit{weighted} rational tangle strand, where on $fP$ it has weight 2, and on $fT$ it has weight 1.

  Suppose $TP$ as a rational tangle strand is characterized by $\frac{p''}{q''}$, $p'',q''\geq 1$, $(p'',q'')=1$. If $p'',q''\geq 2$, we consider the old sink tube push defined in Part 3 for the \textit{weighted} rational tangle strand $TP$. Let $P',Q',R',S'$ be 4 marks of $TP$ as in Part 3, where $P',S'$ are the endpoints, $Q'$ is the midpoint of the boundary $\tilde{\beta}$-arc of the sink bigon, and $R'$ the midpoint of the boundary $\tilde{\beta}$-arc of the source bigon (which is $fg$). Then up to a $\frac{1}{2}$ length of $R'f$, either $P'R'$ is the weight 2 part, or $R'S'$ is the weight 2 part.
  
  \begin{figure}[!hbt]
    \begin{overpic}[scale=0.45]{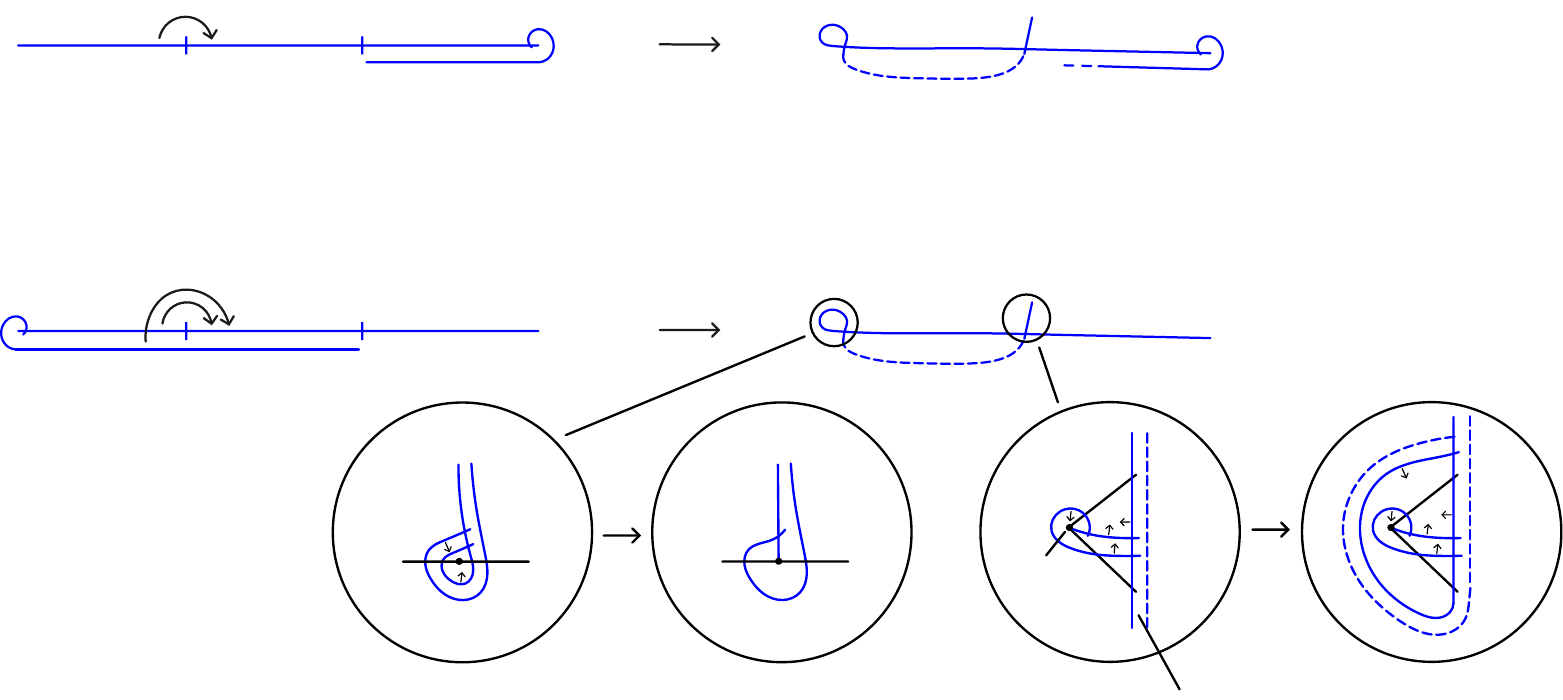}
      \put(0,38){\small $P'$}
      \put(11,38){\small $Q'$}
      \put(22,38){\small $R'$}
      \put(33,38){\small $S'$}
      \put(50,43){\small $Q'$}
      \put(66,43.8){\small $P'$}
      \put(78,42.4){\small $S'$}
      \put(0,19.5){\small $P'$}
      \put(11,19.5){\small $Q'$}
      \put(22,19.5){\small $R'$}
      \put(33,19.5){\small $S'$}
      \put(50,25){\small $Q'$}
      \put(66,26){\small $P'$}
      \put(77.7,23){\small $S'$}
      \put(65,7){\Small $P'$}
      \put(70,-2){\Small could be of weight 1 or 2}
      \put(32,32){$(a)$ $R'S'$ is of weight 2}
      \put(32,-2){$(b)$ $P'R'$ is of weight 2}
    \end{overpic}
    \caption{Sink tube push for the weighted rational tangle strand}
    \label{fig:weighted_strand}
  \end{figure}

  Suppose first that $R'S'$ is the weight 2 part. Notice from Figure~\ref{fig:alternative_push} that the long tail is in the source direction of $fP$, hence it will not be pushed over. Moreover, since it is parallel to $R'S'$, it also will not be pushed. See Figure~\ref{fig:weighted_strand}.$(a)$. It follows that in this case we can just do the old sink tube push for the $TP$ strand and keep the long tail unchanged. After collapsing the $\tilde{\beta}$-loop generated by the push at the sink bigon (or $Q'$), we get the same configuration as before, plus a short tail. Similar to before, the short tail is in the sink bigon of the new strand $Q'S'$, and will not obstruct further sink tube push.

  Now suppose $P'R'$ is the weight 2 part. In this case the long tail still will not be pushed over; however, part of it will be pushed together with $P'Q'$, as shown in Figure~\ref{fig:weighted_strand}.$(b)$ (see also the second arrow of Figure~\ref{fig:tangle_example}.$(b)$). Now near $Q'$, the figure looks like a ``bubbled'' loop, see the first local picture in Figure~\ref{fig:weighted_strand}.$(b)$. Similar to before, we can argue that double points on the inner loop are all safe, and can thus collapse the inner loop to the basepoint inside. After the collapsing the local configuration looks like the original weight-2 end of the strand, see the second local picture in Figure~\ref{fig:weighted_strand}.$(b)$. Moreover, near $P'$, we obtain a ``bubbled tail'' that is almost inside the sink triangle of $TP$, see the third local picture in Figure~\ref{fig:weighted_strand}.$(b)$, where the black dot denotes the basepoint $P'$. If we neglect this bubbled tail, then we obtain the same ``rational tangle strand and long parallel tail'' configuration, where the rational tangle strand is the subarc $Q'S'$, and the long tail is parallel to $Q'R'$. Moreover, this bubbled tail almost lies in the sink bigon of the new strand $Q'S'$, hence will be closed into the (possibly bubbled) loop of the next sink tube push (if there is one) for the new strand $Q'S'$ with tail, see the fourth local picture in Figure~\ref{fig:weighted_strand}.$(b)$. If we still can collapse the loop, then this bubbled tail will not obstruct further sink tube push.
  
  It remains for us to check that new double points are all safe, and in particular that we can collapse the loop with bubbled tail inside, as in the fourth local picture in Figure~\ref{fig:weighted_strand}.$(b)$. The new double points all appear in the four local pictures in Figure~\ref{fig:weighted_strand}.$(b)$, where the black arrows indicate the branch directions. One can directly check that sectors without basepoint on the boundary are not sink disks. For those sectors with basepoint on the boundary, we recall the basepoint is from a collapsed loop, and as before these sectors cannot be sink disks.

  It follows that we could repeatedly perform the old sink tube push for the weighted rational tangle strand $TP$, whenever $p'',q''\geq 2$, without creating new double points that are not safe. Similar to before $(p''+q'')$ strictly decreases along the procedure, and hence it must end in finitely many steps. It stops when $p''=1$. Now as before the weighted rational tangle strand intersects (the interior of) $\tilde{\alpha}$ in the same direction. Since the tail is parallel to part of the rational tangle strand, we know the actual $\tilde{\beta}$-arcs - the rational tangle strand and the long tail - also intersect (the interior of) $\tilde{\alpha}$ in the same direction. Hence we can still use the argument at the end of Part 2 and argue that all remaining double points are safe. 

  \vspace{6pt}

  \textbf{Part 6: finish the proof.}
  
  Parts $1\sim 5$ showed that we can always split $\mathcal{A''}$ by pushing arcs to get a new branched surface $\mathcal{A'''}$ where all double points are safe. Since splittings by pushing arcs are always modelled on the middle picture of Figure~\ref{fig:splittings}, they do not change the topology of $\partial_v N(\mathcal{A''})$. Hence the branch locus of $\mathcal{A'''}$ still consists of 4 immersed curves. These immersed curves intersect at vertices of $S_0$ and $M,N$ (which are never moved). Hence we can apply Lemma~\ref{lem:sk_corner} here and conclude that $\mathcal{A'''}$ is sink disk free. 
  
  We now check we can apply Lemma~\ref{lem:sk_disk_free} to $\mathcal{A'''}$. Again, these splittings do not change the topology of $\partial_h N(\mathcal{A''})$, hence $\partial_h N(\mathcal{A'''})$ is homeomorphic to $\partial_h N(\mathcal{A''})$. Recall from our constructions that $\mathcal{A}=\mathcal{B''}-\mathcal{C'}$, and that we only attached $\mathcal{A}$ to one side of $\partial_h^{\pm}$. Hence if we attach the other side of $\partial_h^{\pm}$ of $\mathcal{A''}$ back to $N(\mathcal{C'})$, we get an embedding $N(\mathcal{A''})\hookrightarrow N(\mathcal{B''})$, where vertical boundary annuli of $N(\mathcal{A''})$ are embedded into the vertical boundary annuli of $N(\mathcal{B''})$, except for the one coming from the $\mathcal{C'}$-cusp. Hence if there is a disk component of $\partial_h N(\mathcal{A''})$, its boundary must correspond to the $\mathcal{C'}$-cusp, for otherwise we would get a meridional disk in $N(\mathcal{B''})$, which is homeomorphic to the (1,1) knot complement. However, the two horizontal boundary components of $N(\mathcal{A''})$ that have $\mathcal{C'}$-cusp on the boundary are two copies of $\partial_h^{\pm}$ respectively, corresponding to the side that $\mathcal{A}$ is not attached to; in particular they are pairs of pants. It follows that no component of $\partial_h N(\mathcal{A''})$ is a disk. Hence no component of $\partial_h N(\mathcal{A'''})$ is a disk. Moreover, $\mathcal{A'''}$ also cannot carry any sphere or torus, for it will then be carried by $\mathcal{B''}$, and thus become a non-separating surface in the ambient rational homology sphere (that the (1,1) knot lives in). Now we can apply Lemma~\ref{lem:sk_disk_free} to $\mathcal{A'''}$ and conclude that it fully carries a lamination. Since $\mathcal{A'''}$ is obtained by splitting $\mathcal{A''}$, we know $\mathcal{A''}$ also fully carries a lamination.
\end{proof}

\bibliography{references.bib}
\bibliographystyle{alpha}

\end{document}